\documentclass{amsart}
\usepackage{dsfont}
\usepackage[usenames,dvipsnames,svgnames,table]{xcolor}
\usepackage[utf8]{inputenc}
\usepackage[T1]{fontenc}
\usepackage{ tipa }
\usepackage{helvet}
\usepackage{color}
\usepackage{mathrsfs}  
\usepackage{stmaryrd}  
\usepackage{amsmath,amsxtra,amsthm,amssymb,xr,enumerate,fullpage,comment, graphicx, mathtools}
\usepackage[all]{xy}
\usepackage[bbgreekl]{mathbbol}
\usepackage{bbm}

\DeclareMathSymbol{\invques}{\mathord}{operators}{`>}
\DeclareUnicodeCharacter{00BF}{\tmquestiondown}
\DeclareRobustCommand{\tmquestiondown}{%
  \ifmmode\invques\else\textquestiondown\fi
}

\usepackage{verbatim}
\numberwithin{equation}{section}

\makeatletter
\newcommand{\mylabel}[2]{#2\def\@currentlabel{#2}\label{#1}}
\makeatother

\newtheorem{theorem}{Theorem}[section]
\newtheorem{lemma}[theorem]{Lemma}
\newtheorem{conj}[theorem]{Conjecture}
\newtheorem{proposition}[theorem]{Proposition}
\newtheorem{corollary}[theorem]{Corollary}
\newtheorem{defn}[theorem]{Definition}

\newtheorem{remark}[theorem]{Remark}

\newcommand{\Gal}{\operatorname{Gal}}

\newcommand{\NN}{\mathbb{N}}
\newcommand{\QQ}{\mathbb{Q}}
\newcommand{\Qp}{{\mathbb{Q}_p}}
\newcommand{\Zp}{\mathbb{Z}_p}
\newcommand{\ZZ}{\mathbb{Z}}

\newcommand{\FF}{\mathbb{F}}

\newcommand{\g}{\mathbf{g}}

\newcommand{\ord}{\mathrm{ord}}

\newcommand{\fp}{\mathfrak{p}}

\newcommand{\cL}{\mathcal{L}}

\newcommand{\cO}{\mathcal{O}}

\newcommand{\GL}{\mathrm{GL}}

\newcommand{\cyc}{\textup{cyc}}

\newcommand{\Hom}{\mathrm{Hom}}

\newcommand{\LL}{\Lambda}
\newcommand{\TT}{\mathbb{T}}

\newcommand{\f}{\textup{\bf f}}

\newcommand{\h}{\textup{\bf h}}

\newcommand{\lra}{\longrightarrow}

\newcommand{\res}{\textup{res}}

\newcommand{\etale}{\textup{\'et}}

\newcommand{\p}{\mathfrak{p}}

\newcommand{\m}{\mathfrak{m}}

\newcommand{\cW}{\mathcal{W}}

\newcommand{\cR}{\mathcal{R}}

\newcommand{\rec}{\mathrm{rec}}

\newcommand{\cA}{\mathcal{A}}

\newcommand{\Spf}{{\rm Spf}}
\newcommand{\cl}{{\rm cl}}

\newcommand{\wt}{{\rm wt}}
\newcommand{\Fr}{{\rm Fr}}

\newcommand{\Ad}{{\rm ad}}

\newcommand{\rmw}{{\rm w}}

\newcommand{\CH}{{\rm CH}}

\newcommand{\hatotimes}{{\,\widehat\otimes\,}}

\newcommand{\hf}{\f}
\newcommand{\hg}{\g}
\newcommand{\hh}{\h}
\newcommand{\Q}{\QQ}
\newcommand{\Z}{\ZZ}

\newcommand{\C}{{\mathbb C}}

\usepackage{hyperref}

\definecolor{pinegreen}{rgb}{0.0, 0.47, 0.44}

 \definecolor{pAlgae}{RGB}{87,115,135}
\definecolor{airforceblue}{rgb}{0.36, 0.54, 0.66}
	\definecolor{bondiblue}{rgb}{0.0, 0.58, 0.71}
\definecolor{britishracinggreen}{rgb}{0.0, 0.26, 0.15}
\definecolor{camouflagegreen}{rgb}{0.47, 0.53, 0.42}
\definecolor{darkcyan}{rgb}{0.0, 0.55, 0.55}

\hypersetup{
    colorlinks = true,
    linkcolor=blue,
     filecolor=blue,
     citecolor = darkcyan,      
     urlcolor=cyan,
    linkbordercolor = {white},
}


\begin{document}

\title{O\lowercase{n the} A\lowercase{rtin formalism for triple product $p$-adic}  $L$-\lowercase{functions}: \\
C\lowercase{how}--H\lowercase{eegner points vs.}  H\lowercase{eegner points}}

\author{K\^az\i m B\"uy\"ukboduk}
\address{K\^az\i m B\"uy\"ukboduk\newline UCD School of Mathematics and Statistics\\ University College Dublin\\ Ireland}
\email{kazim.buyukboduk@ucd.ie}

\author{Daniele Casazza} 
\address{Daniele Casazza\newline UCD School of Mathematics and Statistics\\ University College Dublin\\ Ireland}
\email{daniele.casazza@ucd.ie}

\author{Aprameyo Pal}
\address{Aprameyo Pal \newline Harish-Chandra Research Institute, A CI of Homi Bhabha National Institute, Chhatnag Road, Jhunsi, Prayagraj - 211019
\\ India}
\email{aprameyopal@hri.res.in}

\author{Carlos de Vera-Piquero}
\address{Carlos de Vera-Piquero \newline Universidad de Zaragoza\\
Facultad de Ciencias\\
C. Pedro Cerbuna 12, 50009 Zaragoza \\ Spain}
\email{devera@unizar.es}

\maketitle

\begin{abstract}
    Our main objective in this paper (which is expository for the most part) is to study the necessary steps to prove a factorization formula for a certain triple product $p$-adic $L$-function guided by the Artin formalism. The key ingredients are 
    \begin{itemize}
     \item[a)] the explicit reciprocity laws governing the relationship of diagonal cycles and generalized Heegner cycles to $p$-adic $L$-functions;
    \item[b)] a careful comparison of  Chow--Heegner points and twisted Heegner points in  Hida families, via formulae of Gross--Zagier type.
\end{itemize}
\end{abstract}

\section{Introduction}

The purpose of the present article (which is largely a survey) is to partially execute the strategy outlined in \cite[\S2.5]{BCS} to study the $p$-adic Artin formalism for a certain triple product $p$-adic  $L$-function (where $p$ is an odd prime that we fix forever), and in that sense, it should be thought of as a continuation of this work. We will therefore use the conventions and notation of op. cit. whenever we can, indicating where these objects were defined in \cite{BCS}. 

To be able to spell out our goals in more precise wording, let $\f, \g$ be a pair of Hida families (cf. \S2.1 in op. cit. for the hypotheses on these families) over the respective weight spaces $\cW_\f$ and $\cW_\g$, and let $\g^c$ denote the Hida family that is conjugate to $\g$ (cf. \S2.1.7 in op. cit.). We put $\cW_3:=\cW_\f\times \cW_\g\times \cW_{\g}$ (which we think of as the weight space for the 3-parameter family $\f\otimes\g\otimes \g^c$ of triple products of cusp forms) and $\cW_2:=\cW_\f\times \cW_\g$. Note that we have a natural injection $\cW_2 \stackrel{\iota_{2,3}}{\hookrightarrow} \cW_3$  (cf. \S2.1.8 in op. cit.) given by $(\kappa,\lambda)\mapsto (\kappa,\lambda,\lambda)$. 

We will explain in this article that a stronger form of the results of \cite{YZZ10,YZZ12,YZZ23} (the proof of the Gross--Kudla conjecture for the central derivatives of triple product $L$-functions) would imply that the $\g$-dominant triple product $p$-adic  $L$-function $\cL_p^\hg(\hf\otimes\hg\otimes\hg^c)^2_{\vert_{\cW_2}}$ (cf. Theorem 3.4 in \cite{BCS} for its definition) factors in accord with the Artin formalism. Note that the  $p$-adic  $L$-function $\cL_p^\hg(\hf\otimes\hg\otimes\hg^c)^2_{\vert_{\cW_2}}$ has empty range of interpolation. Therefore, the factorization predicted by Conjecture~\ref{conj_main_6_plus_2} below, formulated in \cite{BCS} as Conjecture 2.2, does not follow directly from the Artin formalism for complex  $L$-series via the interpolation properties of $p$-adic $L$-functions.

\begin{conj}
\label{conj_main_6_plus_2}
Suppose that $\varepsilon(\hf)=-1=\varepsilon^{\rm bal}(\hf\otimes\g\otimes\g^c)$. We then have the following factorization of $p$-adic $L$-functions:
\begin{equation}
\label{eqn_conj_main_6_plus_2}
     \cL_p^\hg(\hf\otimes\hg\otimes\hg^c)^2_{\vert_{\cW_2}}= \mathscr C\cdot \cL_p^\Ad(\hf\otimes \Ad^0\hg) \cdot {\rm Log}_{\omega_\f}({\rm BK}_{\f}^\dagger)\,,
\end{equation}
        where $\mathscr C\in {\rm Frac}(\cR)$ is a meromorphic function in 2 variables with an explicit algebraicity property at crystalline specializations $(\kappa,\lambda)$ (cf. Theorem~8.11 in op. cit.).
\end{conj}
In the statement of Conjecture~\ref{conj_main_6_plus_2}:
\begin{itemize}
    \item ${\rm Log}_{\omega_\f}({\rm BK}_\f^\dagger)$ denotes the logarithm of the big Beilinson--Kato class (constructed by Ochiai); cf. \cite[\S6.1.2]{BCS}. We refer the reader to \S7.2.9 of op. cit. for a justification (in view of the Artin formalism) of the presence of this factor in \eqref{eqn_conj_main_6_plus_2} as an avatar of $L$-values.
    \item $\varepsilon(\hf)$ is the common global root number of the family $\f$, whereas $\varepsilon^{\rm bal}(\hf\otimes\g\otimes\g^c)$ is the same for the family $\hf\otimes\g\otimes\g^c$ at those weights $(\kappa,\lambda,\mu)\in \cW_3$ that are \emph{balanced} (cf. \S2.1.5 and \S2.2.3 of op. cit.).
    \item We implicitly assume throughout this article the existence of the $p$-adic $L$-function $\cL_p^\Ad(\hf\otimes \Ad^0\hg)$, which is characterized by the interpolation properties described in \cite[Conjecture 3.6]{BCS} (which is concurrent with the general Coates--Perrin-Riou formalism).
\end{itemize}

Conjecture~\ref{conj_main_6_plus_2} was proved in \cite[\S8]{BCS} under the additional hypothesis that the family $\g$ has CM. We refer the reader to the extensive discussion in \S2.2.6 and \S2.2.7 in op. cit. for the motivation behind this conjecture. 

\begin{remark}
     We would like to underscore the comparison of Conjecture~\ref{conj_main_6_plus_2} with and its divergence from earlier work with similar flavour; e.g. that of \cite{Gross1980Factorization, Dasgupta2016}. 
     
     As our strategy to attempt Conjecture~\ref{conj_main_6_plus_2} will make it clear, the present factorization problem at hand amounts to a comparison of algebraic cycles in families, that explain the vanishing of central critical values of $L$-functions due to sign reasons. In contrast, in  \cite{Gross1980Factorization, Dasgupta2016}, the factorization problem is governed by a comparison of elements in the motivic cohomology (e.g. elliptic units vs cyclotomic units in the former, Beilinson--Flach elements vs cyclotomic units in the latter) that explain the vanishing of $L$-values at non-critical points due to $\Gamma$-factors. Hand-in-hand with this, where we need a comparison of height pairings, Gross and Dasgupta rely on a comparison of regulators.
     
     We refer to \cite[\S2.4]{BCS} for a detailed discussion on this topic, including the differences (in comparison to \cite{palvannan_factorization}) in the factorization of algebraic $p$-adic $L$-functions, and from the perspective of Perrin-Riou's theory of $p$-adic $L$-functions.
\end{remark}

Our goal in this paper is to explain that this conjecture can be proved as soon as the following two ingredients become available:
\begin{itemize}
    \item[(\mylabel{item_deg6}{\bf $L_p^{\rm ad}$})]  The construction of the $p$-adic $L$-function $\cL_p^\Ad(\hf\otimes \Ad^0\hg)$ with the expected interpolative properties (that are recorded as \cite[Conjecture 3.6]{BCS}; see also Lemma~\ref{lemma_first_reduction} below). 
    \item[(\mylabel{item_GKS}{\bf ${\rm GK}^+$})] An extension of a conjecture of Gross and Kudla\footnote{We remark that this conjecture has been settled by Yuan--Zhang--Zhang in \cite{YZZ10, YZZ12, YZZ23} in many cases, which unfortunately does not cover the level of generality required in the present work.} expressing the central critical derivatives of certain triple product $L$-functions in terms of the Beilinson--Bloch heights of Gross--Kudla--Schoen cycles (recorded as Conjecture~\ref{conj:GK-delta} below). 
\end{itemize}
\subsection{Set-up}
\label{subsec_the_set_up_intro}
As we have noted above, we shall closely follow the notation and conventions of \cite[\S2.1]{BCS}. We include in this subsection a review of some of those that play a key role in our paper.

\subsubsection{}
\label{subsubsec_2023_09_26_1144}
Fix forever a prime $p{>2}$. Let us fix an algebraic closure $\overline{\QQ}$ of $\QQ$ and fix embeddings $\iota_\infty: \overline{\QQ} \hookrightarrow \mathbb{C}$ and $\iota_p: \overline{\QQ} \hookrightarrow \mathbb{C}_p$ as well as an isomorphism $j:\mathbb{C}\stackrel{\sim}{\longrightarrow} \mathbb{C}_p$ in a way that the diagram
$$\xymatrix@R=.1cm{&\mathbb{C}\ar[dd]^{j}\\
\overline{\QQ}\ar[ur]^{\iota_\infty}\ar[rd]_{\iota_p} &\\
&\mathbb{C}_p
}$$
commutes. 

\subsubsection{} For a field $F$, let us fix a separable closure $\overline{F}$ of $F$ and denote by $G_F:=\Gal(\overline{F}/F)$ its absolute Galois group. If $F\subset F' \subset \overline{F}$ is a finite subextension, we denote by $\res_{F'/F}$ the restriction map
$$H^\bullet(F,\,\cdot\,)=H^\bullet(G_F,\,\cdot\,)\xrightarrow{\res_{F'/F}} H^\bullet(G_{F'},\,\cdot\,)=H^\bullet(F',\,\cdot\,)$$
on Galois cohomology induced from the inclusion $G_{F'}\subset G_F$. 

\subsubsection{} For an abelian group $G$, let us denote by $\LL(G):=\ZZ_p[[G]]$ its $p$-adically complete group ring.
\subsubsection{Hida families}
\label{subsubsec_2022_05_16_1506}
Let $p$ be an odd prime and let $\cO$ be the ring of integers of a finite extension $E$ of $\Qp$. Let  $\chi_\cyc: G_{\QQ}\to \ZZ_p^\times$ (resp. $\omega$) denote the $p$-adic cyclotomic (resp. Teichm\"uller) character. Let us put $\langle \chi_\cyc\rangle:=\omega^{-1}\chi_\cyc$ and note that $\langle \chi_\cyc\rangle$ takes values in $1+p\ZZ_p$.

We put $\LL_{\rm wt}:=\LL(\ZZ_p^\times)$ and denote by $\ZZ_p^\times \xrightarrow{[\,\cdot\,]} \LL_{\rm wt}^\times$ the natural injection. We let $\bbchi$ denote the universal weight character:
\[
    \bbchi: G_\QQ\stackrel{\chi_\cyc}{\lra}\ZZ_p^\times\hookrightarrow \LL_{\rm wt}^\times\,.
\]
An arithmetic specialization of weight $k\in \ZZ$ is a ring homomorphism 
\[
    \nu:\quad \LL_{\rm wt} \lra \cO
\]
such that the map $G_\QQ\stackrel{\bbchi}{\lra} \LL_{\rm wt}^\times\stackrel{\nu}{\lra}\cO$ agrees with $\chi_\cyc^{k}$ (for some natural number $k$) on an open subgroup of $G_\QQ$. In more precise terms, we have $\nu\circ \bbchi=\omega^j\psi_\nu\langle\chi_\cyc\rangle^{k}$, and $\psi_\nu$ (called the wild character of $\nu$) is a Dirichlet character of $p$-power order and $p$-power conductor.

We let $\hh=\sum_{n=1}^{\infty} \mathbb{a}_{n}(\hh)q^n \in \cR_\hh[[q]]$ denote the branch of the primitive Hida family of tame conductor $N_\hh$ and nebentype character $\varepsilon_\hh$ (which is a Dirichlet character modulo $N_\hh p$). Here, $\cR_\hh$ is the branch (i.e. the irreducible component) of Hida's universal ordinary Hecke algebra (cf. \cite{EPW2006}, \S2.7).  Let us write $\varepsilon_\hh=\varepsilon_\hh^{(\rm t)}\varepsilon_\hh^{(p)}$, where $\varepsilon_\hh^{(\rm t)}$ (resp. $\varepsilon_\hh^{(p)}$) is a Dirichlet character modulo $N_\h$ (resp. $p$).

We have $a_p(\hh) \in \mathcal \cR_\hh^\times$, and $\bbchi$ gives rise to the character
\[
    \bbchi_{\hh}:  G_\QQ\xrightarrow{\bbchi} \LL_{\rm wt}^\times \lra \cR_\hh^\times\,.
\]
 For any $\kappa\in \cW_\hh :=\Spf(\cR_\hh)(\C_p)$, let us write ${\rm wt}(\kappa)\in \Spf(\LL_{\rm wt})(\C_p)$ for the point that $\kappa$ lies over and call it the weight character of $\kappa$, so that ${\rm wt}(\kappa)\circ \bbchi =\kappa\circ \bbchi_{\hh}$. We say that $\kappa$ is classical if ${\rm wt}(\kappa)$ is an arithmetic specialization in the above sense. 

We call $\cW_\hh$ the weight space for the Hida family $\hh$. Let us denote by $\cW_\hh^{\rm cl}\subset \cW_\hh$ the set of classical specializations.  For $\kappa\in \cW_\hh^{\rm cl}$,  we let $\fp_\kappa \subset \cR_\hh$ denote the corresponding prime ideal that comes attached to $\kappa$ (cf. \cite{dJ95}, \S7.1.9--\S7.1.10). We put $F_\kappa:=(\cR_\hh)_{\p_\kappa}/\p_\kappa(\cR_\hh)_{\p_\kappa}$, which is a finite extension of $\QQ_p$, and denote its ring integers by $\cO_\kappa$. Then 
${\rm wt}(\kappa)\circ \bbchi =\kappa\circ \bbchi_{\hh}$ takes values in $F_\kappa^\times$, and according to the discussion in the preceding paragraph, it has the form 
\begin{equation}
    \label{eqn_2023_09_26_1021}
    {\rm wt}(\kappa)\circ \bbchi = \omega^{j}\psi_\kappa\chi_\cyc^k\,,\quad j,k\in \NN
\end{equation}
where $k_\circ:=j+k\pmod{p-1}$ is independent of $\kappa$, and where $\psi_\kappa$ is a Dirichlet character modulo $p^{r(\kappa)}$ of $p$-power order and $p$-power conductor $p^{s_\kappa}$. By slight abuse, we denote also by ${\rm wt}(\kappa)$ the positive integer $k$ given as in \eqref{eqn_2023_09_26_1021}. We call $\psi_\kappa$ the wild character of $\kappa$. The specialization 
$$\hh_\kappa:=\sum_{n=1}^{\infty} \kappa(\mathbb{a}_{n}(\hh))q^n \in F_\kappa[[q]]$$ 
is a $p$-stabilized cuspidal eigenform of weight ${\rm wt}(\kappa)+2$,  level $\Gamma_1(Np^{s_\kappa})$ and nebentype\footnote{In the main body of our article, we shall take $k_\circ=0$ for the Hida families $\f$ and $\g$ we consider below. Therefore, the branches of all Hida families that appear in our work are supported over the connected component of the weight space ${\rm Spf}(\LL_{\rm wt})$ that is centered at a point of weight $0$.} $\varepsilon_\hh\psi_\kappa\omega^{k_\circ-{\rm wt}(\kappa)}$.

Hida has attached a Galois representation
\[
    \rho_\hh : G_{\Q,\Sigma} \lra \GL_2(V_\hh)
\]
to $\hh$, where $\Sigma$ is a finite set of primes containing all those dividing $pN_\hh\infty$ and $V_\hh$ is a 2-dimension $\mathrm{Frac}(\cR_\hh)$-vector space. The Galois representation $\rho_\hf$ is characterized by the property that
\[
    \mathrm{Tr}\, \rho_\hh(\Fr_\ell) = a_\ell(\hh), \qquad  \ell\not\in\Sigma.
\]
We denote by $T_\hh\subset V_\h$ the Ohta lattice (cf. \cite{ohta99,ohta00}, see also \cite{KLZ2} where our $T_\hh$ corresponds to $M(\hh)^*$ in op. cit.) that realizes $\rho_\hh$ in \'etale cohomology groups of a tower of modular curves. Under the hypothesis that
    \begin{itemize}
       \item[\mylabel{item_Irr}{\bf (Irr)}]  the residual representation $\bar{\rho}_\hh$ is absolutely irreducible
    \end{itemize}
(which we assume throughout for all the Hida families that appear in our work), it follows that any $G_\QQ$-stable lattice in $ V_\hh$ is homothetic to $T_\hh$. If $\kappa\in \cW_\hh$, then $V_{\hh_\kappa}:=T_\hh\otimes_{\kappa} F_\kappa$ is Deligne's representation associated to the cuspidal eigenform $\hh_\kappa$.

Thanks to Wiles, we have
\[
    {\rho_\hh}_{|_{G_{\Qp}}} \simeq \begin{pmatrix} \delta_\hh & * \\ 0 & \alpha_\hf\end{pmatrix}, 
\]
where $\alpha_\hh: G_{\QQ_p}\to \cR_\hh^\times$ is the unramified character given by $\alpha_\hh(\Fr_p)=a_p(\hh)$ and 
$\delta_\hh:=\bbchi_{\hh}\, \chi_\cyc\, \alpha_\hh^{-1}\,\varepsilon_\hh$. Whenever
    \begin{itemize}
         \item[\mylabel{item_Dist}{\bf (Dist)}] $\delta_\hh \not\equiv \alpha_\hh \mod \mathfrak{m}_{\cR_\hh}$
    \end{itemize}
    (we assume the validity of this condition throughout this paper for all the Hida families that appear in our work), the lattice $T_\hh$ fits in an exact sequence
\begin{equation} \label{eqn:filtrationf}
    0\lra T_\hh^+ \lra T_\hh \lra T_\hh^- \lra 0 
\end{equation} 
of $\cR_\hh[[G_{\Qp}]]$-modules, where the action on $ T_\hh^+$ (resp. $ T_\hh^-$) is given by $\delta_\hh$ (resp. $\alpha_\hh$). 

\subsubsection{Self-dual triple products}
\label{subsubsec_211_2022_06_01_1635}
Let $\f$ and $\g$ be primitive Hida families of ordinary $p$-stabilized newforms of tame levels $N_\f$ and  $N_\g$ (as in \S\ref{subsubsec_2022_05_16_1506}), where the tame nebentype $\varepsilon_\f$ of the family $\f$ is the trivial character $\mathds{1}$ modulo $N_\f$, and $\varepsilon_\g=\varepsilon_\g^{\rm (t)}$ as required in \cite[\S5]{BSV}. We require that $N_\f$ is square-free, which is a strengthening of the condition (sf) in \cite{Hsieh}. 

Let us set $N:={\rm LCM}(N_\f, N_\g)$ and put $T := T_\f\,\widehat\otimes\,_{\ZZ_p} {\rm ad}(T_\g)$. Then $T$ is a Galois representation of rank $8$ over the complete local Noetherian ring
\[
    \cR := \cR_\f \,\widehat\otimes\,_{\ZZ_p} \cR_\g\,.
\]
Since $\varepsilon_\f=\mathds{1}$ and ${\rm ad}(T_\g)$ is self-dual, we have a perfect $G_\QQ$-equivariant Poincar\'e duality pairing (see \S\ref{subsubsec_114_2023_09_27_1616} and \S\ref{subsubsec_118_2024_07_02_1311}, where we employ the discussion therein with $\h=\f=\h^c$, and with $\h=\g$) 
\[
   T\otimes T\lra \bbchi_{\f}\chi_\cyc\,.
\]
Since $p>2$ by assumption, there exists a unique character $\bbchi_{\f}^{\frac{1}{2}}:G_\Q\to \cR_\f^{\times}\hookrightarrow \cR^\times$ with $(\bbchi_{\f}^{\frac{1}{2}})^2=\bbchi_{\f}$. Then the Galois representations
\[
  T_\f^\dagger:=T_\f\otimes\bbchi_{\f}^{-\frac{1}{2}}\,,\quad   T^\dagger := T\otimes \bbchi_{\f}^{-\frac{1}{2}}= T_\f^\dagger\widehat\otimes {\rm ad}(T_\g)
\]
are self-dual, in the sense that $T_\f^\dagger\simeq \Hom_{\cR_\f}(T^\dagger_\f, \cR_\f)(1)$, and $T^\dagger \simeq \Hom_{\cR}(T^\dagger, \cR)(1)$ as $G_\QQ$-representations.

\subsubsection{} The Galois representation $M:=T_\f^\dagger\widehat\otimes_{\ZZ_p}{\rm ad}^0(T_\g)$ is a free of rank-6  self-dual direct summand of $T^\dagger$. We view it as a submodule of $T$ naturally, considering ${\rm ad}^0(T_\g)$ as the kernel of the $G_\QQ$-equivariant trace map ${\rm ad}(T_\g)\xrightarrow{\rm tr} R_\g$, where the target $R_\g$ is endowed with the trivial Galois action.

\subsubsection{Root numbers} 
\label{subsubsec_root_numbers_2024_02_09_1652}
For any $\kappa\in \cW_\f$ and $\lambda\in \cW_\g$, let us denote by $F_{\kappa,\lambda}$ the field generated by $F_\kappa$ and $F_\lambda$, and let $\cO_{\kappa,\lambda}$ denote its ring of integers. We set $T_{\kappa,\lambda}:=T \otimes_{\cR,\kappa\otimes\lambda} \cO_{\kappa,\lambda}$\,, and similarly define $T_{\kappa,\lambda}^\dagger$. We require that we have for some (equivalently, every) classical point $(\kappa,\lambda)\in \cW_2$
$$\varepsilon({\rm WD}_\ell(T_{\kappa,\lambda}^\dagger))=+1$$ 
for the local root number at every $\ell \mid N_\f N_\g$. Note that this shows that the global root number $\varepsilon(T_{\kappa,\lambda}^\dagger)$ of $T_{\kappa,\lambda}^\dagger$ (which is given as the product of local root numbers, including the one at the archimedean place) for such $(\kappa,\lambda)$ equals $-1$ (resp. $+1$) if ${\rm wt}(\kappa)\leq 2{\rm wt}(\lambda)+1$ (resp. if ${\rm wt}(\kappa)> 2{\rm wt}(\lambda)+1$). We refer the reader to \cite[\S1.2]{Hsieh} for a detailed discussion on the local and global root numbers in the setting of the present paper. We remark that our hypotheses on the local root numbers are required for the construction of the unbalanced triple product $p$-adic $L$-functions, cf. Assumption (1) for \cite[Theorem A]{Hsieh}.

Throughout this paper, we also assume that the global root number $\varepsilon(T_{\f_\kappa}^\dagger)$ of $T_{\f_\kappa}^\dagger$ equals $-1$ for some (equivalently, every) $\kappa\in \cW_\f^{\rm cl}$.

\subsubsection{}
\label{subsubsec_114_2023_09_27_1616}
Let us denote by $\h^c:=\h\otimes \varepsilon_\h^{-1}$ the conjugate family. We recall that we have $\varepsilon_\h=\varepsilon_\h^{(t)}$ by assumption. As remarked in \cite[Lemma 3.4]{LoefflerCMB}, the Hida family $\h^c$ is also primitive of level $N_\h$. We identify the weight space $\cW_{\h^c}$ with $\cW_\h$. We have a perfect $G_\QQ$-equivariant Poincar\'e duality pairing
\begin{equation}
    \label{eqn_2023_09_26_1114}
    T_\h \otimes_{\cR_\h} T_{\h^c}\xrightarrow{\langle\,,\,\rangle_\h} \bbchi_\h\chi_\cyc\,,
    \end{equation}

Since the precise definition of this pairing is important for our eventual purposes (e.g. for the deduction of the commutative diagram \eqref{eqn_2023_09_27_1038} below), we briefly review its basic properties. Our discussion relies greatly on the exposition in \cite[\S7]{KLZ2}. 

We first recall the construction of $T_\h$. Let us put 
$$H^1_{\rm ord}(Y_1(N_\h p^\infty)):=e_{\rm ord}' \varprojlim_r H^1_{\textup{\'et}}(Y_1(N_\h p^r)_{\overline{\QQ}},\ZZ_p(1))\,,$$
where $e_{\rm ord}':=\lim_n (U_p')^{n!}$ is Hida's ordinary projector associated to $U_p'$. Let $\mathbb T_{N_\h p^\infty}$ denote the Hecke algebra acting on $H^1_{\rm ord}(Y_1(N_\h p^\infty))$, generated by $\{T_\ell': \ell\nmid N_\h p\}$, $\{U_\ell': \ell\mid N_\h p\}$, and the diamond operators $\langle d \rangle$ for integers $d$ coprime to $N_\h p$. The ring $\mathbb T_{N_\h p^\infty}$ is semi-local, and $\h$ determines a unique maximal ideal $\m_\h$ of $\mathbb T_{N_\h p^\infty}$. The localization $H^1_{\rm ord}(Y_1(N_\h p^\infty))_{\m_\h}$ at this maximal ideal is denoted by $M(\h)^*$ in \cite{KLZ2}. The local algebra $\mathbb T_{N_\h p^\infty,\m_\h}$ has finitely many minimal primes, and $\h$ corresponds exactly to one of these (which we denote by $\mathfrak{a}_\h$). We let $\cR_\h$ denote the normalization of the integral domain $\mathbb T_{Np^\infty,\m_\h}/\mathfrak{a}_\h$ and put 
\begin{equation}
\label{eqn_2023_11_13_1642}
    T_\h:=H^1_{\rm ord}(Y_1(N_\h p^\infty))_{\m_\h}\otimes_{\mathbb T_{N_\h p^\infty,\m_\h}}\cR_\h\,.
\end{equation}

The $\mathbb T_{N_\h p^\infty}$-module $H^1_{\rm ord}(Y_1(N_\h p^\infty))$ interpolates ordinary modular forms with tame level $N_\h$ in the following sense: $H^1_{\rm ord}(Y_1(N_\h p^\infty))$ comes equipped with the diamond action of $\LL^{\diamondsuit}_{N_\h  p^r}:=\LL(\ZZ_p^\times\times (\ZZ/N_\h\ZZ)^\times)$; we let $[z]$ denote the diamond operator corresponding to $z\in \ZZ_p^\times\times (\ZZ/N_\h\ZZ)^\times$. Then for any natural number $k$ and positive integer $r$, we have an isomorphism
$$H^1_{\rm ord}(Y_1(N_\h p^\infty))/I_{k,r}\xrightarrow{\,\,\sim\,\,} e_{\rm ord}'H^1_{\textup{\' et}}(Y_1(N_\h p^r), {\rm TSym}^{k}(\mathscr{H}_{\ZZ_p})(1))=: {\mathscr T}_{\rm ord}(N_\h p^r,k)\,,$$
where 
\begin{itemize}
\item $I_{k,r}\subset \LL^{\diamondsuit}_{N_\h  p^r}$ is the ideal generated by $[1+p^r]-(1+p^r)^k$, 
\item $\mathscr{H}_{\ZZ_p}$ is the \'etale sheaf on $Y_1(N_\h p^r)$ given as in \cite[\S2.3]{KLZ2};
\end{itemize}
 The Hecke module ${\mathscr T}_{\rm ord}(N_\h p^r,k)$ can be identified with (2 copies of) the space of $p$-ordinary modular forms of level $N_\h p^r$ and weight $k+2$ thanks to the Eichler--Shimura isomorphism, and comes equipped with the following Galois-equivariant perfect pairing:
 \begin{equation}
     \label{eqn_2023_09_29_1444}
  \langle\,,\,\rangle_{k,r}\,:\quad   
  {\mathscr T}_{\rm ord}(N_\h p^r,k)\otimes {\mathscr T}_{\rm ord}(N_\h p^r,k)\xrightarrow{x\otimes y \mapsto \langle x, W_{N_\h p^r}^{-1}(U_p')^ry \rangle} \ZZ_p[\Delta_{N_\h p^r}]\otimes\chi_\cyc^{k+1}\,,
 \end{equation}
where $\Delta_{m}=\Gal(\QQ(\mu_{m})/\QQ)\simeq (\ZZ/m\ZZ)^\times$, and $\langle\,,\,\rangle$ is the Poincar\'e duality pairing.

Let us put 
$$T_{\h}[k,r]:=T_\h\otimes_{\LL^{\diamondsuit}_{N_\h  p^r}}\LL^{\diamondsuit}_{N_\h  p^r}/I_{k,r}={\mathscr T}_{\rm ord}(N_\h p^r,k)\otimes_{\mathbb T_{N_\h p^\infty,\m_\h}}\cR_\h\,.$$
Then the pairing \eqref{eqn_2023_09_29_1444} gives rise to the perfect pairing
$$T_{\h}[k,r] \otimes T_{\h}[k,r] \xrightarrow{\langle\,,\,\rangle_{k,r}} \ZZ_p[\Delta_{p^r}]\otimes \varepsilon_\h\chi_\cyc^{k+1}\,,$$
which gives rise to the pairing
\begin{equation}
    \label{eqn_2023_09_29_1518}
    T_{\h}[k,r] \otimes T_{\h^c}[k,r]=T_{\h}[k,r] \otimes (T_{\h}[k,r] \otimes\varepsilon_\h^{-1}) \lra \ZZ_p[\Delta_{p^r}]\otimes\chi_\cyc^{k+1}\,.
\end{equation}
The pairings \eqref{eqn_2023_09_29_1518} are interpolated by \eqref{eqn_2023_09_26_1114} as $k$ and $r$ vary.

\subsubsection{}
\label{subsubsec_118_2024_07_02_1311}
We further elaborate on the modified Poincar\'e duality pairing \eqref{eqn_2023_09_29_1444}. In view of the perfectness of the usual  Poincar\'e duality pairing
$$e_{\rm ord}'\,H^1_{\textup{\' et}}(Y_1(N_\h p^r), {\rm TSym}^{k}(\mathscr{H}_{\ZZ_p})(1))\,\otimes\,e_{\rm ord}\,H^1_{\textup{\' et}}(Y_1(N_\h p^r), {\rm Sym}^{k}(\mathscr{H}_{\ZZ_p}^{\vee}))\xrightarrow{\langle\,,\,\rangle} \ZZ_p$$
(where $\mathscr{H}_{\ZZ_p}^{\vee}$ is the sheaf on $Y_1(N_\h p^r)$ that is dual to $\mathscr{H}_{\ZZ_p}$ and $e_{\rm ord}:=\lim U_p^{n!}$ is Hida's ordinary projector associated to $U_p$), the perfectness of \eqref{eqn_2023_09_29_1444} amounts to the statement that we have an isomorphism
\begin{equation}
    \label{eqn_2023_11_06_1518}
   e_{\rm ord}'\,H^1_{\textup{\' et}}(Y_1(N_\h p^r), {\rm TSym}^{k}(\mathscr{H}_{\ZZ_p})(1))\xrightarrow[W_{N_\h p^r}^{-1}(U_p')^r]{\sim} e_{\rm ord}\,H^1_{\textup{\' et}}(Y_1(N_\h p^r), {\rm Sym}^{k}(\mathscr{H}_{\ZZ_p}^{\vee}))\,,
\end{equation}
with inverse (as the Atkin--Lehner involution $W_{N_\h p^r}$ intertwines the action of $U_p'$ and $U_p$, cf. \cite{ohta99}, Equation 1.5.4)
\begin{equation}
    \label{eqn_2023_11_06_1525}
   e_{\rm ord}\,H^1_{\textup{\' et}}(Y_1(N_\h p^r), {\rm Sym}^{k}(\mathscr{H}_{\ZZ_p}^\vee))\xrightarrow[U_p^{-r}W_{N_\h p^r}]{\sim} e_{\rm ord}'\,H^1_{\textup{\' et}}(Y_1(N_\h p^r), {\rm TSym}^{k}(\mathscr{H}_{\ZZ_p}(1)))\,.
\end{equation}

We invite the readers to compare this discussion in \cite[\S2.2]{ohta99}, \cite[\S7.4]{KLZ2} and that lying between Proposition 3.2 and Remark 3.3 in \cite{BSV}.

\subsubsection{}
\label{subsubsec_117_2024_02_08_1841}
The pairing~\eqref{eqn_2023_09_26_1114} induces a natural isomorphism
\begin{equation}
    \label{eqn_2023_09_26_1043}
    {\rm ad}(T_\g)= T_\g\otimes_{\cR_\g} {\rm Hom}_{\cR_\g}(T_{\g},\cR_\g)\xrightarrow{\,\,\sim\,\,} T_\g \otimes (T_{\g^c} \otimes \bbchi_\g^{-1}\chi_\cyc^{-1})\,,
\end{equation}
as well as the following commutative diagram:
\begin{equation}
    \label{eqn_2023_09_26_1321}
    \begin{aligned}
        \xymatrix{
    {\rm ad}(T_\g) \ar[r]^-{{\rm tr}} \ar[d]_-{\eqref{eqn_2023_09_26_1043}} & \cR_g\ar@{=}[d]\\
    T_\g \otimes T_{\g^c} \otimes \bbchi_\g^{-1}\chi_\cyc^{-1}\ar[r]_-{\langle\,,\,\rangle_\g}&\cR_\g\,.
    }
    \end{aligned}
\end{equation}
For  $\lambda \in \cW^{\rm cl}_\g$ with ${\rm wt}(\lambda)=0$ and wild character $\psi_\lambda$ (with conductor $p^{s_\lambda}$), we have 
$$\g_\lambda\in S_2(\Gamma_1(N_\g p^{s_\lambda}),\varepsilon_\g\psi_\lambda)\quad,\quad \overline{\g}_\lambda:=\g_\lambda^c\otimes\psi_\lambda^{-1}\in S_2(\Gamma_1(N_\g p^{s_\lambda}),\varepsilon_\g^{-1}\psi_\lambda^{-1})\,,$$
and $\overline{\g}_\lambda$ indeed coincides (utilizing the identifications \S\ref{subsubsec_2023_09_26_1144}) with the complex conjugate of the eigenform $\g_\lambda$. The specializations of the pairing \eqref{eqn_2023_09_26_1114}, the isomorphism \eqref{eqn_2023_09_26_1043}, and the diagram \eqref{eqn_2023_09_26_1321} to $\lambda$ read
\begin{equation}
    \label{eqn_2023_09_26_116}
    V_{\g_\lambda}\otimes_{F_\lambda} V_{\overline{\g}_\lambda}\xrightarrow{\langle\,,\,\rangle_\lambda}  \chi_\cyc\,,
\end{equation}
\begin{equation}
    \label{eqn_2023_09_26_1044}
    {\rm ad}(V_{\g_\lambda})\xrightarrow{\,\,\sim\,\,} V_{\g_\lambda} \otimes (V_{\overline{\g}_\lambda} \otimes \chi_\cyc^{-1})\,,
\end{equation}
\begin{equation}
    \label{eqn_2023_09_26_1331}
    \begin{aligned}
        \xymatrix{
    {\rm ad}(V_{\g_\lambda}) \ar[r]^-{{\rm tr}} \ar[d]_-{\eqref{eqn_2023_09_26_1044}} & F_\lambda\ar@{=}[d]\\
    V_{\g_\lambda} \otimes V_{\overline{\g}_\lambda} \otimes \chi_\cyc^{-1}\ar[r]_-{\langle\,,\,\rangle_\lambda}& F_\lambda\,.
    }
    \end{aligned}
\end{equation}

\subsection*{Acknowledgements} We are grateful to the referee for valuable comments and suggestions on an earlier version of the article.

K.B. thanks Wei Zhang for enlightening discussions and extremely helpful exchanges on Conjecture~\ref{conj:GK-delta}. He also thanks Henri Darmon for his encouragement. 

K.B.’s research conducted in this publication was funded by the Irish Research Council under grant number IRCLA/2023/849 (HighCritical). C.dVP.'s research for this work was funded by the {\em Departamento de Ciencia, Universidad y Sociedad del Conocimiento} of the {\em Gobierno de Arag\'on} (E22\_23R: ``\'Algebra y Geometr\'ia'').

\tableofcontents

\section{Big Heegner points and reformulation of Conjecture~\ref{conj_main_6_plus_2}} 

\subsection{$p$-local constructions}$\,$
Let $\h$ be a Hida family as in \S\ref{subsubsec_2022_05_16_1506}. Following \cite[\S8.2]{KLZ2}, let us put ${\bf D}(T_\h^{-}):=\left(T_\h^{-}\hatotimes_{\Zp}\Zp^{\rm ur}\right)^{G_{\Qp}}$ and ${\bf D}(T_\h^{+}):=\left(T_\h^{+}\otimes(\bbchi_\h\chi_\cyc\varepsilon_\h)^{-1}\hatotimes_{\Zp}\Zp^{\rm ur}\right)^{G_{\Qp}}$.  We recall from \cite[Proposition 10.1.1]{KLZ2} that the overconvergent Eichler--Shimura theorem gives rise to a pair of canonical maps 
$$\omega_\h: {\bf D}(T_\h^{+}) \stackrel{\sim}{\lra} \cR_\f\,\,,\,\,\quad \eta_\h: {\bf D}(T_\h^{-}) \lra  \frac{1}{H_\h}\cR_\h\,,$$
where $H_\h$ is Hida's congruence ideal associated with the cuspidal family $\h$.

\subsubsection{} 
\label{subsubsec_211_2023_09_27_1236}
We define the big Perrin-Riou logarithm map
\begin{equation}
\label{eqn_2022_12_15_1132}
{\rm Log}_{T_\f^+}\, \colon \, H^1_{\rm Iw}(\Qp(\mu_{p^\infty}),T_\f^{+})\stackrel{\sim}{\lra} {\bf D}(T_\f^{+})\hatotimes_{\Zp}\LL(\Gamma_\cyc)
\end{equation}
as in \cite[Theorem 8.2.3]{KLZ2} (see also \cite{BO}, \S4), where $H^1_{\rm Iw}(\Qp(\mu_{p^\infty}),\bullet):=\varprojlim_n H^1(\Qp(\mu_{p^n}),\bullet)$ denotes Iwasawa cohomology.  Recall our notation $T_{\f}^{\dagger} := T_{\f} \otimes \bbchi_{\f}^{-\frac{1}{2}}$. Let us put $F^+T_\f^{\dagger}:=T_\f^{+}\otimes \bbchi_\f^{-\frac{1}{2}}$ and consider the map 
\begin{equation}
\label{eqn_2022_12_15_1140}
{\rm pr}_\f^\dagger\,: \,\cR_\f\hatotimes_{\Zp} \LL(\Gamma_\cyc)\lra  \cR_\f
\end{equation}
induced from $\gamma\mapsto \bbchi_\f^{-\frac{1}{2}}(\gamma)$. We assume that 
 \begin{itemize}
         \item[\mylabel{item_NA}{\bf NA})] $\alpha_\f\not\equiv \mathds{1} \mod \mathfrak{m}_{\cR_\hh}$\,.
    \end{itemize}
It follows (thanks to \eqref{item_NA}) that the natural map $H^1_{\rm Iw}(\Qp(\mu_{p^\infty}),T_\f^{+})\xrightarrow{{\rm pr}_\f^\dagger} H^1(\Qp,F^+T_\f^{\dagger})$ is surjective, which in turn allows us to define
\begin{equation}
\label{eqn_2022_12_15_1134}
{\rm Log}_{F^+T_\f^{\dagger}}\,: \, H^1(\Qp,F^+T_\f^{\dagger}) \xrightarrow{{\rm Log}_{T_\f^+}\otimes_{{\rm pr}_\f^\dagger} \cR_\f}  {\bf D}(T_\f^{+})\,.
\end{equation}

 Finally, we let ${\rm Log}_{\omega_\f}^\dagger$ denote the composite map
 $$ H^1(\QQ_p,F^+T_\f^\dagger)\xrightarrow{{\rm Log}_{F^+T_\f^{\dagger}}} {\bf D}(T_\f^{+})\xrightarrow{\,\omega_\f\,}\cR_\f\,,$$
where $H^1_{\rm f}(\QQ_p,T_\f^\dagger):=H^1_{\rm f}(\QQ_p,F^+T_\f^\dagger)$. We also denote by $H^1_{\rm f}(\QQ,T_\f^\dagger)$ the Greenberg Selmer group attached to $T_\f^\dagger$ with local conditions at $p$ induced from the inclusion $F^+T_\f^\dagger\subset T_\f^\dagger$. By a slight abuse of notation, we will sometimes write ${\rm Log}_{\omega_\f}^\dagger$ to notate the map
$$H^1_{\rm f}(\QQ,T_\f^\dagger) \xrightarrow{{\rm Log}_{\omega_\f}^\dagger\,\circ\,\res_p} \cR_\f\,. $$
\subsubsection{}
\label{subsec_212_2023_09_27_1232}
The $G_{\QQ_p}$-stable submodules $T_\g^+\subset T_\g$ and $T_{\g^c}^+\subset T_{\g^c}$ are orthogonal complements of one another under the perfect pairing \eqref{eqn_2023_09_26_1114}. As a result, it induces a $G_{\QQ_p}$-equivariant isomorphism 
$$T_\g^\pm\otimes T_{\g^c}^\mp\xrightarrow{\,\,\sim\,\,} \bbchi_\g\chi_\cyc\,.$$
This in turn induces a canonical isomorphism
$${\bf D}(T_\g^{-})\otimes {\bf D}(T_{\g^c}^{+})\xrightarrow{\,\,\sim\,\,} \cR_\g\,,$$
which fits in the following commutative diagram:
\begin{equation}
    \label{eqn_2023_09_27_1038}
    \begin{aligned}
        \xymatrix{
    {\bf D}(T_\g^{-})\otimes {\bf D}(T_{\g^c}^{+}) \ar[r]^-{\sim} \ar[d]_{\eta_\g\otimes \omega_{\g^c}} &\cR_\g\ar[d]
    \\
     \dfrac{1}{H_\g} \cR_\g \otimes \cR_\g\ar[r]_-{a\otimes b \mapsto ab} &\dfrac{1}{H_\g} \cR_\g\,.
    }
    \end{aligned}
\end{equation}
where $\eta_\g$ and $\omega_\g$ are constructed in \cite[Proposition 10.1.1]{KLZ2}.

\subsubsection{} Let us put 
$$M^{(\g)}:=T_\f^+(\bbchi_\f^{-1}\chi_\cyc^{-1})\widehat\otimes_{\ZZ_p} (T_\g^-\otimes_{\cR_\g} T_{\g^c}^+(\bbchi_\g^{-1}\chi_\cyc^{-1}))\,,$$ which is the unramified twist of the $G_{\QQ_p}$-representation $T_\f^+\widehat\otimes_{\ZZ_p} T_\g^-\otimes_{\cR_\g} T_{\g^c}^+$. Let us also set 
$${\bf D}(T_\f^+\widehat\otimes_{\ZZ_p} T_\g^-\otimes_{\cR_\g} T_{\g^c}^+):=(M^{(\g)}\widehat\otimes_{\ZZ_p}\ZZ_p^{\rm ur})^{G_{\QQ_p}}={\bf D}(T_\f^+)\widehat\otimes_{\ZZ_p}{\bf D}(T_\g^-)\otimes_{\cR_\g} {\bf D}(T_{\g^c}^+)\,.$$
As above, the general discussion in \cite[\S8.2]{KLZ2} applies and gives rise to, in the terminology of \cite{BSV}, the $\g$-logarithm map
\begin{align*}
        {\rm Log}^{(\g)}\,:\, H^1_{\rm Iw}(\QQ_p(\mu_{p^\infty}),T_\f^+\widehat\otimes_{\ZZ_p} T_\g^-\otimes_{\cR_\g} T_{\g^c}^+(\bbchi_\g^{-1}\chi_\cyc^{-1}))\lra &\,{\bf D}(T_\f^+\widehat\otimes_{\ZZ_p} T_\g^-\otimes_{\cR_\g} T_{\g^c}^+) \otimes \LL(\Gamma_\cyc)\\
        &\quad = {\bf D}(T_\f^+)\widehat\otimes_{\ZZ_p}{\bf D}(T_\g^-)\otimes_{\cR_\g} {\bf D}(T_{\g^c}^+)\otimes \LL(\Gamma_\cyc)\,.
\end{align*}
Note that, thanks to our discussion in \S\ref{subsec_212_2023_09_27_1232}, the $G_{\QQ_p}$-action on the factor $T_\g^-\otimes_{\cR_\g} T_{\g^c}^+(\bbchi_\g^{-1}\chi_\cyc^{-1})$ is trivial. As a result, arguing as in \S\ref{subsubsec_211_2023_09_27_1236} (and still assuming that the non-anomality condition \eqref{item_NA} holds), we obtain a map
\begin{equation}
\notag
    {\rm Log}^{(\g),\dagger}\,:\, H^1(\QQ_p,F^+T_\f^\dagger\widehat\otimes_{\ZZ_p} T_\g^-\otimes_{\cR_\g} T_{\g^c}^+(\bbchi_\g^{-1}\chi_\cyc^{-1}))\lra  {\bf D}(T_\f^+)\widehat\otimes_{\ZZ_p}{\bf D}(T_\g^-)\otimes_{\cR_\g} {\bf D}(T_{\g^c}^+)\,.
\end{equation}
Let us denote by $\omega^{(\g)}$ the map given by
$${\bf D}(T_\f^+)\widehat\otimes_{\ZZ_p}{\bf D}(T_\g^-)\otimes_{\cR_\g} {\bf D}(T_{\g^c}^+)\xrightarrow{\omega_\f\otimes\eta_\g\otimes \omega_{\g^c}} \cR_\f \widehat{\otimes}_{\ZZ_p}\dfrac{1}{H_\g} \cR_\g\otimes_{\cR_\g}\cR_\g\xrightarrow{a\otimes b\otimes c\,\mapsto\, a \otimes bc} \dfrac{1}{H_\g} \cR\,,$$
and by ${\rm Log}^{(\g),\dagger}_{\omega^{(\g)}}$ the composite map (we recall that $T^{\dagger}:=T\otimes\bbchi_{\f}^{-\frac{1}{2}} = T_{\f}^{\dagger}\widehat\otimes\rm{ad}(T_{\g}$))
\begin{align}
\begin{aligned}
    \label{eqn_2024_07_02_1246}
     H^1_{\rm bal}(\QQ_p,T^\dagger)\xrightarrow{\eqref{eqn_2023_09_26_1043}} H^1_{\rm bal}(\QQ_p,T_\f^\dagger\widehat\otimes_{\ZZ_p}& T_\g\otimes_{\cR_\g} T_{\g^c}(\bbchi_\g^{-1}\chi_\cyc^{-1}))\\
     &\lra  H^1(\QQ_p,F^+T_\f^\dagger\widehat\otimes_{\ZZ_p} T_\g^-\otimes_{\cR_\g} T_{\g^c}^+(\bbchi_\g^{-1}\chi_\cyc^{-1}))\\
    &\qquad\xrightarrow[{\rm Log}^{(\g),\dagger}]{} {\bf D}(T_\f^+)\widehat\otimes_{\ZZ_p}{\bf D}(T_\g^-)\otimes_{\cR_\g} {\bf D}(T_{\g^c}^+)\xrightarrow[\omega^{(\g)}]{} \dfrac{1}{H_\g} \cR\,,
\end{aligned}
\end{align}
where the balanced local conditions $H^1_{\rm bal}(\QQ_p,T_\f^\dagger\widehat\otimes_{\ZZ_p} T_\g\otimes_{\cR_\g} T_{\g^c}(\bbchi_\g^{-1}\chi_\cyc^{-1}))$ are given as in \cite[\S7.2]{BSV} (explicitly, these are the Greenberg local conditions given by 
\begin{equation}
    \label{eqn_2024_07_09_1145}
    F_{\rm bal}T^\dagger:=\left(F^+T_\f^\dagger \widehat{\otimes}_{\ZZ_p} (F^+T_\g\otimes_{\cR_\g}T_{\g^c}+T_\g\otimes_{\cR_\g}F^+T_{\g^c})+T_\f^\dagger \widehat{\otimes}_{\ZZ_p} F^+T_\g\otimes_{\cR_\g}F^+T_{\g^c}\right)(\bbchi_{\g}^{-1}\chi_\cyc^{-1})\,,
\end{equation}
which is the sum of the modules in \eqref{eqn_2024_07_03} to the left of the three arrows), and the second map above is given in the statement of Proposition 7.3 of op. cit. (where this map would be denoted ${p}_{\g\, *}$ in the notation therein). We also define the balanced Selmer group $H^1_{\rm bal}(\QQ,T_\f^\dagger\widehat\otimes_{\ZZ_p} {\rm ad}(T_\g))$ that consists of global classes that are unramified away from $p$ and verify the balanced conditions at $p$, namely that
$$H^1_{\rm bal}(\QQ,\,T^\dagger)\xrightarrow{\res_p} H^1(\QQ_p,\,T^\dagger)\lra \frac{H^1(\QQ_p,\,T^\dagger)}{H^1_{\rm bal}(\QQ_p,\,T^\dagger)}$$  
is the zero map. 

In view of our discussion in \S\ref{subsubsec_114_2023_09_27_1616} and \S\ref{subsec_212_2023_09_27_1232}, the following diagram commutes:
\begin{equation}
    \label{eqn_2023_09_27_1233}
    \begin{aligned}
        \xymatrix{
        H^1_{\rm bal}(\QQ_p,T^\dagger)\ar[rr]^-{{\rm Log}^{(\g),\dagger}_{\omega^{(\g)}}} \ar[d]_{{\rm id}\otimes {\rm tr}}&&\dfrac{1}{H_\g} \cR\\
        H^1_{\rm f}(\QQ_p,T_\f^\dagger)\widehat\otimes\cR_\g\ar[rr]_-{{\rm Log}_{\omega_\f}^\dagger\otimes {\rm id}} && \cR\ar[u]
        }
    \end{aligned}
\end{equation}
Here the vertical map on the left is given as follows. The pairing \eqref{eqn_2023_09_26_1114} induces $G_{\QQ_p}$-equivariant maps  
\begin{align}
   \begin{aligned}
   \label{eqn_2024_07_03}
       F^+T_\f^\dagger\widehat\otimes_{\ZZ_p} \left(T_\g^+\otimes_{\cR_\g} T_{\g^c} \right)(\bbchi_\g^{-1}\chi_\cyc^{-1})\,&\xrightarrow{{\rm id}\otimes \langle\,,\,\rangle_\g} F^+T_\f^\dagger
\\
F^+T_\f^\dagger\widehat\otimes_{\ZZ_p} \left(T_\g\otimes_{\cR_\g} T_{\g^c}^+ \right)(\bbchi_\g^{-1}\chi_\cyc^{-1})\,&\xrightarrow{{\rm id}\otimes \langle\,,\,\rangle_\g} F^+T_\f^\dagger
\\
T_\f^\dagger\widehat\otimes_{\ZZ_p} \left(T_\g^+\otimes_{\cR_\g} T_{\g^c}^+ \right)(\bbchi_\g^{-1}\chi_\cyc^{-1})\,&\xrightarrow{{\rm id}\otimes \langle\,,\,\rangle_\g} 0
   \end{aligned}
\end{align}
since $T_{\g}^+$ and $T_{\g^c}^+$ are orthogonal complements under this pairing. This, together with \eqref{eqn_2023_09_26_1321}, shows that $H^1_{\rm bal}(\QQ_p,T^\dagger)$ maps under the indicated map to $H^1_{\rm f}(\Q,T_\f^\dagger) \widehat\otimes_{\ZZ_p}\cR_\g$.

\subsubsection{}
\label{subsubsec_2023_10_02_1712}
As we have noted in \S\ref{subsubsec_2022_05_16_1506}, the weight spaces $\cW_\f$ and $\cW_\g$ of the Hida families $\f$ and $\g$ are both supported over the connected component of $k_\circ=0$ in ${\rm Spf}(\LL_{\rm wt})$. As a result, since we assume that $\varepsilon_\f=\mathds{1}=\varepsilon_\g^{(p)}$, it follows that they both admit a crystalline specialization of weight $2$. We denote by $f\in S_2(\Gamma_0(Np))$ the $p$-old form that is the said specialization of $\f$. 

Let us consider the following weakening of the anomality condition: 
 \begin{itemize}
         \item[\mylabel{item_NA_bis}{\bf NA$^\prime$})] $a_p(f)\neq 1$\,.
    \end{itemize}
Note that \eqref{item_NA} requires $a_p(f)-1$ be a $p$-adic unit, and it is therefore stronger than \eqref{item_NA_bis}. When \eqref{item_NA_bis} holds true, one may choose a wide open disc in $U_\f\subset {\rm Spf}(\LL_{\rm wt})$ such that the torsion $\LL_{\rm wt}$-module 
$${\rm coker}(H^1_{\rm Iw}(\Qp(\mu_{p^\infty}),T_\f^{+})\xrightarrow{{\rm pr}_\f^\dagger} H^1(\Qp,F^+T_\f^{\dagger}))$$ has no support in $U_\f$ (see the proof of \cite[Proposition 7.3]{BSV} for a similar argument). Working over such $U_\f$ (rather than the entire weight space) as in op. cit. and inverting $p$, we obtain the big Perrin-Riou map (which we still denote ${\rm Log}_{\omega_\f}$)
$$H^1_{\rm f}(\QQ,T_\f^\dagger)\otimes_{\LL_{\rm wt}} \cO(U_\f)\xrightarrow{{\rm Log}_{\omega_\f}} \cO(U_\f)\,,$$
where $\cO(U_\f)\simeq \LL(1+p\ZZ_p)[\frac{1}{p}]$ denotes the ring of power-bounded functions on $U_\f$. 

The same applies\footnote{This was already included in the definition of these maps in \cite{BSV}; see especially \S7.3 in op. cit.} to ${\rm Log}^{(\g),\dagger}_{\omega^{(\g)}}$: we have a big Perrin-Riou map
$$H^1_{\rm bal}(\QQ,T^\dagger)\otimes_{\LL_{\rm wt}} \cO(U_\f)\xrightarrow{{\rm Log}^{(\g),\dagger}_{\omega^{(\g)}}} \cR\otimes_{\LL_{\rm wt}} \cO(U_\f)\,.$$
These versions of Perrin-Riou maps on smaller discs in the respective weight spaces are unfortunately insufficient for our purposes since our methods will require that we work with an infinite sequence of weight-$2$ specializations of $\f$ (as a result, allow wild characters of arbitrary order), whereas $U_\f$ contains at most finitely many such specializations.

\subsection{Reformulation of Conjecture~\ref{conj_main_6_plus_2} in terms of Heegner points}
Let us fix a quadratic imaginary field $K$ with maximal order $\cO_K$, and an ideal $\mathfrak N \subset \cO_K$ such that $\cO_K/\mathfrak N \simeq \Z/N\Z$ (we recall that $p\nmid N$). We review the main constructions and results in \cite{howard2007Central, howard2007Inventiones}, and our notation in this section is borrowed from these works. 
\subsubsection{Big Heegner points}
\label{subsubsec_2024_02_07_1331}
The construction of \cite[\S 2.2]{howard2007Inventiones} provides us with a class $\mathfrak X = \mathfrak X_1 \in H^1(H, T_\hf^\dagger)\,.$
Let us set 
\[
    \mathfrak z := \mathrm{Cores}_{H/K}\, \mathfrak X \in H^1(K, T_\f^\dagger)\,,
\]
where $H/K$ is the Hilbert class field of $K$. In fact, Howard constructs classes $\mathfrak X_{p^s}\in H^1(H_{p^s}, T_\hf^\dagger)$ for every natural number $s$, where $H_{p^s}$ is the ring class extension of $K$ of conductor $p^s$. These are related to $\mathfrak X$ via the equation
\begin{equation}
    \label{eqn_subsubsec_2024_02_07_1331}
    \sum_{\sigma\in \Gal(H_{p^s}/H)} \mathfrak X_{p^s}^\sigma = U_p^s \cdot \mathfrak X\,.
\end{equation}

\subsubsection{Twisted Heegner points}
\label{subsubsec_2024_07_09_0945}
For any $s\geq 0$, let us define the elliptic curve $E_s(\C)\simeq \C/\cO_{p^s}$ (where $\cO_{p^s}=\ZZ+p^s\cO\subset \cO$ is the order of conductor $p^s$), as well as its subgroup $\mathfrak n_s:=E_s[\mathfrak N\cap\cO_{p^s}]$. The inclusion of the orders $\cO_{p^{s+1}}\subset \cO_{p^s}$ induces a $p$-isogeny $E_{s+1}\to E_s$ that is compatible with the action of $\cO_{p^{s+1}}$ on the source and target, and it maps $\mathfrak n_{s+1}$ isomorphically onto $\mathfrak n_s$. The kernel of the isogeny $E_s \to E_0$ is the $p^s\cO$-torsion of $E_s$, and it is cyclic of order $p^s$. Any choice of a generator $\varpi$ of $\cO/\Z$ (which we fix once and for all) gives rise to a generator $\pi_s\in \ker(E_s \to E_0)$, and it in turn defines a $\Gamma_1(p^s)$-level structure $(E_s,\pi_s)$. As a result, we have a point
\[
    h_s := (E_s, \mathfrak n_s, \pi_s)\in X_s(L_s)
\]
 on the modular curve $X_s$ of level $\Gamma_0(N)\cap \Gamma_1(p^s)$, where $L_s:=H_{p^s}(\mu_{p^s})$. 
 
Let us define $\pmb\chi:K^\times \setminus \mathbb A_K^\times \to \mathcal R_\f^\times$ by setting
\[
    \pmb \chi(x) =  \bbchi_\hf^{1/2}\Bigl(\rec_\Q(\mathrm{N}_{K/\Q}(x))\Bigr)\,,
\] 
where $\rec_\Q$ is the geometrically normalized Artin map of class field theory. Let us fix a rational cusp $\mathfrak c\in X_s(\Q)$ and define
\[
    Q_\kappa = \sum_{\sigma\in \Gal(L_s/K)} \pmb\chi_\kappa^{-1}(\sigma)(h_s-\mathfrak c)^\sigma \in J_s(L_s)\otimes F_\kappa
\]
for any $\kappa\in \cW_\f$ of weight $2$ and wild character $\psi_\kappa$ of conductor $p^s$,
where $J_s={\rm Jac}(X_s)$ is the Jacobian of $X_s$ and $F_\kappa:=\cR_{\f,\p_\kappa}/\p_\kappa \cR_{\f,\p_\kappa}$ is a finite extension of $\QQ_p$. 

\subsubsection{} Let $\kappa\in \cW_\f^{\rm cl}$ be a specialization  with wild character $\psi_\kappa$, that has the property that the conductor of $\psi_\kappa$ equals $p^{s_\kappa}$ and the associated eigenform $\f_\kappa$ is of weight $2$ and new of level $\Gamma_0(N_\f)\cap \Gamma_1(p^{s_\kappa})$. For any completion $F_{s_\kappa}$ of $L_{s_\kappa}$ at a prime above $\p$, we will denote by $\log_{\omega_{\f_\kappa}}$ (with an admitted abuse of notation) the formal logarithm $J_s(F_{s_\kappa})\to \mathbb{C}_p$ associated to the differential $\omega_{\f_\kappa}$.

\subsubsection{Specializations of big Heegner points}
\label{subsubsec_222_2023_11_16_1203}
For $\kappa$ and $s_\kappa=s$ as in the previous paragraph, the specialization map $T_\f\to T_{\f_\kappa}$ factors as 
$$T_\f\lra \mathrm{Ta}_p^{\rm ord}(J_s) \lra T_{\f_\kappa},$$
where $\mathrm{Ta}_p(J_s)$ is the $p$-adic Tate module of $J_s$ and $\mathrm{Ta}_p^{\rm ord}(J_s):=e_{\rm ord}\mathrm{Ta}_p(J_s)$. We therefore have a natural map
\[
    J_s(L_s) \lra H^1(L_s, \mathrm{Ta}_p(J_s)) \xrightarrow{e_{\rm ord}} H^1(L_s, \mathrm{Ta}_p^{\rm ord}(J_s)) \lra H^1(L_s, T_{\f_\kappa}),
\]
where the first arrow is the Kummer map. By an abuse of notation, let us denote the image of $Q_\kappa$ under this map also by $Q_\kappa$. We also remark that we have $T_{\f_\kappa}\simeq T_{\f_\kappa}^\dagger$ as $G_{L_s}$-representations. We note that this isomorphism is uniquely determined by the choice of a generator of $\mu_{p^\infty}$, which we have fixed throughout. We therefore have a fixed isomorphism
$$H^1(L_s,T_{\f_\kappa}^\dagger)\simeq H^1(L_s,T_{\f_\kappa})\,.$$

\begin{lemma}
    \label{lemma_2023_11_18_1624}
For $\kappa$ as above,  we have  
$$\res_{L_s/K}\,\mathfrak{z}_\kappa= [L_s:H_{p^s}]^{-1}\cdot a_p(\f_\kappa)^{-s}\,Q_\kappa\,.$$
\end{lemma}

\begin{proof}
    This is immediate from definitions of these objects in \cite{howard2007Central, howard2007Inventiones}. See also the discussion in \cite[Page 809]{howard2007Central}.
\end{proof}

\subsubsection{Reciprocity law} \label{BDP_p-adic_L-function} 
Before stating Castella's reciprocity laws for big Heegner points (which amounts to an interpolation of the celebrated Bertolini--Darmon--Prasanna formula), we first recall the $p$-adic $L$-function (in two variables) introduced in \cite[\S 2.7]{CastellapadicvariationofHeegnerpoints}, which comes attached to the Hida family $\f$ with $\varepsilon_\f=\mathds 1$\footnote{We remark that Castella's $k_\nu$ in op. cit. coincides with what we denote by ${\rm wt}(\nu)+2$. Moreover, our $\mathcal R_\f$ is denoted by $\mathbb I$ in \cite{CastellapadicvariationofHeegnerpoints}, whereas our $\varepsilon_f$ is $\psi_0$, and our $N_\f$ is corresponds to $N$ in op. cit. Finally, the character denoted by $\Theta$ (cf. Definition 2.8 of \cite{CastellapadicvariationofHeegnerpoints}) is our $\bbchi_{\f}^{\frac{1}{2}}$.}.

Let us put $D_\infty:=\varprojlim_s \Gal(H_{p^s}/H)\simeq \ZZ_p^\times$ (where the isomorphism is determined by a choice of a topological generator of $D_\infty$, which we fix henceforth). We denote by $\LL^{\rm ac}$ the completed group ring $W(\overline{\FF}_p)[[D_\infty]]$, where $W(\overline{\FF}_p)$ is the ring of Witt vectors of $\overline{\FF}_p$. We shall consider ${\rm Spf}(\LL^{\rm ac})$ as the anticyclotomic weight space. Note that $\LL^{\rm ac}\simeq \LL_{\rm wt}$, with which we regard $\kappa\in \cW_\f^{\rm cl}$ also as an element of the anticyclotomic weight space.  Let $\pmb\xi$ denote the universal anticyclotomic character given as in \cite[Definition 2.8(3)]{CastellapadicvariationofHeegnerpoints}.

Let us consider a finite-order Hecke character $\chi$ over $K$ such that $\chi\vert_{\mathbb A_\Q^\times} = \mathds 1$, and an anticyclotomic character $\psi$ of conductor $c\cO_K$. The Hecke  $L$-series 
\[
    L(f_\kappa/K, \chi\psi,s) = L\left(s-\frac{{\rm wt}(\kappa)+2}{2},\pi_K\otimes \chi\psi\right)
\]
satisfies a functional equation relating values at $s$ and $\wt(\kappa)+2-s$. Theorem 2.11 of \cite{CastellapadicvariationofHeegnerpoints}, which extends the previous constructions in \cite{bertolinidarmonprasanna13, miljan}, proves the existence of a $p$-adic $L$-function $\mathcal L_{p,\pmb\xi}$ with the following properties.
\begin{theorem}[Castella] \label{CastellaInterpolation}
    We have a two-variable $p$-adic L-function $\mathcal L_{p,\pmb\xi}(\f\otimes \pmb\xi^{-1})\in \cR_\f\,\widehat{\otimes}_{\ZZ_p}\,\LL^{\rm ac}$ that is characterized by the following interpolation property.  Suppose that $\kappa\in \mathcal W_\f^\cl$ is a crystalline classical point. Let us put $k:={\rm wt}(\kappa) \geq 0$ so that the specialization $\f_\kappa$ is of weight $k+2$. Then:
    \begin{align*}
        \frac{\mathcal L_{p,\pmb\xi}(\f\otimes \pmb\xi^{-1})^2(\kappa,\mu)}{\Omega_p^{2k+4+4m} }
        =
        \mathcal E_p(\f_\kappa,\pmb\chi_\kappa\pmb\xi_\kappa\pmb\xi_\mu^{-1},k+1)^2 \cdot &\,\frac{\pmb\xi_\mu(\mathfrak N^{-1})\cdot 2 \cdot \varepsilon(\f_\kappa)\cdot w_K^2\cdot \sqrt{D}}{{\rm Im}(\vartheta)^{k+2+2m}} 
        \\
        &\qquad\qquad  \times\frac{\pi^{2k+4+4m}\cdot \Lambda(\f_\kappa/K, \pmb\chi_\kappa\pmb\xi_\kappa\pmb\xi_\mu^{-1},k+1)}{\Omega_K^{2k+4+4m}}\,.
    \end{align*}
      Here:
    \begin{itemize}
    \item $\mu \in {\rm Spf}(\LL^{\rm ac})\simeq {\rm Spf}(\LL_{\rm wt})$ is an arithmetic specialization with  ${\rm wt}(\mu)=:-m \in 2\ZZ_{\leq 0}$, so that its specialization $\pmb \xi_\mu$ is an algebraic Hecke character with infinity type $(-m/2,m/2)$. 
        \item $\Omega_K$ and $\Omega_p$ are the complex and $p$-adic CM periods as defined in \cite[\S 2.5]{CastellaHsiehGHC}.
        \item For an anticyclotomic Hecke character $\psi$ of conductor $p^n\cO_K$, the $p$-adic multiplier $\mathcal E_p(\f_\kappa,\pmb\chi_\kappa\psi,k+1)$ is given by
        $$\mathcal E_p(\f_\kappa,\pmb\chi_\kappa\psi,k+1)   = 
    \begin{cases}
    (1-a_p(\f_\kappa)\cdot (\pmb\chi_\kappa\psi)_{\overline{\p}}(p)\cdot p^{-k/2-1}) 
    \cdot 
    (1-a_p(\f_\kappa)\cdot (\pmb\chi_\kappa\psi)_{\overline{\p}}(p)\cdot p^{-k/2-1})   &\text{if $n=0$}\\
    \varepsilon((\pmb\chi_\kappa\psi)^{-1}_\p)\cdot p^{-n} &\text{if $n\geq 1$}\,,
    \end{cases}$$
    where $\varepsilon((\pmb\chi_\kappa\psi)^{-1}_\p)$ is the epsilon-factor given as in \cite[p. 2132]{CastellapadicvariationofHeegnerpoints}.
        \item $w_K= |\cO_K^\times|$.
        \item $\varepsilon(\f_\kappa)$ is the global root number of $\f_\kappa$\,.
        \item $\vartheta:=(D_K+\sqrt{-D_K})/2$.
    \end{itemize}
\end{theorem}

In what follows, we will denote by
\[
\mathcal L_p^\dagger(\f/K) = \mathcal L_{p,\pmb\xi}(\f\otimes \pmb\xi^{-1})^2(\kappa,\kappa) \in \cR_\f\otimes_{\ZZ_p}W(\overline{\FF}_p)
\]
the restriction of $\mathcal L_{p,\pmb\xi}$ to the line $\mu=\kappa$.

\begin{corollary} \label{SpecializationBDP}
    Let $\kappa\in \mathcal W_\f$ be a classical point of weight $2$ and non-trivial wild character $\chi_\kappa$ with conductor $p^{s_\kappa}$ such that the associated specialization $\f_\kappa$ is new of level $\Gamma_0(N)\cap \Gamma_1(p^{s_\kappa})$. Then, 
    \[
        \mathcal L_p^\dagger(\f/K)(\kappa) 
        =    
       \pmb\xi_{\kappa,\p}(-1)\varepsilon(\pmb\xi_{\kappa,\p})\,p^{1-s_\kappa}[L_{s_\kappa}:H_{p^{s_\kappa}}]^{-1} \cdot \log_{\omega_{\f_\kappa}}(Q_\kappa)\,.
    \]
\end{corollary}

\begin{proof}
    This follows unfolding the proof of \cite[Proposition 5.4]{CastellapadicvariationofHeegnerpoints}. In particular, from the first displayed equality on Page 2157 of op. cit., it follows that
$$
        \mathcal L_p^\dagger(\f/K)(\kappa) = \mathcal L_{\p,\pmb\xi}(\f\otimes \pmb\xi^{-1})(\kappa,\kappa) =
       \pmb\xi_{\kappa,\p}(-1)\varepsilon(\pmb\xi_{\kappa,\p})\,p^{1-s} \cdot \sum_{\sigma \in \Gal({H}_{p^{s}}/K)} \log_{\omega_{\f_\kappa}}(\res_\p\kappa(\mathfrak{X}_{p^s}^\sigma))\,,
$$
where we have put $s=s_\kappa$ to lighten our notation. Combining this identity with \eqref{eqn_subsubsec_2024_02_07_1331} and the definition of the class $\mathfrak{z}$, we infer that 
\begin{align*}
    \mathcal L_p^\dagger(\f/K)(\kappa)=\pmb\xi_{\kappa,\p}(-1)\varepsilon(\pmb\xi_{\kappa,\p})\,p^{1-s} a_p(\f_\kappa)^s \log_{\omega_{\f_\kappa}}(\res_\p ({\rm tr}_{H/K}\mathfrak{X}_{\kappa})) \\
    = \pmb\xi_{\kappa,\p}(-1)\varepsilon(\pmb\xi_{\kappa,\p})\,p^{1-s} a_p(\f_\kappa)^s \log_{\omega_{\f_\kappa}}(\res_\p (\mathfrak{z}_{\kappa}))\,.\\
\end{align*}
The proof follows from this equality combined with Lemma~\ref{lemma_2023_11_18_1624}.
\end{proof}

\subsubsection{$p$-adic $L$-functions of Hida and Bertolini--Darmon--Prasanna}
\label{subsubsec_226_2024_03_04}
In what follows (until the end \S\ref{subsubsec_226_2024_03_04}), we closely follow \cite[\S8]{BCS}, where we also borrow our notation below.

We begin this subsection by recalling the following result due to Hida:
\begin{theorem}[Hida \cite{AIF_Hida88}] \label{thm:Hida-interp}
      There exist a unique pair of $p$-adic  $L$-functions 
    \[
    \cL_p^\hf(\hf\otimes \hg),\quad \cL_p^\hg(\hf\otimes \hg)\in \cR_\hf \hatotimes_{\ZZ_p} \cR_\hg \hatotimes_{\ZZ_p} \Lambda(\Gamma_\cyc)
    \]
    with the following properties:
    \begin{align*}
        \cL_p^\hf({\hf\otimes \hg})(\kappa, \lambda,j) &= \mathcal E_p^{\f}({\f_\kappa^\circ}\otimes {\g_\lambda^\circ},j) \cdot (\sqrt{-1})^{\rmw(\lambda)+3-2j}\cdot \mathfrak C_{\rm exc}(\hf\otimes\hg) \cdot \frac{\Lambda({\f_\kappa^\circ}\otimes {\g_\lambda^\circ},j)}{\Omega_{\f_\kappa}},  \quad \forall (\kappa, \lambda,j) \in \cW_{\hf\hg\bf{j}}^\hf \cap \cW_{\f\g\bf{j}}^{\rm cris}\,,
        \\
        \cL_p^\hg({\hf\otimes \hg})(\kappa, \lambda,j) &= \mathcal E_p^{\g}({\f_\kappa^\circ}\otimes {\g_\lambda^\circ},j)
        \cdot 
        (\sqrt{-1})^{\rmw(\kappa)+3-2j}\cdot \mathfrak C_{\rm exc}(\hf\otimes\hg) \cdot \frac{\Lambda({\f_\kappa^\circ}\otimes {\g_\lambda^\circ},j)}{\Omega_{\g_\lambda}},  \quad \forall (\kappa, \lambda,j) \in \cW_{\hf\hg\bf{j}}^\hg \cap \cW_{\f\g\bf{j}}^{\rm cris}\,.
    \end{align*}
    Here, 
    $$\mathfrak C_{\rm exc}(\hf\otimes\hg)=\prod_{q\in \Sigma_{\rm exc}(\hf\otimes \hg)}(1+q^{-1})\,,$$
    where $\Sigma_{\rm exc}(\hf\otimes \hg)$ is a set of primes determined by the local properties of $\f$ and $\g$ and it is independent of the choice of the pair $(\kappa,\lambda)$ (see \cite{ChenHsieh2020}, \S 1.5).
\end{theorem}

\begin{proposition} \label{propositionHidaBDP}
    We have the following factorization of $p$-adic $L$-functions:
    \[
         \mathcal L_p^{\g_K}(\f\otimes\g_K)(\kappa,1,{\rm wt}(\kappa)/2) = \mathscr C_\mathrm{BDP}(\kappa)\cdot \mathcal L_p^\dagger(\f/K)(\kappa)^2\,.
    \]
    Here, $\g_K$ is the Hida family specializing in weight $1$ to the Eisenstein series associated with $K$ described in \cite[\S 8.1.8]{BCS}, and
    \[
        \mathscr C_{\rm BDP}(\kappa)= \frac{2^{2k-1} \cdot(-1)^k }{\pmb\xi_\kappa(\mathfrak N^{-1})}
        \frac{\mathfrak C_\mathrm{exc}(\f\otimes \Phi(\Theta))  \cdot \mathfrak c_{g_{\rm Eis}} }{\omega_K h_K \sqrt{-D_K} \cdot \varepsilon(\f_\kappa)\cdot  (1-p^{-1})  \log_p(u) }\,,
    \]
    where the Hida family $\Phi(\Theta)$ is given as in \cite[\S 8.1.2]{BCS}, and $u\in \mathcal O_K[\frac{1}{p}]^\times$ is such that $(u) = \p^{h_K}$.
\end{proposition}

\begin{proof}
For classical specializations $\kappa\in \mathcal W_\f^\cl$ that are crystalline with weight ${\rm wt}(\kappa)=:k>0$, we have $\pmb\chi_\kappa(\mathfrak a) = {\rm N}_{K/\Q}(\mathfrak a)^{\frac{k}{2}}$.  For such specializations $\kappa$, and all arithmetic crystalline specializations $\mu$ with $\wt(\mu)=-m\in 2\ZZ_{\leq 0}$ (so that the infinity type of $\pmb\xi_\mu^{-1}$ is $(m/2,-m/2)$), we have
    \begin{align*}
        \begin{aligned}
        \frac{\mathcal L_{p,\pmb\xi}(\f\otimes \pmb\xi^{-1})^2(\kappa,\mu)}{\Omega_p^{2k+4+2m} }
        &=
         \frac{ 2\cdot \varepsilon(\f_\kappa)\cdot w_K^2\cdot \sqrt{D}}{\pmb\xi_\mu(\mathfrak N^{-1}) \cdot {\rm Im}(\vartheta)^{k+2+m}} \cdot \mathcal E_p(f_\kappa,\pmb\xi_\kappa\pmb\xi^{-1}_\mu,k/2)^2 \cdot\frac{\pi^{2k+2+2m}\cdot \Lambda(f_\kappa/K, \pmb\xi_\kappa\pmb\xi^{-1}_\mu,k/2)}{\Omega_K^{2k+2\ell}}\\
         &=
         \frac{ 2\cdot \varepsilon(\f_\kappa)\cdot w_K^2\cdot \sqrt{D}}{\pmb\xi_\mu(\mathfrak N^{-1})\cdot {\rm Im}(\vartheta)^{k+2+m}} \cdot \mathcal E_p(f_\kappa\otimes \Phi(\Theta)_{\kappa\otimes\mu},k-m/2)^2 \\
         &\hspace{6cm}\times\frac{\pi^{2k+2+2m}\cdot \Lambda(f_\kappa\otimes \Phi(\Theta)_{\kappa\otimes\mu},k-m/2)}{ \Omega_K^{2k+4+2m} }
         \end{aligned}
    \end{align*}
    thanks to Theorem~\ref{CastellaInterpolation}. 
    
    We recall from \cite[\S 8.11]{BCS} (see also \S8.14--\S8.16 in op. cit.; we freely use the notation therein without further indication) that 
    $\Phi$ is a two-parameter family of characters of $G_K$ whose specialization to $\kappa\otimes\mu$ is given by $\pmb\xi_\kappa\pmb\xi_\mu^{-1} \mathbf{N}_K^{ \frac{\mathrm{wt}(\kappa)}{2}-\frac{\mathrm{wt}(\mu)}{2}}$, which are the $p$-adic avatars of Hecke characters with infinity type $(k+m,0)$. In particular, we remark that the modular form $\Phi(\Theta)_{\kappa\otimes\kappa}$ is the unique $p$-stabilization of the Eisenstein series ${\rm Eis}_1(\varepsilon_K)$. Applying Theorem~\ref{thm:Hida-interp} with $\g=\Phi(\Theta)$, we infer that
    \begin{equation} \label{ComparisonBDP1}
        \begin{aligned}
        \frac{\mathcal L_{p,\pmb\xi}(\f\otimes \pmb\xi^{-1})^2(\kappa,\mu)}{\Omega_p^{2k+4+2m} }
        = \,&
        \frac{\pmb\xi^{-1}_\mu(\mathfrak N^{-1})\cdot 2\cdot \varepsilon(\f_\kappa)\cdot w_K^2\cdot \sqrt{D}}{{\rm Im}(\vartheta)^{k+2+m} \cdot \mathfrak C_\mathrm{exc}(\f\otimes \Phi(\Theta)) (\sqrt{-1})^{m-k+3 
        } } 
        \cdot
         \frac{\pi^{2k+4+2m}\cdot \Omega_{\Phi(\Theta)_{\kappa\otimes\mu}} }{\Omega_K^{2k+4+2m}}\\
         &\hspace{4 cm} \times \mathcal L_p^{\Phi(\Theta) }(\f\otimes\Phi(\Theta))(\kappa, \kappa\otimes\mu, k-m/2) \, ,
         \end{aligned}
    \end{equation}
    where we recall that $\Omega_{\Phi(\Theta)_{\kappa\otimes\mu}}$ is the modified Hida period given as in \cite[\S. 3.2.1]{BCS}. 
    
    Using \cite[Equation (8.13)]{BCS}, we may calculate
    \begin{align}
    \begin{aligned}
\label{eqn_2024_03_04_1655}
         \frac{\pi^{2k+4+2m}\cdot \Omega_{\Phi(\Theta)_{\kappa\otimes\lambda}} }{\Omega_K^{2k+4+2m}} =  \frac{(-2\sqrt{-1})^{k+m+4}}{\mathfrak c_{\Phi(\Theta)(\kappa\otimes\mu)}} \,\cdot\,&  \mathscr O(\Phi_{\kappa\otimes\mu}) \cdot \frac{\pi^{k+m+1}\cdot (k+2+m)!}{2^{2k+2m+5}}\\
        &\times \mathcal E_p(\Phi(\Theta)_{\kappa\otimes\mu},{\rm ad}) \cdot \frac{L_{\mathbf{N}_{K/\Q}(\mathfrak f_{\Phi_{\kappa\otimes\mu}})}(\Phi_{\kappa\otimes\mu}\Phi_{\kappa\otimes\mu}^{c,-1},1)}{\Omega_K^{2k+4+2m}} \,.
    \end{aligned}
    \end{align} 
    
    On plugging $\widetilde{\Phi}_{\kappa\otimes\mu} = \Phi_{\kappa\otimes\mu}\Phi_{\kappa\otimes\mu}^{c,-1}\mathbf{N}_K = \pmb\xi_\kappa^2\pmb\xi_\lambda^{-2} \mathbf{N}_K$ (which is the $p$-adic avatar of a Hecke character with infinity type $(k+2+m+1,-k-2-m+1)$) in \cite[Eqn. (3-1)]{BDP2}, we also have that
    \begin{equation}
        \left(\frac{\sqrt{D_K}}{2\pi}\right)^{k+2+m-1} \frac{\mathcal L_p^{\rm Katz}(\widetilde{\Phi}_{\kappa\otimes\lambda})}{\Omega_p^{2k+4+2m}} = \mathcal E_p(\Phi(\Theta)_{\kappa\otimes\mu},{\rm ad})  \cdot (k+2+m)! \cdot \frac{L(\widetilde{\Phi}_{\kappa\otimes\mu}^{-1},0)}{\Omega_K^{2k+4+2m}}\,,
       \end{equation} 
    which, combined with \eqref{eqn_2024_03_04_1655}, permits us to conclude
    \[
    \frac{\pi^{2k+4+2m}\cdot \Omega_{\Phi(\Theta)_{\kappa\otimes\mu}} }{\Omega_K^{2k+4+2m}} =  \frac{(-2\sqrt{-1})^{k+2+m+2}}{\mathfrak c_{\Phi(\Theta)(\kappa\otimes\mu)}} \cdot  \mathscr O(\Phi_{\kappa\otimes\mu}) \cdot \frac{(\sqrt{D_K})^{k+2+\ell-1}}{2^{3k+6+3m}}\cdot \frac{\mathcal L_p^{\rm Katz}(\widetilde{\Phi}_{\kappa\otimes\mu})}{\Omega_p^{2k+4+2m}}.
    \]
    Substituting this identity in Equation \eqref{ComparisonBDP1}, we deduce that
    \begin{equation} 
        \begin{aligned}
        \mathcal L_{p,\pmb\xi}(\f\otimes \pmb\xi^{-1})^2(\kappa,\mu)
        &=
        \mathscr C_{\rm BDP}(\kappa,\mu)^{-1} \cdot \mathcal L_p^{\Phi(\Theta) }(\f\otimes\Phi(\Theta))(\kappa, \kappa\otimes\mu, k-m/2)\,,
         \end{aligned}
    \end{equation}
    where
        \begin{align*}
            \begin{aligned}
        \mathscr C_{\rm BDP}(\kappa,\mu) &= 
        \frac{
        {\rm Im}(\vartheta)^{k+2+m} \cdot \mathfrak C_\mathrm{exc}(\f\otimes \Phi(\Theta)) (\sqrt{-1})^{\ell-k-2+1 
        }  
        }{
        \pmb\xi^{-1}_\mu(\mathfrak N^{-1})\cdot 2\cdot \varepsilon(\f_\kappa)\cdot w_K^2\cdot \sqrt{D} \cdot \mathscr O(\Phi_{\kappa\otimes\mu})
        }
        \frac{
        \mathfrak c_{\Phi(\Theta)(\kappa\otimes\mu)}
        }{
        (-2\sqrt{-1})^{k+2+m+2}
        } \\
        &\hspace{6 cm}\times
        \frac{
        2^{3k+6+3m}
        }{
        (\sqrt{D_K})^{k+2+m-1}
        } \cdot\frac{1}{\mathcal L_p^{\rm Katz}(\widetilde{\Phi}_{\kappa\otimes\mu})} \\
        &
        = \frac{2^{k+2+m-4} (-1)^{k+2+m}\cdot(\sqrt{-1})^{-2k-4-1} \cdot \mathfrak C_\mathrm{exc}(\f\otimes \Phi(\Theta))  \cdot \mathfrak c_{g_{\rm Eis}} }{\omega_K h_K \sqrt{D_K} \cdot \varepsilon(\f_\kappa)\cdot \pmb\xi_\lambda(\mathfrak N^{-1}) \cdot \mathcal L_p^{\rm Katz}(\widetilde{\Phi}_{\kappa\otimes\mu})}\,.
        \\
        \end{aligned}
        \end{align*}
        
    Using the fact that $\mathcal L_p^{\rm Katz}(\mathbf{N}_K)=-\frac{1}{2}(1-p^{-1})\log_p(u)$ (cf. \cite[\S10.4]{Katz76}; see also \cite[p. 90]{Gross1980Factorization} where $u$ corresponds to the inverse of Gross'\, $\overline{\alpha}$), the proof follows on passing to limit $\mu\to \kappa$ (so that $\widetilde{\Phi}_{\kappa\otimes\mu} \to \mathbf{N}_K$).
\end{proof}

\subsubsection{Reformulation of Conjecture~\ref{conj_main_6_plus_2} in terms of big Heegner points}
\label{subsubsec_222_2023_11_16_1152}

The goal of this subsection is to reformulate Conjecture~\ref{conj_main_6_plus_2} in terms of twisted Heegner points instead of the Beilinson--Kato elements, which is more suitable for the approach we describe in the present article to prove this conjecture. Note that it is implicit in our conjecture that we assume the validity of the hypothesis \eqref{item_deg6}.

\begin{conj}
\label{conj_main_6_plus_2_bis_compact}
Suppose that $\varepsilon(\hf)=-1=\varepsilon^{\rm bal}(\hf\otimes\g\otimes\g^c)$. We have the following factorization of $p$-adic $L$-functions:
\begin{equation}
\label{eqn_conj_main_6_plus_2_bis_compact}
     \cL_p^\hg(\hf\otimes\hg\otimes\hg^c)^2_{\vert_{\cW_2}} =  \mathscr{D}\cdot \cL_p^\Ad(\hf\otimes \Ad^0\hg) \cdot \mathcal L_p^\dagger(\f/K)^2\,.
\end{equation}
        Here, $\mathscr D\in {{\rm Frac}(\cR)}$ is a meromorphic function in 2 variables, with an explicit algebraicity property at crystalline specializations (cf. \eqref{eqn_2024_02_04_1756} below).
\end{conj}

Equivalently, in light of Corollary~\ref{SpecializationBDP}, we can rephrase Conjecture~\ref{conj_main_6_plus_2_bis_compact} as follows:

\begin{conj}
\label{conj_main_6_plus_2_bis}
Suppose that $\varepsilon(\hf)=-1=\varepsilon^{\rm bal}(\hf\otimes\g\otimes\g^c)$. Then,  
\begin{equation}
\label{eqn_conj_main_6_plus_2_bis}
     \cL_p^\hg(\hf\otimes\hg\otimes\hg^c)^2(\kappa,\lambda,\lambda) 
     =  \mathscr{D}(\kappa,\lambda)\cdot \cL_p^\Ad(\hf\otimes \Ad^0\hg)(\kappa,\lambda) \cdot \varepsilon(\pmb\xi_{\kappa,\p})^2\,p^{2-2s_\kappa}[L_{s_\kappa}:H_{p^{s_\kappa}}]^{-2} \cdot \log_{\omega_{\f_\kappa}}(Q_\kappa)^2\,,
\end{equation}
for all specializations $\kappa$ (resp. $\lambda$) of $\cR_\f$ (resp. of $\cR_\g$) of weight $2$ and non-trivial wild character $\psi_\kappa$ (resp. $\psi_\lambda$) of conductor $p^{s_\kappa}$ (resp. $p^{s_\lambda}$) for which $f_\kappa$ and $g_\lambda$ are newforms, where $\mathscr D\in {{\rm Frac}(\cR)}$ is  as in Conjecture~\ref{conj_main_6_plus_2_bis_compact}.
\end{conj}

\subsubsection{} 
\label{subsubsec_228_2024_03_06_1611}
Let us briefly explain that Conjecture~\ref{conj_main_6_plus_2_bis_compact}  is indeed equivalent to Conjecture~\ref{conj_main_6_plus_2}. We recall that  Conjecture~\ref{conj_main_6_plus_2} asserts that
\begin{equation}   
\cL_p^\hg(\hf\otimes\hg\otimes\hg^c)^2_{\vert_{\cW_2}}= \mathscr C\cdot \cL_p^\Ad(\hf\otimes \Ad^0\hg) \cdot {\rm Log}_{\omega_\f}({\rm BK}_{\f}^\dagger)
\end{equation}
 under the assumption that $\varepsilon(\hf)=-1=\varepsilon^{\rm bal}(\hf\otimes\g\otimes\g^c)$, where $\mathscr C\in {\rm Frac}(\cR)$ is a meromorphic function in 2 variables with an explicit algebraicity property at crystalline specializations $(\kappa,\lambda)$ (cf. \cite{BCS}, Theorem 8.11). 
Combining Proposition~\ref{propositionHidaBDP} with \cite[Proposition 8.9 ]{BCS}, we deduce that
    \begin{equation} \label{eqn_2024_03_02_2220}
        {\rm Log}_{\omega_\f}({\rm BK}_\f^\dagger(\kappa)) = \frac{\mathscr C_{\rm BDP}(\kappa)}{\mathscr C_{\rm Hida}(\kappa) } \frac{ \cL_p^\dagger(\f/K)(\kappa)^2}{\cL_p^{\rm Kit}(\f\otimes\epsilon_K)(\kappa, \rmw(\kappa)/2+1)}\,.
    \end{equation}
    Here, we recall from op. cit. that 
    \begin{equation}
    \label{eqn_2022_12_21_1139}
        \mathscr C_{\rm Hida} = \frac{\mathfrak c_{g_{\rm Eis}}\cdot \eta_{g_{\rm Eis}}( v_{g_{\rm Eis}}^-)}{ 2\, \omega_{g_{\rm Eis}}(v_{g_{\rm Eis}}^+)} \cdot \frac{\mathscr N_{\hf,c}^\dagger}{\mathscr N_{g_{\rm Eis},c}^\dagger} \cdot \mathcal C_1 
        =
        \frac{\mathfrak c_{g_{\rm Eis}}\cdot \eta_{g_{\rm Eis}}( v_{g_{\rm Eis}}^-)}{ 2\, \omega_{g_{\rm Eis}}(v_{g_{\rm Eis}}^+)} \cdot \frac{w_\f}{w_{g_{\rm Eis}} } \cdot \mathcal C_1 
    \end{equation}
and $\mathcal C_1\in \cR_\hf[\frac{1}{p}]^\times$ verifies
        \begin{equation}
        \label{eqn_2022_12_21_1140}
        \mathcal C_1(\kappa)=\frac{ \mathfrak C_{\rm exc}(\hf\otimes\hg_K)}{(-2\sqrt{-1})^{\wt(\kappa)+1} \cdot C_{\hf_\kappa}^+C_{\hf_\kappa}^-\cdot \mathcal E(\hf_\kappa^\circ,\Ad)}
        \end{equation}
        for all $\kappa\in \cW_\f^{\rm cl}$\,. Applying now the formula \eqref{eqn_2024_03_02_2220} with Conjecture \ref{conj_main_6_plus_2}, we conclude that
\begin{equation} 
\begin{aligned}
\cL_p^\hg(\hf\otimes\hg\otimes\hg^c)^2_{\vert_{\cW_2}} (\kappa,\lambda,\lambda) 
&= \mathscr D(\kappa,\lambda) \cdot \cL_p^\Ad(\hf\otimes \Ad^0\hg) (\kappa,\lambda) \cdot 
\mathcal L_p^\dagger(\f/K)^2
\end{aligned}
\end{equation}
where
\begin{equation}
    \label{eqn_2024_02_04_1756}
     \mathscr D(\kappa,\lambda) = \frac{ \mathscr C(\kappa,\lambda)\cdot \mathscr C_\mathrm{BDP}(\kappa)}{\mathscr C_{\rm Hida}(\kappa) \cdot \cL_p^{\rm Kit}(\f\otimes\epsilon_K)(\kappa, \rmw(\kappa)/2+1)}.
\end{equation}
   
Note that all the terms that one needs in order to compare $\mathscr D (\kappa,\lambda)$ with $\mathscr C(\kappa,\lambda)$ have been made explicit, and the crystalline specializations of $\mathscr C$ have been explicated in \cite[Theorem 8.11]{BCS}.

\section{First reduction step: Weight-2 specializations}
\label{sec_3_2023_09_08_1159}
Our goal in this section is to explain the first step towards the proof of Conjecture~\ref{conj_main_6_plus_2_bis}, which reduces its statement to a comparison of the specializations of diagonal cycles with Heegner points.

\subsection{Big diagonal cycles}
We recall our notation for the self-dual Galois twists $T_{\f}^{\dagger} := T_\f\otimes\bbchi_{\f}^{-\frac{1}{2}}$, and $T^{\dagger} := T\otimes \bbchi_{\f}^{-\frac{1}{2}}= T_\f^\dagger\widehat\otimes {\rm ad}(T_\g)$. We also recall the balanced Selmer group $H^1_{\rm bal}(\QQ,T^\dagger)$ determined by the balanced local conditions at $p$, given by the $G_{\QQ_p}$-stable submodule $F_{\rm bal}T^\dagger$ given as in \eqref{eqn_2024_07_09_1145}. More precisely, we put
$$H^1_{\rm bal}(\QQ,T^\dagger):=\ker\left(H^1(\QQ,T^\dagger)\xrightarrow{(\res_\ell)_\ell} \frac{H^1(\QQ_p,T^\dagger)}{{\rm im}\left(H^1(\QQ_p,F_{\rm bal}T^\dagger)\to H^1(\QQ_p,T^\dagger) \right)}\times \prod_{\ell\neq p} H^1(\QQ_\ell^{\rm ur},T^\dagger) \right)\,.$$

There is a canonical element   
$$\Delta^{\etale}(\f\otimes{\rm ad}(\g))\in H^1_{\rm bal}(\QQ,T^{\dagger})\,,$$
which we call the \emph{big diagonal cycle} and is constructed as in \cite{BSV,DarmonRotger}, which interpolates the Abel--Jacobi images of the Gross--Kudla--Schoen cycles in an appropriate sense (we will discuss its construction in \S\ref{subsec_big_diagonal} below). As we have explained in \S\ref{subsubsec_2023_10_02_1712}, we will primarily follow the construction in \cite{DarmonRotger}, as it fits better with our purposes, because it allows us to work with specializations with wild characters of arbitrarily high order.

\subsubsection{Reciprocity laws for big diagonal cycles} 
Let us put $F_{\g}T^\dagger:=T_\f^\dagger\widehat{\otimes}_{\ZZ_p}(F^+T_\g\otimes_{\cR_\g}T_{\g^c})(\bbchi_{\g}^{-1}\chi_\cyc^{-1})$. Observe then that
$$F_{\g}T\cap F_{\rm bal}T=\left(F^+T_\f^\dagger\widehat{\otimes}_{\ZZ_p}(F^+T_\g\otimes_{\cR_\g}T_{\g^c})+T_\f^\dagger\widehat{\otimes}_{\ZZ_p}(F^+T_\g\otimes_{\cR_\g}F^+T_{\g^c})\right)(\bbchi_{\g}^{-1}\chi_\cyc^{-1})\,,$$
and that
$$F_{\rm bal}T^\dagger/F_{\g}T^\dagger\cap F_{\rm bal}T^\dagger\xrightarrow{\,\,\sim\,\,} F^+T_\f^\dagger\widehat{\otimes}_{\ZZ_p}(F^-T_\g\otimes_{\cR_\g} F^+T_{\g^c})(  \bbchi_{\g}^{-1}\chi_\cyc^{-1})\,.$$
As a result, we have a natural morphism
\begin{align}
    \label{eqn_2023_10_02_1728}
    \begin{aligned}
        \res_p^{(\g)}\,:\, H^1_{\rm bal}(\QQ,T^\dagger)\xrightarrow{\res_p} H^1(&\QQ_p,F_{\rm bal}T^\dagger) \\
        &\lra H^1(\QQ_p,F^+T_\f^\dagger\widehat{\otimes}_{\ZZ_p}(F^-T_\g\otimes_{\cR_\g} F^+T_{\g^c})(  \bbchi_{\g}^{-1}\chi_\cyc^{-1}))
    \end{aligned}
\end{align}
landing in the domain $H^1(\QQ_p,F^+T_\f^\dagger\widehat{\otimes}_{\ZZ_p}(F^-T_\g\otimes_{\cR_\g} F^+T_{\g^c})(  \bbchi_{\g}^{-1}\chi_\cyc^{-1}))$ of the Perrin-Riou map ${\rm Log}_{\omega^{(\g)}}^{(\g),\dagger}$, that takes values in $\frac{1}{H_\g}\cR$; see Equation~\eqref{eqn_2024_07_02_1246}.

\begin{theorem}[Bertolini--Seveso--Venerucci, Darmon--Rotger]
\label{thm_reciprocity_big_diag}   
We have 
\begin{equation}
\label{eqn_2023_10_03_1126}
    {\rm Log}_{\omega^{(\g)}}^{(\g),\dagger}\left(\res_p^{(\g)}\left(\Delta^{\etale}(\f\otimes{\rm ad}(\g))\right)\right)=\cL_p^\hg(\hf\otimes\hg\otimes\hg^c)_{\vert_{\cW_2}}\,.
\end{equation}
\end{theorem}

\subsubsection{} Let $\kappa$ (resp. $\lambda$) be a specialization of $\cR_\f$ (resp. of $\cR_\g$) of weight $2$ and wild character $\psi_\kappa$ (resp. $\psi_\lambda$). We will explicate the specialization of both sides of \eqref{eqn_conj_main_6_plus_2_bis_compact} when specialized at such $(\kappa,\lambda)$. 

We first recall from the commutative diagram \eqref{eqn_2023_09_27_1233} that 
$${\rm Log}_{\omega^{(\g)}}^{(\g),\dagger}\left(\res_p^{(\g)}\left(\Delta^{\etale}(\f\otimes{\rm ad}(\g))\right)\right)={\rm Log}_{\omega_{\f}}^{\dagger}\left(\Delta^{\rm (tr)}_p(\f\otimes{\rm ad}(\g))\right)$$
where we have put\footnote{The superscript ``${\rm tr}$'' in our notation is to remind the reader that we are considering the image of the family of diagonal cycles $\Delta^{\etale}(\f\otimes{\rm ad}(\g))$ under the map that is induced from ${\rm tr}$, given as in \eqref{eqn_2023_09_26_1321}.} 
$$\Delta^{\rm (tr)}_p(\f\otimes{\rm ad}(\g)):=({\rm id} \otimes {\rm tr})\circ \res_p\left(\Delta^{\etale}(\f\otimes{\rm ad}(\g))\right)\,.$$

In view of Theorem~\ref{thm_reciprocity_big_diag}, Conjecture~\ref{conj_main_6_plus_2_bis_compact} is therefore equivalent to the assertion that
\begin{equation}
    \label{eqn_conj_main_6_plus_2_bis_bis}
    {\rm Log}_{\omega_{\f}}^{\dagger}\left(\Delta^{\rm (tr)}_p(\f\otimes{\rm ad}(\g))\right)^2\,\stackrel{?}{=}  \mathscr{D}\cdot \cL_p^\Ad(\hf\otimes \Ad^0\hg) \cdot \mathcal L_p^\dagger(\f/K)^2\,,
\end{equation}
where $\mathscr{D}\in {\rm Frac}(\cR)$ is as in the statement of Conjecture \ref{conj_main_6_plus_2_bis_compact}.

\begin{lemma}
    \label{lemma_first_reduction}
    Conjecture~\ref{conj_main_6_plus_2_bis} follows if \eqref{item_deg6} holds and we have
    \begin{equation}
        \label{eqn_2023_10_03_1202}
        \begin{aligned}
            \log_{\omega_{\f_\kappa}}\left(\Delta^{\rm (tr)}_p(\f\otimes{\rm ad}(\g))_{\vert_{\kappa,\lambda}}\right)^2&
            =
            \mathscr{D}(\kappa,\lambda) \times \mathfrak{g}(\psi_{\kappa}^{\frac{1}{2}})\cdot C_{\f_\kappa}^- \cdot \frac{\Lambda(\hf_\kappa\otimes\Ad^0\hg_\lambda,\psi_\kappa^{-\frac{1}{2}},1)}{a_p(\f_\kappa)^{-s_\kappa}\,\Omega_{\hf_\kappa}^- \, \Omega_{\hg_\lambda}^{\rm ad}} \\
            &\quad\quad \times  p^{2-2s_\kappa}[L_{s_\kappa}:H_{p^{s_\kappa}}]^{-2} a_p(\f_\kappa)^{-2s_\kappa}\cdot\log_{\omega_{\f_\kappa}}(Q_\kappa)^2
        \end{aligned}
    \end{equation}
    for all specializations $\kappa$ (resp. $\lambda$) of $\cR_\f$ (resp. of $\cR_\g$) of weight $2$ and non-trivial wild character $\psi_\kappa$ (resp. $\psi_\lambda$) of conductor $p^{s_\kappa}$ (resp. $p^{s_\lambda}$), for which $f_\kappa$ and $g_\lambda$ are newforms. Here:
    \begin{itemize}
        \item $\mathscr D$ is as in the statement of Conjecture \ref{conj_main_6_plus_2_bis_compact}, and its specializations $\mathscr D(\kappa, \lambda)$ at arithmetic points are given by \eqref{eqn_2024_02_04_1756}. 
        \item $\Lambda(\hf_\kappa\otimes\Ad^0\hg_\lambda,\psi_\kappa^{-\frac{1}{2}},s)$ is the twisted complete  $L$-series associated to $\hf_\kappa\otimes\Ad^0\hg_\lambda$, cf. \cite[\S3.4.2]{BCS}.
        \item $\Omega_{\hg_\lambda}^{\rm ad}:=\Omega_{\hg_\lambda}\Omega_{\overline{\hg}_\lambda}$, where $$\Omega_{\hg_\lambda}={8\sqrt{-1}\,a_p(g_\lambda)^{-2s_\lambda}\mathfrak{g}(\psi_\lambda) || g_\lambda||^2}/{\mathfrak{c}_\g(\lambda)}\,.$$ 
        Here, $\mathfrak{g}(\psi_\lambda)$ is the Gauss sum, $|| g_\lambda||^2$ is the Petersson-norm of the newform $g_\lambda$ (cf. \cite{Hsieh}, \S1.4), and $\mathfrak{c}_\g(\lambda)$ is the congruence number of $g_\lambda$ (cf. \cite[(0.3)]{Hida81}, see also \cite{BCS}, \S3.2.1).
        \item $\Omega_{\hf_\kappa}^-$ (resp. $C_{\f_\kappa}^-$) is the complex (resp. $p$-adic) period that appears in the interpolation formula of the Mazur--Kitagawa $p$-adic  $L$-function (cf. \cite{BCS}, Theorem 3.1). 
        \item The fields $L_{s_\kappa}$ and $H_{p^{s_\kappa}}$ are given as in \S\ref{subsubsec_2024_02_07_1331} and \S\ref{subsubsec_2024_07_09_0945}.
    \end{itemize}
\end{lemma}

\begin{proof}
    Note that specializations $(\kappa,\lambda)$ as in the statement of our lemma are dense in ${\rm Spec}(\cR)(\overline{\QQ}_p)$. Therefore, it suffices to verify that the specialization of \eqref{eqn_conj_main_6_plus_2_bis_bis} to such $(\kappa,\lambda)$ amounts to the asserted equality \eqref{eqn_2023_10_03_1202}. This is an immediate consequence of the following interpolative properties (where the second one is conjectural, see however \cite[Theorem 3.7]{BCS} for a partial result):
    \begin{itemize}
\item[i)] $\omega_\f$ specializes to $\omega_{\f_\kappa}$, and ${\rm Log}_{\omega_{\f}}^{\dagger}$ to $a_p(\f_\kappa)^{s_\kappa}\varepsilon(\psi_{\kappa}^{{1}/{2}})\log_{\omega_{\f_\kappa}}$. Moreover, $\psi_{\kappa}^{{1}/{2}}\pmb\xi_{\kappa,\p}^{-1}$ is unramified at $\p$. Here, the universal anticyclotomic character $\pmb\xi$ is as in \S\ref{BDP_p-adic_L-function} (and $\pmb\xi_{\kappa,\p}$ is the $\p$-component of its specialization at $\kappa$), and our notation is borrowed from \cite{CastellapadicvariationofHeegnerpoints} where it was introduced in  Definition 2.8 of op. cit.
        \item[ii)] Conjecturally (cf. \cite[Conjecture 3]{BCS}, extended in line with the Coates--Perrin-Riou formalism):
        $$\cL_p^\Ad(\hf\otimes \Ad^0\hg)(\kappa,\lambda)=\mathfrak{g}(\psi_{\kappa}^{\frac{1}{2}})\cdot C_{\f_\kappa}^- \cdot \mathcal E^\Ad(\hf_\kappa\otimes\Ad^0\hg_\lambda, 1) \cdot  \frac{\Lambda(\hf_\kappa\otimes\Ad^0\hg_\lambda,1)}{a_p(\f_\kappa)^{-s_\kappa}\,\Omega_{\hf_\kappa}^- \, \Omega_{\hg_\lambda}^{\rm ad}} \,,$$
        where $\mathcal E^\Ad(\hf_\kappa\otimes\Ad^0\hg_\lambda, 1)=1$ since $F^\pm T_{\f_\kappa}^\dagger:=T_{\f_\kappa}^\pm(\psi_\kappa^{-\frac{1}{2}})$ is ramified as the wild character $\psi_\kappa$ is non-trivial.
    \end{itemize}
\end{proof}

\section{Review on Gross--Kudla--Schoen cycles and big diagonal cycles}
\label{subsec_big_diagonal}

Let $X$ be a smooth, projective, and connected curve over a number field $K$. The Chow group $\CH^2(X^3)$ of codimension two cycles on $X^3$ admits a filtration (see \cite[\S 3.1]{YZZ12} for details)
\[
\CH^2(X^3) \supset \CH^{2,1}(X^3) \supset \CH^{2,2}(X^3) \supset \CH^{2,3}(X^3),
\]
where $\CH^{2,1}(X^3)$ ($=\CH^2_0(X^3)$, the subgroup of $\CH^{2}(X^3)$ consisting of homologically trivial cycles) is defined as the kernel of the cycle class map $\CH^2(X^3) \to H^4(\overline{X}^{3},\Q_{\ell})$, where $\overline {X}^3 := X^3_{\overline K}$ and $\ell$ is an auxiliary prime. Using the comparison isomorphisms between \'etale and singular cohomology over $\C$, one shows that $\CH^{2,1}(X^3)$ is independent of the choice of $\ell$.

The filtration above can be upgraded to a decomposition of $\CH^2({X}^3)$ (cf. Equation (3.1.1) in op. cit.). The projection of the diagonal cycle $X_{123}\subset X^3$ onto $\CH^{2,1}(X^3)$ is precisely the Gross--Kudla--Schoen cycle $\Delta_{\mathrm{GKS}}$ that we discuss in \S\ref{subsubsec_421_2023_11_02_830} below.

\subsection{Gross--Kudla--Schoen modified diagonal cycles} 
\label{subsubsec_421_2023_11_02_830}

Let $X = X_1(M)$ denote the modular curve of $\Gamma_1(M)$-level. For $i=1,2,3$, write $X_i$ for a copy of $X$, and $X^3 = X_1 \times X_2 \times X_3$, so that we label each factor in the triple self-product. Fix once and for all a base point $o \in X(\Q)$, and let $\iota_o: X \to X$ be the constant morphism with image $\{o\}$. If $I$ is a subset of $\{1,2,3\}$, we denote by $\iota_I: X \to X^3$ the morphism characterized as $\mathrm{id}_X$ on the factors $X_i$ with $i\in I$ and as $\iota_o$ on the factors $X_i$ with $i\not\in I$. And to ease notation, we write $\iota_{12}$ for $\iota_{\{1,2\}}$, $\iota_3$ for $\iota_{\{3\}}$, and so on. We also write $X_I$ for the image of $\iota_I$ inside $X^3$, and simplify notation in the indices in the same manner as for $\iota_I$.

With this, the {\em Gross--Kudla--Schoen modified diagonal cycle} in $X^3$ is defined to be the cycle (or rather, class)
\[
\Delta_{\mathrm{GKS}} = X_{123} - X_{12} - X_{13} - X_{23} + X_1 + X_2 + X_3 \in \CH^2(X^3)(\Q).
\]

If the fixed base point $o \in X(K)$ is only $K$-rational, for some extension $K/\Q$, the previously defined cycle will belong to $\CH^2(X^3)(K)$ instead (note that each of the cycles involved in the definition will be $K$-rational, if so is $o$).

\begin{lemma}
The cycle $\Delta_{\mathrm{GKS}}$ is cohomologically trivial, that is $\Delta_{\mathrm{GKS}} \in \CH^2(X^3)_0(\Q)$.
\end{lemma} 
\begin{proof}
    This is \cite[Proposition 3.1]{GrossSchoen-cycle}.
\end{proof}

The work of Darmon--Rotger~\cite{DarmonRotger} and Bertolini--Seveso--Venerucci~\cite{BSV} interpolates $p$-adically the images of diagonal cycles under $p$-adic Abel--Jacobi maps. This requires a further modification of the diagonal appropriate to this purpose, which we summarize in the next section.

\subsection{Big diagonal cycles}
We present an overview of the construction of modified diagonal cycles as in \cite{DarmonRotger}. We will follow Section 2 in op. cit., and stick to the notation therein (unless we explicitly indicate otherwise). We claim no originality in this subsection: we very marginally expand on a few constructions in \cite{DarmonRotger}, but all the objects we consider here are already present in op. cit. One exception is \eqref{eqn_2023_11_08_1351}, where we define $\Delta^{\rm ord}_r$ on applying the ordinary projector on the level of cycles (in the expense of passing to $p$-adic coefficients) and interpolate these in place of those denoted by $\Delta^{\circ}_r$ (cf. Definition~\ref{defn_2023_11_17_1544}). We remark that this small deviation has no effect on the results we borrow from \cite{DarmonRotger}, as the authors' main results also concern ordinary cohomology.

Let us fix a positive integer $r$ and denote by $\varpi_1$ the degeneracy map 
$$X_1(Np^r)\lra X_1(Np^{r-1})\,,\qquad (A,P)\longmapsto (A,pP)\,,$$
where $(A,P)\in X_1(Np^r)$ is a pair representing an isomorphism class of an elliptic curve $A$ with a $\Gamma_1(Np^r)$-level structure (so $P$ is a point on $A$ of order $Np^r$). We also put 
$$\omega_1^r\,:\, X_1(Np^r) \,\xrightarrow{\,\,\overbrace{\varpi_1\circ\,\cdots\,\circ\varpi_1}^{r\, \textup{times}}\,\,}\, X_1(N)\,.$$
We define the cycle 
$$\Delta_r:=X_{123}\times_{X_1(N)^3}X_1(Np^r)^{3}$$ 
where the fibre product is respect to the maps $X_1(Np^r)^{3}\xrightarrow{(\varpi_1^r)^3} X_1(N)$ and the natural injection $X_{123}\hookrightarrow X_1(N)^{3}$. The cycle $\Delta_r$ therefore fits in the Cartesian diagram
$$\xymatrix{
\Delta_r\ar@{^{(}->}[r]\ar@{->>}[d] & X_1(Np^r)^{3}\ar@{->>}[d]^{(\varpi_1^r)^3}\\
X_{123}\ar@{^{(}->}[r]&X_1(N)^{3}\,.
}$$
It turns out that $\Delta_r$ is geometrically reducible. For each 
$$[d_1,d_2,d_3]\in (\ZZ/p^r\ZZ)^{\times}\times (\ZZ/p^r\ZZ)^{\times}\times (\ZZ/p^r\ZZ)^{\times}=:\widetilde{G}_r\,,$$
let us denote by 
$\Delta_r[d_1,d_2,d_3]\subset X_1(Np^r)^3$ the geometrically irreducible component defined over $\QQ(\mu_{p^r})$, which is the schematic closure of the locus of points $\left((A,P_1),(A,P_2), (A,P_3)\right)$ that satisfy 
$$\langle P_1,P_2\rangle=\zeta_{p^r}^{d_3}\,,\quad \langle P_2,P_3\rangle=\zeta_{p^r}^{d_1}\,,\quad \langle P_3,P_1\rangle=\zeta_{p^r}^{d_2}\,,$$
where $\zeta_{p^{r}}\in \mu_{p^r}$ is a fixed $p^r$-th root of unity. The diamond action of $\widetilde{G}_r$ and $\Gal(\QQ(\mu_{p^r})/\QQ)\simeq (\ZZ/p^r\ZZ)^\times$ on the collection of cycles
$$\{\Delta_r[d_1,d_2,d_3]\,:\,[d_1,d_2,d_3]\in \widetilde{G}_r\}$$ 
can be described as follows: For any $\langle a_1,a_2,a_3\rangle \in \widetilde{G}_r$ and $\sigma_m \in \Gal(\QQ(\mu_{p^r})/\QQ)$, which is characterized by the property that $\sigma_m(\zeta_{p^r})=\zeta_{p^r}^m$, we have
\begin{equation}
    \label{eqn_2023_11_01_1241}
   \langle a_1,a_2,a_3\rangle \,\Delta_r[d_1,d_2,d_3]=\Delta_r[a_2a_3 d_1,a_1a_3d_2,a_1a_2d_3]
\end{equation}
\begin{equation}
    \label{eqn_2023_11_01_1242}
    \sigma_m\, \Delta_r[d_1,d_2,d_3]=\Delta_r[md_1,md_2,md_3]\,.
\end{equation}
In particular, when $r\geq 1$ and $m$ is a quadratic residue modulo $p^r$ (e.g. when $\sigma_m$ acts trivially on $\mu_p$), we have  
\begin{equation}
    \label{eqn_2023_11_01_12417}
   \sigma_m\, \Delta_r[d_1,d_2,d_3]=\langle m,m,m\rangle^{\frac{1}{2}}\,\Delta_r[d_1,d_2,d_3]\,.
\end{equation}

\subsubsection{$p$-adic cycle class maps}
Let us put $\overline{X}_1(Np^r)^3:={X}_1(Np^r)^3\times_\QQ \overline{\QQ}$ and denote by
$${\rm CH}^2_0(X_1(Np^r)^3; F) := \ker \left( {\rm CH}^2(X_1(Np^r)^3; F) \xrightarrow{{\rm cl}_0} H^4_{\etale} \left(\overline{X}_1(Np^r)^3 , \ZZ_p(2)\right)^{G_F}\right)$$ 
the group of null-homologous algebraic cycles defined over a number field $F$ (which is the kernel of the cycle class map ${\rm cl}_0$). We let ${\rm CH}^2_0(X_1(Np^r)^3; F)_{\ZZ_p}$ denote its $p$-adic completion. 

Since the target of ${\rm cl}_0$ is $p$-adically complete, ${\rm cl}_0$ factors through the $p$-adic completion ${\rm CH}^2(X_1(Np^r)^3; F)_{\ZZ_p}$ of ${\rm CH}^2(X_1(Np^r)^3; F)$, and we denote the induced map on  ${\rm CH}^2(X_1(Np^r)^3; F)_{\ZZ_p}$ still by   ${\rm cl}_0$. Moreover,  using K\"{u}nneth decomposition formula, together with the fact that integral cohomology of smooth projective curves is torsion-free (cf. \cite{DarmonRotgerDiagonal2}, Proposition 1.4), we see that the $\ZZ$-module $H^4_{\etale} \left(\overline{X}_1(Np^r)^3, \ZZ_p(2)\right)$ is torsion-free\footnote{Note that $\ZZ$ acts on $H^4_{\etale} \left(\overline{X}_1(Np^r)^3, \ZZ_p(2)\right)$ via its image under the canonical injection $\ZZ\hookrightarrow \ZZ_p$.}. Then it follows that ${\rm im}({\rm cl}_0)$ is also torsion-free, hence it is flat. It then follows from \cite[\href{https://stacks.math.columbia.edu/tag/0315}{Lemma 0315}]{stacks-project} that 
$${\rm CH}^2_0(X_1(Np^r)^3; F)_{\ZZ_p}=\ker \left( {\rm CH}^2(X_1(Np^r)^3; F)_{\ZZ_p} \xrightarrow{{\rm cl}_0} H^4_{\etale} \left(\overline{X}_1(Np^r)^3 , \ZZ_p(2)\right)\right)\,.$$

Note that $\TT_{N p^r}$ acts on ${\rm CH}^2(X_1(Np^r)^3; F)_{\ZZ_p}$, and as a result, so does $e_{\rm ord}':=\lim_n (U_p')^{n!}$. 
This allows us to define\footnote{The appearance of $W_{Np^r}$ should be compared to the discussion in the portion of \cite{BSV} that lies between Proposition 3.2 and Remark 3.3.} 
\begin{equation}
    \label{eqn_2023_11_08_1351}
    \Delta_r^{\rm ord}[d_1,d_2,d_3]:=(e_{\rm ord}'W_{Np^r})^{\otimes 3}\Delta_r[d_1,d_2,d_3]\in {\rm CH}^2(X_1(Np^r)^3;\QQ(\mu_{p^r}))_{\ZZ_p}\,.
\end{equation}
This definition is made so that we have:

\begin{lemma}
    \label{DR_Lemma_2_5}
    \item[i)] We have $\Delta_r^{\rm ord}[d_1,d_2,d_3]\in {\rm CH}^2_0(X_1(Np^r)^3;\QQ(\mu_{p^r}))_{\ZZ_p}$. 
    \item[ii)] For any integer $m$ coprime to $p$, we have
    \begin{equation}
        \label{eqn_2023_11_11_1700}
\sigma_m\Delta_r^{\rm ord}[d_1,d_2,d_3]=\Delta_r^{\rm ord}[m^{-1}d_1,m^{-1}d_2,m^{-1}d_3]\,.
    \end{equation}
\end{lemma}

\begin{proof}
The argument we present for the first claim closely follows the proof of \cite[Lemma 2.5]{DarmonRotger}. The correspondence $U_p'$ acts by multiplication by $p$ on 
$H^2_{\etale}(\overline{X}_1(Np^r)^3,\ZZ_p(2))$, therefore Hida's idempotent $e_{\rm ord}'$ annihilates this module. As a result, $e_{\rm ord}'$ annihilates all the terms that appear in the K\"unneth decomposition of $H^4_{\etale}(\overline{X}_1(Np^r)^3,\ZZ_p(2))$, therefore also $H^4_{\etale}(\overline{X}_1(Np^r)^3,\ZZ_p(2))$ itself, which is the target of the Hecke equivariant cycle class map ${\rm cl}_0$. This concludes the proof of the first assertion. For the second, we note that
\begin{align*}
    \sigma_m\Delta_r^{\rm ord}[d_1,d_2,d_3]&=\sigma_m(e_{\rm ord}'W_{Np^r})^{\otimes 3}\Delta_r[d_1,d_2,d_3]=(e_{\rm ord}'W_{Np^r})^{\otimes 3}\sigma_{m^{-1}}\Delta_r[d_1,d_2,d_3]\\
&=(e_{\rm ord}'W_{Np^r})^{\otimes 3}\Delta_r[m^{-1}d_1,m^{-1}d_2,m^{-1}d_3]=\Delta_r^{\rm ord}[m^{-1}d_1,m^{-1}d_2,m^{-1}d_3]
\end{align*}
where the second equality can be verified using the moduli description of the Atkin--Lehner operator and the action of $\Gal(\QQ(\mu_{p^r}/\QQ)$, whereas the third is \eqref{eqn_2023_11_01_1242}.
\end{proof}

The modification of diagonal cycles as in \eqref{eqn_2023_11_08_1351} is especially useful for the purpose of interpolating their images under the $p$-adic Abel--Jacobi map (see \S\ref{subsubsec_426_2023_11_17_1541}), whereas the following modification (which is copied from \cite{DarmonRotger}, \S2.2) will be used for archimedean aspects (see \S\ref{sec_ChowHeegandGKS} below):
\begin{defn}
    \label{defn_2023_11_17_1544}
    Let us fix a prime $q$ that is coprime to $Np$, and let us set
    $$\Delta^\circ_{r}[d_1,d_2,d_3]:=\left(T_q'-(q+1)\right)^{\otimes 3}W_{Np^r}^{\otimes 3}\Delta_r[d_1,d_2,d_3]\in {\rm CH}^2(X_1(Np^r)^3;\QQ(\mu_{p^r}))_{\ZZ_p}$$
    for any positive integer $r$ and $(d_1,d_2,d_3)\in \widetilde{G}r$.
\end{defn}
As explained in \cite[Lemma~2.5]{DarmonRotger}, the operator $T_q'-(q+1)$ annihilates the image of the cycle class  map ${\rm cl}_0$. As a result,
$$\Delta^\circ_{r}[d_1,d_2,d_3]\in {\rm CH}^2_0(X_1(Np^r)^3;\QQ(\mu_{p^r}))\,.$$

\subsubsection{\'Etale Abel--Jacobi map}
\label{subsubsec_2023_11_21_1550}
We have a map 
$${\rm cl}_1\,:\, {\rm CH}^2_0(X_1(Np^r)^3; F)\lra H^1(F,  H^3_{\etale} (\overline{X}_1(Np^r)^3, \ZZ_p(2)))\,, $$
which is often referred to as the \'etale Abel--Jacobi map (and denoted by $\text{AJ}_{\etale}$ in \cite{DarmonRotger}). 

Since the target of ${\rm cl}_1$ is $p$-adically complete, ${\rm cl}_1$ factors through the $p$-adic completion ${\rm CH}^2_0(X_1(Np^r)^3; F)_{\ZZ_p}$ of ${\rm CH}^2_0(X_1(Np^r)^3; F)$, and we denote the induced map on  ${\rm CH}^2_0(X_1(Np^r)^3; F)_{\ZZ_p}$ still by   ${\rm cl}_1$.

\begin{remark}
    \label{remark_2023_11_21_1623}
    Let $\p\subset \cO_F$ denote any prime of $F$. The constructions above can be carried out also with $F$ replaced by $F_\p$, to yield the \'etale Abel--Jacobi map
    $${\rm cl}_1\,:\, {\rm CH}^2_0(X_1(Np^r)^3; F_\p)\lra H^1(F_\p,  H^3_{\etale} (\overline{X}_1(Np^r)^3, \ZZ_p(2)))$$
    over the base field $F_\p$. We then have the following commutative diagram, where the arrow on the left is the natural morphism induced from the inclusion $F\hookrightarrow F_\p$:
    $$
    \xymatrix{
    {\rm CH}^2_0(X_1(Np^r)^3; F)\ar[r]^-{{\rm cl}_1}\ar[d]& H^1(F,  H^3_{\etale} (\overline{X}_1(Np^r)^3, \ZZ_p(2)))\ar[d]^-{\res_\p}\\
    {\rm CH}^2_0(X_1(Np^r)^3; F_\p)\ar[r]_-{{\rm cl}_1}& H^1(F_\p,  H^3_{\etale} (\overline{X}_1(Np^r)^3, \ZZ_p(2)))\,.
    } 
    $$
\end{remark}

Thanks to Lemma~\ref{DR_Lemma_2_5}, we may put
\begin{equation}
\label{eqn_2023_11_08_1512}
   \Delta_r^{\etale}[d_1,d_2,d_3] := {\rm cl}_1(\Delta_r^{\ord}[d_1,d_2,d_3])\in H^1(\QQ(\mu_{p^r}),  H^3_{\etale} (\overline{X}_1(Np^r)^3, \ZZ_p(2)))\,.
\end{equation} 
By a slight abuse of notation, we will also denote by $\Delta_r^{\etale}[d_1,d_2,d_3] $ the image of $\Delta_r^{\ord}[d_1,d_2,d_3]$ under the compositum
$${\rm CH}^2_0(X_1(Np^r)^3;\QQ(\mu_{p^r}))\xrightarrow{{\rm cl}_1} H^1(\QQ(\mu_{p^r}),  H^3_{\etale} (\overline{X}_1(Np^r)^3, \ZZ_p(2)))\xrightarrow{\,\,\texttt{K}\,\,} H^1(\QQ(\mu_{p^r}),  H^1_{\etale} (\overline{X}_1(Np^r), \ZZ_p)^{\otimes 3})(2))\,,$$
where the final projection $\texttt{K}$ is induced from K\"unneth decomposition. 

\subsubsection{Field of definition}
Let $G_r$ denote the $p$-Sylow subgroup of $\widetilde{G}_r$. In order to descend the field of definition of the cohomology class $ \Delta_r^{\etale}[d_1,d_2,d_3]$, we define (following \cite[\S2.2]{DarmonRotger}; see especially Equation 2.9 and the proof of Lemma 2.6 in op. cit.) 
\begin{equation}
    \label{eqn_2023_11_10_1126}
    \Delta_r^{\etale}[[a,b,c]]:=\sum_{\langle d_1,d_2,d_3\rangle \in G_r}\Delta_r^{\etale}[d_2d_3a,d_1d_3b,d_1d_2c]\,\langle d_1,d_2,d_3\rangle\in H^1(\QQ(\mu_{p^r}),  H^1_{\etale} (\overline{X}_1(Np^r), \ZZ_p)^{\otimes 3}(2))
\end{equation}
for each $(a,b,c)\in(\FF_p^\times)^3$, which we also regard in the formula above as an element of $\widetilde{G}_r$ via Teichm\"uller lift. In \eqref{eqn_2023_11_10_1126}, we regard each $\Delta_r^{\etale}[e_1,e_2,e_3]$ as a cocycle taking values in 
$${\rm Hom}_{\ZZ_p}\left(H^1_{\etale} (\overline{X}_1(Np^r), \ZZ_p)^{\otimes 3}(1),\ZZ_p\right)\xrightarrow{\,\,\sim\,\,} H^1_{\etale} (\overline{X}_1(Np^r), \ZZ_p)^{\otimes 3}(2)$$
(where the isomorphism follows from Poincar\'e duality), and $\Delta_r^{\etale}[d_2d_3a,d_1d_3b,d_1d_2c]\,\langle d_1,d_2,d_3\rangle$ is a cocycle with values in the module 
$${\rm Hom}_{\ZZ_p}\left(H^1_{\etale} (\overline{X}_1(Np^r), \ZZ_p)^{\otimes 3}(1),\ZZ_p[G_r]\right)\,.$$
It is easy to see that the expression
$$\sum_{\langle d_1,d_2,d_3\rangle \in G_r}\Delta_r^{\etale}[d_2d_3a,d_1d_3b,d_1d_2c]\,\langle d_1,d_2,d_3\rangle$$
on the right of the formula \eqref{eqn_2023_11_10_1126} belongs to
\begin{align*}
    {\rm Hom}_{\ZZ_p[G_r]}\left(H^1_{\etale} (\overline{X}_1(Np^r), \ZZ_p)^{\otimes 3}(1),\ZZ_p[G_r]\right)&\xrightarrow[\texttt{G}]{\,\,\sim\,\,} {\rm Hom}_{\ZZ_p}\left(H^1_{\etale} (\overline{X}_1(Np^r), \ZZ_p)^{\otimes 3}(1),\ZZ_p\right)\\
    &\qquad\qquad\qquad\qquad \xrightarrow{\,\,\sim\,\,} H^1_{\etale} (\overline{X}_1(Np^r), \ZZ_p)^{\otimes 3}(2),
\end{align*}
where the isomorphism $\texttt{G}$ and its inverse $\texttt{G}^{-1}$ are given by 
$$\psi(\bullet)=\sum_{g\in G_r} \psi_g(\bullet)g \mapsto \psi_e(\bullet)\,,\quad \sum_{g\in G_r}\phi(g^{-1}\bullet)g \mapsfrom \phi(\bullet)\,. $$
\begin{remark}
\label{remark_2023_11_13_1509}
 \item[i)]   The formation of the cohomology class $ \Delta_r^{\etale}[[a,b,c]]$ can be recast as follows: It is the image of the class $\Delta_r^{\etale}[a,b,c]$ under the chain of natural isomorphisms induced from
    \begin{align*}
      H^1_{\textup{\'et}} (\overline{X}_1(Np^r), \ZZ_p)^{\otimes 3}(2)\stackrel{\sim}{\lra} & \,{\rm Hom}_{\ZZ_p}\left(H^1_{\textup{\'et}} (\overline{X}_1(Np^r), \ZZ_p)^{\otimes 3}(1),\ZZ_p\right)\\
      &\xrightarrow[\sim]{\textup{\texttt{G}}^{-1}}  {\rm Hom}_{\ZZ_p[G_r]}\left(H^1_{\textup{\'et}} (\overline{X}_1(Np^r), \ZZ_p)^{\otimes 3}(1),\ZZ_p[G_r]\right) \stackrel{\sim}{\lra} H^1_{\textup{\'et}} (\overline{X}_1(Np^r), \ZZ_p)^{\otimes 3}(2)\,.
    \end{align*}
  \item[ii)]  To ease notation, let us put (only in this remark) ${}_r\mathscr{S}:= H^1_{\textup{\'et}} (\overline{X}_1(Np^r), \ZZ_p)^{\otimes 3}(1)$. We then have the following diagram where all squares are commutative, and which extends the chain of isomorphisms above:
    \begin{align}
    \label{eqn_2023_11_13_1224}
        \begin{aligned}\xymatrix{
            {}_{r+1}\mathscr{S}(1) \ar[r]^-{\sim} \ar[d]_{\varpi_{1,*}^{\otimes 3}}& {\rm Hom}_{\ZZ_p}({}_{r+1}\mathscr{S},\ZZ_p) \ar[d]_{(\varpi_{1}^{*\otimes 3})'}\ar[r]_-{\sim}^-{\textup{\texttt{G}}^{-1}} &{\rm Hom}_{\ZZ_p[G_{r+1}]}({}_{r+1}\mathscr{S},\ZZ_p[G_{r+1}]) \ar[r]^-{\sim}\ar[d] & {}_{r+1}\mathscr{S}(1)\ar[d]_{\varpi_{1,*}^{\otimes 3}}\\
            {}_{r}\mathscr{S}(1) \ar[r]^-{\sim} & {\rm Hom}_{\ZZ_p}({}_{r}\mathscr{S},\ZZ_p) \ar[r]^-{\sim}_-{\textup{\texttt{G}}^{-1}} &{\rm Hom}_{\ZZ_p[G_{r}]}({}_{r}\mathscr{S},\ZZ_p[G_{r}]) \ar[r]^-{\sim}& {}_{r}\mathscr{S}(1)
            }
        \end{aligned}
    \end{align}
    where $(\varpi_{1}^{*\otimes 3})^{\prime}$ is the pullback of ${}_{r}\mathscr{S}\xrightarrow{\varpi_{1}^{*,\otimes 3}}{}_{r+1}\mathscr{S}$, and the third vertical arrow is given by 
    \begin{align}
        \label{defn_3rd_vertical_map}
        \begin{aligned}
            {\rm Hom}_{\ZZ_p[G_{r+1}]}({}_{r+1}\mathscr{S},\ZZ_p[G_{r+1}]) \,&\,\lra\, {\rm Hom}_{\ZZ_p[G_{r}]}({}_{r}\mathscr{S},\ZZ_p[G_{r}])\\
            \psi=\sum_{g\in {G}_{r+1}}\psi_g\cdot g \,&\,\longmapsto\, \frac{1}{p} \sum_{g'\in {G}_{r}}\sum_{\substack{g \in G_{r+1}\\
            {\rm pr}(g)=g'}} \psi_g\circ \varpi_{1}^{*\otimes 3}\cdot g'\,,
        \end{aligned}
    \end{align}
    where ${\rm pr}: {G}_{r+1}\to G_r$ is the natural projection. 
\end{remark}

\begin{defn}
    \label{defn_Gal_twisted_general} 
    Recall from \S\ref{subsubsec_2022_05_16_1506} the universal cyclotomic character $\bbchi$. We denote by $\langle\bbchi\rangle$ its restriction to the maximal pro-$p$-subgroup $\Gal(\QQ(\mu_{p^\infty})/\QQ(\mu_p))$ of $\Gal(\QQ(\mu_{p^\infty})/\QQ)$, which we also regard as a character of $\Gal(\QQ(\mu_{p^\infty})/\QQ)$ via the natural decomposition $\Gal(\QQ(\mu_{p^\infty})/\QQ)\simeq \Gal(\QQ(\mu_{p^\infty})/\QQ(\mu_p))\times \Gal(\QQ(\mu_{p})/\QQ)$. We let $\langle\bbchi\rangle_r: \Gal(\QQ(\mu_{p^\infty})/\QQ)\to \ZZ_p[G_r]$ denote the compositum
$$\Gal(\QQ(\mu_{p^\infty})/\QQ)\xrightarrow{\langle\bbchi\rangle}\LL(1+p\ZZ_p)\xrightarrow{[1+p^{r-1}]\mapsto 1} \ZZ_p[\ZZ/p^{r-1}\ZZ]\xrightarrow{\rm diag}\ZZ_p[G_r]$$
    where the final arrow is the diagonal map. We denote by $\ZZ_p[G_r]^\dagger$ the free $\ZZ_p[G_r]$-module of rank one on which $G_\QQ$ acts via $\langle\bbchi\rangle_r^{-\frac{1}{2}}$. For any $\ZZ_p[G_r]$-module $M$, we put $M^\dagger:=M \otimes_{\ZZ_p[G_r]}\ZZ_p[G_r]^\dagger$.
\end{defn}

It is clear that $\langle\bbchi\rangle_r$ factors through $\Gal(\QQ(\mu_{p^r})/\QQ)$, and that it is a faithful character of $\Gal(\QQ(\mu_{p^r})/\QQ(\mu_p))$.

\begin{lemma}[\cite{DarmonRotger}, Lemma 2.6]
    \label{lemma_Gal_action_on_bold_Delta}
    \item[i)] For any $a,b,c\in \FF_p^\times$, the class $\Delta_r^{\etale}[[a,b,c]]$ is the image of a unique class (which we also denote by $\Delta_r^{\etale}[[a,b,c]]$) under the compositum of the arrows
    \begin{align*}
        H^1(\QQ(\mu_{p}),  H^1_{\etale} (\overline{X}_1(Np^r), \ZZ_p)^{\otimes 3})(2)^\dagger)\xrightarrow{\rm res} &\, H^1(\QQ(\mu_{p^r}),  H^1_{\etale} (\overline{X}_1(Np^r), \ZZ_p)^{\otimes 3})(2)^\dagger)\\
        &\qquad \xrightarrow{\,\,\sim\,\,} H^1(\QQ(\mu_{p^r}),  H^1_{\etale} (\overline{X}_1(Np^r), \ZZ_p)^{\otimes 3})(2))\,.
    \end{align*}
    \item[ii)] For any $m\in \FF_p^\times$, we have 
    $$\sigma_m\Delta_r^{\etale}[[a,b,c]]=\Delta_r^{\etale}[[m^{-1}a,m^{-1}b,m^{-1}c]]\,.$$
\end{lemma}

\begin{proof}
    The uniqueness claim in (i) follows from the inflation-restriction sequence (which tells us that the relevant restriction map is injective). To prove the existence (also by the inflation-restriction sequence), we must show that for any integer $m\equiv 1 \pmod p$, we have
    $$\sigma_m\Delta_r^{\etale}[[a,b,c]]=\langle m,m,m\rangle^{\frac{1}{2}}\Delta_r^{\etale}[[a,b,c]]\,.$$
    Unravelling the definition of $\Delta_r^{\etale}[[a,b,c]]$ (cf. Equation~\eqref{eqn_2023_11_10_1126}), this is equivalent to checking that 
    \begin{align*}\sum_{\langle \underline{d}\rangle=\langle d_1,d_2,d_3\rangle \in G_r}\Delta_r^{\etale}[d_2d_3a,d_1d_3b,d_1d_2c]^{\sigma_m}\,\langle \underline{d}\rangle&\,=\sum_{\langle d_1,d_2,d_3\rangle \in G_r}\Delta_r^{\etale}[d_2d_3a,d_1d_3b,d_1d_2c]\,\langle m^{\frac{1}{2}}d_1, m^{\frac{1}{2}}d_2, m^{\frac{1}{2}}d_3\rangle\\
    &\,=\sum_{\langle \underline{d}\rangle=\langle d_1,d_2,d_3\rangle \in G_r}\Delta_r^{\etale}[m^{-1}d_2d_3a,m^{-1}d_1d_3b,m^{-1}d_1d_2c]
    \,\langle \underline{d}\rangle\,.
    \end{align*}
    We therefore have reduced to checking that
    $$\Delta_r^{\etale}[e_1,e_2,e_3]^{\sigma_m}=\Delta_r^{\etale}[m^{-1}e_1,m^{-1}e_2,m^{-1}e_3]\,,$$
    which is immediate from \eqref{eqn_2023_11_11_1700}.

    The proof of the second part also follows from \eqref{eqn_2023_11_11_1700}.
\end{proof}

\begin{defn}
    \label{definition_2023_11_11_1723}
    Following \cite[Eqn. (2.27)]{DarmonRotger}, we put 
    $\Delta_r^{\etale}:=\frac{p^3}{(p-1)^3}\sum_{a,b,c\in \FF_p^\times}\Delta_r^{\etale}[[bc,ac,ab]]$.
\end{defn}
We then have 
$$\Delta_r^{\etale}\in H^1(\QQ,H^1_{\etale} (\overline{X}_1(Np^r), \ZZ_p)^{\otimes 3}(2)^\dagger)\,,$$ 
cf. the proof of \cite[Lemma 2.10]{DarmonRotger}.
\begin{remark}
    The relation between $\Delta_r^{\etale}$ and $\Delta_r^{\etale}[1,1,1]$ is akin to the relation between $\mathfrak{X}_{1,s}$ in \cite[Eqn. 8]{howard2007Inventiones} and the Heegner class $Q_\p$ given as in \cite[p. 809]{howard2007Central}.
\end{remark}

\subsubsection{Big diagonal cycles: Construction}
\label{subsubsec_426_2023_11_17_1541}
The following behaviour of the diagonal cycles under the degeneracy map $\varpi_1$ is key to the construction of big diagonal cycles.
\begin{lemma}[\cite{DarmonRotger}, Lemma 2.5]
\label{DR_lemma_2_5_dual}
   For any positive integer $r$ as well as a triple $(d_1',d_2',d_3')\in \widetilde{G}_{r+1}$ whose image in $\widetilde{G}_r$ is $(d_1,d_2,d_3)$, we have 
$$\varpi_{1,*}^{\otimes 3}\,\Delta_{r+1}^{\rm ord}[d_1',d_2',d_3']=U_p'^{\otimes 3}\,\Delta_r^{\rm ord}[d_1,d_2,d_3]\,.$$
\end{lemma}

\begin{proof}
    This is an immediate consequence of the second statement in \cite[Lemma 2.5]{DarmonRotger}, noting that the Atkin--Lehner operators (that appear in the definition \eqref{eqn_2023_11_08_1351} of $\Delta_r^{\rm ord}[d_1,d_2,d_3]$) interchange $\varpi_{1,*}$ and $\varpi_{2,*}$, and intertwine $U_p$ and $U_p'$.
\end{proof}

\begin{corollary}
    \label{cor_DR_lemma_2_5_dual} Suppose that $a,b,c\in \FF_p^\times$ are arbitrary elements. Recall that we regard them as elements of $(\ZZ/p^{r}\ZZ)^\times$ via their Teichm\"uller lifts.
    \item[i)] We have $$\varpi_{1,*}^{\otimes 3}\,Z_{r+1}=U_p'^{\otimes 3}\,Z_r\,, \qquad  Z_j=\Delta_j^{\etale}[a,b,c]\hbox{ or } \Delta_j^{\etale}[[a,b,c]]\,.$$ 
    \item[ii)] $\varpi_{1,*}^{\otimes 3}\,\Delta_{r+1}^{\etale}=U_p'^{\otimes 3}\,\Delta_{r}^{\etale}\,.$
\end{corollary}

\begin{proof}
    The first claim with $Z_j=\Delta_j^{\etale}[a,b,c]$ follows immediately from definitions and Lemma~\ref{DR_lemma_2_5_dual} combined with the Hecke-equivariance of the Abel--Jacobi map ${\rm cl}_1$, whereas the same claim with $Z_j=\Delta_j^{\etale}[[a,b,c]]$ follows from the first\footnote{One may give a direct proof of this assertion with $Z_j=\Delta_j^{\etale}[[a,b,c]]$, using the description in Remark~\ref{defn_3rd_vertical_map}(ii) of the morphism
    $${\rm Hom}_{\ZZ_p[G_{r+1}]}({}_{r+1}\mathscr{S},\ZZ_p[G_{r+1}]) \,\lra\, {\rm Hom}_{\ZZ_p[G_{r}]}({}_{r}\mathscr{S},\ZZ_p[G_{r}])\,.$$
    } in view of Remark~\ref{remark_2023_11_13_1509}(i). 
\end{proof}

Recall from Definition~\ref{defn_Gal_twisted_general} the character $\langle\bbchi\rangle$. Let us put $G_\infty:=\varprojlim_r G_r$. We define $\langle \bbchi^{\otimes 3}\rangle: \Gal(\QQ(\mu_{p^\infty})/\QQ)\to \ZZ_p[[G_\infty]]$ as the compositum
$$ \Gal(\QQ(\mu_{p^\infty})/\QQ)\xrightarrow{\langle\bbchi\rangle} \LL(1+p\ZZ_p)\xrightarrow{\,\,\rm diag\,\,} \ZZ_p[[G_\infty]]\,.$$
We denote by $\ZZ_p[[G_\infty]]^\dagger$ the free $\ZZ_p[[G_\infty]]$-module of rank one on which $G_\QQ$ acts via $\langle \bbchi^{\otimes 3}\rangle^{-\frac{1}{2}}$. For any $\ZZ_p[[G_\infty]]$-module $M$, we put $M^\dagger:=M \otimes_{\ZZ_p[[G_\infty]]}\ZZ_p[[G_\infty]]^\dagger$.

\begin{defn}
    \label{definition_big_Heegner_main}
    We define the big diagonal cycle $\Delta_\infty^{\etale}$ on setting
    $$\Delta_\infty^{\etale}:=\{(U_p'^{\otimes 3})^{-r}\Delta_{r}^{\etale}\}_r\in \varprojlim_{\varpi_{1,*}^{\otimes 3},\, r}  H^1(\QQ,{\rm e}_{\rm ord}'H^1_{\etale}(Y_1(Np^r),\ZZ_p)^{\otimes 3}(2)^\dagger)=H^1(\QQ,H^1_{\rm ord}(Y_1(Np^\infty))^{\otimes 3}(-1)^\dagger)\,.$$
\end{defn}

We note that this definition makes sense thanks to Corollary~\ref{cor_DR_lemma_2_5_dual}, relying also on the fact that $U_p'^{\otimes 3}$ acts invertibly on 
$${\rm e}_{\rm ord}' H^1_{\etale}(Y_1(Np^r),\ZZ_p)^{\otimes 3}(2)^\dagger=:H^1_{\rm ord}(Y_1(Np^\infty))^{\otimes 3}(-1)^\dagger\,,$$
which is where the cohomology class $\Delta_{r}^{\etale}$ takes coefficients in.

\begin{defn}
    \label{definition_big_diagonal_Hida_families}
    Let $(\f,\g,\h)$ denote a triple of primitive Hida families with tame levels $(N_\f, N_\g, N_\h)$ and tame nebentype characters verifying\footnote{As a matter of fact, a stronger version of this hypothesis (recorded in \S\ref{subsubsec_root_numbers_2024_02_09_1652}) is enforced throughout our paper.} $\varepsilon_\f\varepsilon_\g\varepsilon_\h =\mathds{1}$. Let us put $N:={\rm lcm}(N_\f, N_\g, N_\h)$ and denote by 
    $$\Delta^{\etale}(\f,\g,\h)\in H^1(\QQ,T_{\f\g\h}^\dagger)$$
    the associated big diagonal cycle, where
    \begin{itemize}
        \item $T_{\f\g\h}=T_\f \widehat{\otimes} T_{\g}\widehat{\otimes} T_{\h}$,
        \item $T_{\f\g\h}^\dagger:=T_{\f\g\h}\otimes (\bbchi_\f\bbchi_\g\bbchi_\h)^{-\frac{1}{2}}\chi_{\cyc}^{-1}$\,,
    \end{itemize}
    which is given as the image of $\Delta_\infty^{\etale}$ under the map induced from the compositum of the following arrows:
    $$H^1_{\rm ord}(Y_1(Np^\infty))^{\otimes 3}\xrightarrow{\varpi_{1,*}\otimes \varpi_{1,*} \otimes \varpi_{1,*}} H^1_{\rm ord}(Y_1(N_\f p^\infty))\,\widehat{\otimes}\,H^1_{\rm ord}(Y_1(N_\g p^\infty))\,\widehat{\otimes}\,H^1_{\rm ord}(Y_1(N_\h p^\infty))\lra T_{\f\g\h}\,,$$
    where the final arrow is obtained using the definition of $T_{\f\g\h}$, cf.  \eqref{eqn_2023_11_13_1642}.
\end{defn}

We conclude this subsection with the following definition, restricting our attention to the particular scenario where $\h=\g^c$:
\begin{defn}
    \label{defn_2023_11_15_1209}
    Let us denote by $\Delta^\etale(\f\otimes {\rm ad}(\g)) \in H^1(\QQ,T)$ the image of the cohomology class $\Delta^{\etale}(\f,\g,\g^c)\in H^1(\QQ,T_{\f\g\g^c}^\dagger)$ under the natural map
    $$H^1(\QQ,T_{\f\g\g^c}^\dagger)\xrightarrow{\eqref{eqn_2023_09_26_1043}\,\circ\, \iota_{2,3}^*} H^1(\QQ,T)\,.$$
\end{defn}

\subsubsection{Specializations of big diagonal cycles}
\label{subsubsec_2023_10_02_1711}
Let us consider the set 
$$\cA^{(2)}_r:=\{(\kappa,\lambda)\in \cW_2: {\rm wt}(\kappa)=2={\rm wt}(\lambda) \hbox{ and } \f_\kappa,\g_\lambda \hbox{ are new of respective levels } N_\f p^r \hbox{ and } N_\g p^r\}$$
of arithmetic specializations. Let us put $\cA^{(2)}:=\cup_r \cA^{(2)}_r$ and observe that $\cA^{(2)}$ is a dense subset of $\cW_2$. Our main goal in the present subsection is to explicitly describe the specializations of $\Delta^\etale(\f\otimes {\rm ad}(\g))$ to $(\kappa,\lambda)\in \cA^{(2)}$. We begin our discussion with the following variation of \eqref{eqn_2023_11_10_1126} and define, for each $(\kappa,\lambda)\in \cA_r^{(2)}$,
\begin{equation}
    \label{eqn_2023_11_15_1243}
    \Delta_r^{\rm ord}(\kappa,\lambda):=\sum_{(d_1,d_2,d_3)\in \widetilde{G}_r} \Delta_r^{\rm ord}[d_2d_3,d_3d_1,d_1d_2] \cdot \psi_\kappa(d_1)\,\,\in\,\, {\rm CH}_0^2(X_1(Np^r)^3;\QQ(\mu_{p^r}))_{\ZZ_p}\otimes_{\ZZ_p}\cO_{\kappa,\lambda}
\end{equation}
where $\cO_{\kappa,\lambda}$ is the ring of integers of the field generated over $\QQ_p$ by the Hecke fields of $\f_\kappa$, $\g_\lambda$ and $\g^c_\lambda$. We further denote by 
$\Delta_r^\etale(\kappa,\lambda)\in H^1(\QQ(\mu_{p^r}),T)$ the image of $ \Delta_r^{\rm ord}(\kappa,\lambda)$ under the composition of the following arrows:
\begin{align}
    \label{eqn_2023_11_15_1557}
    \begin{aligned}
        {\rm CH}_0^2(X_1(Np^r)^3;\QQ(\mu_{p^r}))_{\ZZ_p}\otimes_{\ZZ_p}\cO_{\kappa,\lambda} \xrightarrow{\,\,{\rm cl}_1\,\,} H^1(\QQ(\mu_{p^r}), & \,H^1_{\etale}(\overline{X}_1(Np^r),\ZZ_p)^{\otimes 3}(2))\\
        &\xrightarrow{{\rm pr}_{\f_\kappa}\otimes {\rm pr}_{\g_\lambda} \otimes {\rm pr}_{\g^c_\lambda} } H^1(\QQ(\mu_{p^r}),T)\,,
    \end{aligned}
\end{align}
    where, for a cuspidal newform $h$ of level $\Gamma_1(N_hp^r)$ with $N_h\mid N$, the map ${\rm pr}_h$ is induced from the pushforward map 
    $H^1_{\etale}(\overline{X}_1(Np^r),\ZZ_p)(1)\xrightarrow{\varpi_{1,*}} H^1_{\etale}(\overline{X}_1(N_hp^r),\ZZ_p)(1)$ composed with projection to the $h$-isotypical Hecke eigenspace. 

\begin{lemma}
    \label{lemma_2023_11_14_1442}
    Suppose that $(\kappa,\lambda)\in \cA_r^{(2)}$. Then we have
    $$\res_{\QQ(\mu_{p^r})/\QQ}\,\,\Delta^\etale(\f\otimes {\rm ad}(\g))_{\vert_{(\kappa,\lambda)}}=a_p(\f_\kappa)^{-r}a_p(\g_\lambda)^{-2r} \Delta_r^\etale(\kappa,\lambda)\,,$$
    where $\Delta^\etale(\f\otimes {\rm ad}(\g))_{\vert_{(\kappa,\lambda)}}$ is the image of $\Delta^\etale(\f\otimes {\rm ad}(\g))$ under the specialization map
    $$H^1(\QQ,T)=H^1(\QQ,T_\f^\dagger\otimes {\rm ad}(T_\g))\xrightarrow{\,\kappa\otimes\lambda\,} H^1(\QQ,T_{\f_\kappa}^\dagger\otimes {\rm ad}(T_{\g_\lambda}))\,.$$
\end{lemma}

\begin{proof}
  Immediate from definitions.
\end{proof}

\begin{remark}
    We invite the reader to compare Lemma~\ref{lemma_2023_11_14_1442} with the discussion in \cite[p. 809]{howard2007Central}, where the latter involves a comparison of Howard's big Heegner point attached to a Hida family $\f$ and his twisted Heegner point associated to the specialization $\f_\kappa$ of $\f$ at an arithmetic point $\kappa\in \cA^{(2)}$.
\end{remark}

For archimedean aspects (e.g. relevant to our discussion in \S\ref{sec_ChowHeegandGKS} below), we shall consider the following variant of the cycle $\Delta_r^{\ord}(\kappa,\lambda)$ given as in \eqref{subsubsec_2023_10_02_1711}:
\begin{defn}
    \label{defn_2023_11_17_1557}
    For any $(\kappa,\lambda)\in \cA_r^{(2)}$, let us choose $q$ as in Definition~\ref{defn_2023_11_17_1544} so that 
    $$C_q:=(a_q(\f_\kappa)-(q+1)) \cdot(a_q(\g_\lambda)-(q+1))\cdot (a_q(\g_\lambda^c)-(q+1))$$
    is a $p$-adic unit (this is possible since $\f_\kappa$ and $\g_\lambda$ are non-Eisenstein mod $p$, cf. the discussion in \cite{DarmonRotger}, \S4). We put
    $$\Delta_r^{\circ}(\kappa,\lambda):=C_q^{-1}\sum_{(d_1,d_2,d_3)\in \widetilde{G}_r} \Delta_r^{\circ}[d_2d_3,d_3d_1,d_1d_2] \cdot \psi_\kappa(d_1)\,\,\in\,\, {\rm CH}_0^2(X_1(Np^r)^3;\QQ(\mu_{p^r}))\otimes_{\ZZ}F_{\kappa,\lambda}\,,$$
    where $F_{\kappa,\lambda}$ is the joint of the Hecke fields of the normalized eigenforms $\f_\kappa$, $\g_\lambda$ and $\g^c_\lambda$. 
\end{defn}

Note that $\Delta_r^{\circ}(\kappa,\lambda)$ is independent of the choice of $q$, thanks to the factor $C_q^{-1}$. 

\begin{lemma}
    \label{lemma_2023_11_17_1613}
    The image of $\Delta_r^{\circ}(\kappa,\lambda)$ under the composite map
    $${\rm CH}_0^2(X_1(Np^r)^3;\QQ(\mu_{p^r}))\otimes_{\ZZ}F_{\kappa,\lambda}\xrightarrow{{\rm cl}_1} H^1(\QQ(\mu_{p^r}), H^1_{\etale}(\overline{X}_1(Np^r),\ZZ_p)^{\otimes 3}(2))\otimes_{\ZZ}F_{\kappa,\lambda}\xrightarrow{{\rm pr}_{\kappa,\lambda}}H^1(\QQ(\mu_{p^{r}}),T_{\kappa,\lambda})$$
   coincides with $\Delta_r^{\etale}(\kappa,\lambda)$. Here, ${\rm pr}_{\kappa,\lambda}$ is the shorthand for ${\rm pr}_{\f_\kappa}\otimes {\rm pr}_{\g_\lambda} \otimes {\rm pr}_{\g^c_\lambda}$, and $T_{\kappa,\lambda}:=T_{\f_\kappa}\otimes {\rm ad}(T_{g_\lambda})$.
\end{lemma}

\begin{proof}
    Note that both operators
    $$C_q^{-1}(T_q'-(q+1))^{\otimes 3}\,\quad,\,\quad (e_{\rm ord}')^{\otimes 3}$$
    act as the identity on the $(\f_\kappa,\g_\lambda,\g^c_\lambda)$-isotypic quotients of $H^1_{\etale}(\overline{X}_1(Np^r),\ZZ_p)^{\otimes 3}(2)$; the latter because the $p$-stabilized eigenforms $\f_\kappa$, $\g_\lambda$ and $\g_\lambda^c$ are $p$-ordinary. The lemma follows from definitions together with this observation.
\end{proof}

\subsection{Second reduction step}
We combine Lemma~\ref{lemma_2023_11_14_1442} and Lemma~\ref{lemma_first_reduction} to reduce the proof of Conjecture~\ref{conj_main_6_plus_2} to the comparison recorded as \eqref{eqn_2023_11_16_1154} below. 

Let us define $\Delta^{\rm (tr)}_r(\kappa,\lambda)\in H^1(\QQ(\mu_{p^r}),T_{\f_\kappa})$ as the image of $\Delta_r^{\etale}(\kappa,\lambda)$ under the map
$$H^1(\QQ(\mu_{p^r}),T_{\kappa,\lambda})=H^1(\QQ(\mu_{p^r}),T_{\f_\kappa}\otimes{\rm ad}(T_{g_\lambda}))\xrightarrow{{\rm id}\otimes {\rm tr}} H^1(\QQ(\mu_{p^r}),T_{\f_\kappa})\,.$$

\begin{lemma}
     \label{lemma_second_reduction}
    Conjecture~\ref{conj_main_6_plus_2_bis} follows if \eqref{item_deg6} holds and we have 
    \begin{equation}
        \label{eqn_2023_11_16_1154}
      \res_\wp \circ \res_{L_r/\QQ(\mu_{p^r})}\, \Delta^{\rm (tr)}_r(\kappa,\lambda)= \pm a_p(\g_\lambda)^{r} \,\lambda_{Np^{r}}(\g_\lambda)\, p^{1-r}[L_{r}:H_{p^r}]^{-1}\cdot \mathscr{M}_{\ref{lemma_second_reduction}}\cdot \res_\wp\, Q_\kappa
    \end{equation}
    for a subset consisting of $(\kappa,\lambda)\in \cA^{(2)}$ which is dense in $\cW_2$, where:
    \begin{itemize}
    \item The positive integer $r$ is such that $(\kappa,\lambda)\in \cA^{(2)}_r$\,, and $s_\kappa=r$;
    \item $\wp$ is a prime of $L_r$ above the unique prime of $K(\mu_{p^r})$ that lies above the prime $\p$ of $\cO_K$, and the equality takes place in $H^1_{\rm f}(L_{r,\wp},T_{\f_\kappa})$\,;
    \item $\lambda_{Np^{r}}(\g_\lambda)$ is the Atkin--Lehner pseudo-eigenvalue of the indicated level\,;
        \item $Q_\kappa$ is Howard's twisted Heegner point, given as in \S\ref{subsubsec_222_2023_11_16_1203}, over the fixed imaginary quadratic field $K$;
        \item $\mathscr{M}_{\ref{lemma_second_reduction}}^2=\mathscr{D}(\kappa,\lambda)\cdot\mathfrak{g}(\psi_{\kappa}^{\frac{1}{2}})\cdot C_{\f_\kappa}^- \cdot \dfrac{\Lambda(\hf_\kappa\otimes\Ad^0\hg_\lambda,\psi_{\kappa}^{-\frac{1}{2}},1)}{a_p(\f_\kappa)^{-r}\,\Omega_{\hf_\kappa}^- \, \Omega_{\hg_\lambda}^{\rm ad}}$.
    \end{itemize}
\end{lemma}

\begin{proof}
We begin our proof recalling from \S\ref{subsubsec_114_2023_09_27_1616}--\S\ref{subsubsec_117_2024_02_08_1841} that the map $T_{\kappa,\lambda} \to T_\f^\dagger$ is induced from the morphism
\begin{equation}
\label{eqn_2024_02_08_1848}
    T_{\g_\lambda}\otimes T_{\g_\lambda^c}\xrightarrow{a_p(\g_\lambda)^r\lambda_{N_\g p^r}^{-1}\,\langle\,\,,\,\,\rangle_r} \overline{\QQ}_p\,,
\end{equation}
where $\langle\,,\,\rangle_r$ is the Poincar\'e duality pairing on the modular curve $X_r$. 
Combining this observation with Lemma~\ref{lemma_2023_11_14_1442} and the compatibility of the Poincar\'e duality pairing with degeneracy maps, we deduce that 
\begin{equation}
\label{eqn_2024_02_08_1858}\res_{\QQ(\mu_{p^r})/\QQ}\,\,\Delta^{\rm tr}(\f\otimes {\rm ad}(\g))_{\vert_{(\kappa,\lambda)}}=a_p(\f_\kappa)^{-r}a_p(\g_\lambda)^{-r} \lambda_{Np^{r}}^{-1}(\g_\lambda)\, \Delta_r^{(\rm tr)}(\kappa,\lambda)\,.
\end{equation}
  Let us assume that \eqref{eqn_2023_11_16_1154} holds, which we use together with \eqref{eqn_2024_02_08_1858} to conclude
\begin{equation}
\label{eqn_2023_11_18_1650}
  a_p(\f_\kappa)^{-r} \res_\wp \circ \res_{L_r/\QQ}\, \Delta^{\rm tr}(\f\otimes {\rm ad}(\g))_{\vert_{(\kappa,\lambda)}}= \pm  p^{1-r}[L_{r}:H_{p^r}]^{-1}\cdot \mathscr{M}_{\ref{lemma_second_reduction}}\cdot \res_\wp\, Q_\kappa\,.
\end{equation}
Since the map 
$$\res_{L_{r,\wp}/K_\p}\,:\, J_r(K_\p)\lra  J_r(L_{r,\wp})$$
is injective, we infer from \eqref{eqn_2023_11_18_1650} that
    \begin{equation}
\label{eqn_2023_11_18_1651}
    \log_{\omega_{\f_\kappa}}\left(\res_p\,\Delta^{\rm tr}(\f\otimes {\rm ad}(\g))_{\vert_{(\kappa,\lambda)}}\right)= \pm  p^{1-r}[L_{r}:H_{p^r}]^{-1}\cdot a_p(\f_\kappa)^{-r}\cdot \mathscr{M}_{\ref{lemma_second_reduction}}\cdot \log_{\omega_{\f_\kappa}}\left(\res_\wp\, Q_\kappa\right)\,.
\end{equation}
The proof of our lemma follows now from Lemma~\ref{lemma_first_reduction}.
\end{proof}

\section{Chow--Heegner points and Generalized Gross--Kudla Conjecture}
\label{sec_ChowHeegandGKS}
In what follows, we shall denote by $\LL(-)$ the completed $L$-functions.

\subsection{Beilinson--Bloch Heights}\label{sec:intersectionpairing_bis}

As in Section~\ref{subsec_big_diagonal}, we denote by $\CH^2(X^3)$ the Chow group of codimension two cycles on $X^3$, for a smooth, projective, and connected curve $X$ over a number field $K$. Recall the filtration
\[
\CH^2(X^3) \supset \CH^{2,1}(X^3)=\CH^2_0(X^3) \supset \CH^{2,2}(X^3) \supset \CH^{2,3}(X^3)\,,
\]
 which can be upgraded to a decomposition of $\CH^2(X^3)$ (cf. Equation (3.1.1) in \cite{YZZ12}). 

The work of Beilinson \cite{Beilinson-higherregulators, Beilinson-heightpairing} and Bloch \cite{bloch} defines a height pairing 
\[
\langle \, , \, \rangle_{\mathrm{BB}}: \CH^2_0(X^3) \times \CH^2_0(X^3) \, \longrightarrow \, \C
\]
on homologically trivial cycles. We refer the reader to \cite[\S 3.1]{YZZ12} for a review of its definition in our case of interest.

\subsection{Generalized Gross--Kudla Conjecture}
Until the end of this section, we fix a positive integer $r$ and $(\kappa,\lambda) \in \mathcal A_r^{(2)}$. We also write $X_r$ for $X_1(Np^r)$ and put $f = \mathbf f_{\kappa}$, $g = \mathbf g_{\lambda}$. We assume that both $f$ and $g$ are newforms and that $s_\kappa=r=s_\lambda$. Let us put
\begin{align}
    \label{eqn_2024_03_11_1724}
    \begin{aligned}
            \Omega_{f\otimes {\rm ad}(g)}:=\,a_p(f)^{-2r}\,\langle f,  f \rangle\,\langle g,  g\rangle^2\,.
    \end{aligned}
\end{align}

Recall the cycle class
\[
\Delta_r^{\circ}(\kappa,\lambda) \in \CH_0^2(X_r^3; \Q(\mu_{p^r}))\otimes_{\Z} F_{\kappa,\lambda}
\]
from the previous section. In line with the Generalized Gross--Kudla Conjecture, as stated (and partially proved) in Yuan--Zhang--Zhang (cf. \cite{YZZ10,YZZ12,YZZ23}), we conjecture that the Beilinson--Bloch height of $\Delta_r^{\circ}(\kappa,\lambda)$ is related to the first central derivative of the complex  $L$-series $L(f\otimes g \otimes g^c, s)$:

\begin{conj}\label{conj:GK-delta} 
With the notation above, we have
\[
\langle \Delta_r^{\circ}(\kappa,\lambda), \Delta_r^{\circ}(\kappa,\lambda) \rangle_{\mathrm{BB}} = \frac{ p^{-r}\cdot C(\kappa,\lambda) }{  \Omega_{f\otimes {\rm ad}(g)}}\LL'(f\otimes g \otimes g^c,\psi_{\kappa}^{-\frac{1}{2}}, 2),
\]
where $C(\kappa,\lambda)$ is (generically) a non-zero algebraic constant which interpolates $p$-adically as $(\kappa,\lambda)$ varies.
\end{conj}

\begin{remark}
    Suppose in this remark that $r=0$, and let us denote by $f_\circ$ and $g_\circ$ the newforms associated with $\f_\kappa$ and $\g_\lambda$, respectively (of respective levels $N_\f$ and $N_\g$). In this scenario, Conjecture~\ref{conj:GK-delta} agrees with \cite[Conjecture 13.2]{GrossKudla1992}, minding that we have used complete $L$-series in our formulation.
    
    We further remark that $\LL(f\otimes g \otimes g^c, {\rm ad}, 1)$, which is the period utilized in \cite{YZZ10, YZZ12, YZZ23}, coincides with an explicit non-zero algebraic multiple of $\Omega_{f\otimes {\rm ad}(g)}$; cf. \cite[Eqn. (2.18)]{Hsieh} and the proof of Corollary 4.13 in op. cit.
\end{remark}

\begin{remark}
Let $\pi$ be the automorphic representation on $\GL_2\times\GL_2\times \GL_2$ associated to $f\otimes g \otimes g^c$.  By the main theorem of \cite{YZZ23} (Theorem 1.3.3 in op. cit.), one may prove a statement towards Conjecture~\ref{conj:GK-delta} with some error terms (which are described in terms of ramified primes of $\pi$).  As a result,  if $\pi$ is unramified, Conjecture \ref{conj:GK-delta} holds. In the ramified case (e.g. when $f$ and $g$ have wild nebentype, which is the typical scenario in our set-up), to the best of our knowledge\footnote{The first named author (K.B.) thanks Wei Zhang for extensive exchanges on this topic.}, there are no definitive results towards Conjecture \ref{conj:GK-delta}.
\end{remark}

\subsubsection{N\'eron--Tate heights of Chow--Heegner points}
\label{subsubsec_NT_ChowHeeg}
We denote by $P_r^{\circ}(\kappa,\lambda)$ the image of the cycle $\Delta_r^{\circ}(\kappa,\lambda)$ under the compositum of the maps
\[
\CH^2_0(X_r^3;\Q(\mu_{p^r}))\otimes_{\Z} F_{\kappa,\lambda} \, \xrightarrow{\mathrm{pr}_{123}^*} \, \CH^2_0(X_r^4;\Q(\mu_{p^r}))\otimes_{\Z} F_{\kappa,\lambda} \, \xrightarrow{\mathrm{pr}_{4,*}(\mathscr Z_r\cdot(-) )} \, \CH^1_0(X_r)\otimes_{\Z} F_{\kappa,\lambda},
\]
and call it (following Darmon--Rotger--Sols) the {\em Chow--Heegner point}. Here, $\mathrm{pr}_{123}: X_r^4 \to X_r^3$  (resp. $\mathrm{pr}_4: X_r^4 \to X$) is the natural projection onto the first three factors (resp. the fourth factor), and where $\mathscr Z_r$ is the cycle class represented by the image of 
$$\alpha: X_r^2 \lra X_r^4\,,\qquad (x,y)\mapsto (y,x,x,y)\,,$$
cf. \cite[\S1.4]{YZZ10}, see also\footnote{To explicate this comparison, let us denote (following \cite{DRS}) by $Z:=X_{23}\subset X_r^2=:X_2 \times X_3$ the diagonal embedding of $X_r$ into $X_r^2$ (where we label the factors as $X_2$ and $X_3$), and let us denote by $\Pi_Z:=X_{14}\times Z\subset X_r^4=X_1\times X_2\times X_3\times X_4$, where $X_{14}$ is the diagonal copy of $X_r$ embedded in $X_1\times X_4 =X_r^2$. Then the cycle class $\mathscr{Z}_r$ is indeed represented by $\Pi_Z$.} \cite[\S2]{DRS}. According to \cite[\S1.3.1]{YZZ12} (see also \cite{YZZ23}, Remark 3.1.1), the N\'eron--Tate height pairing of $P_r^{\circ}(\kappa,\lambda)$ is related to the Beilinson--Bloch height pairing of $\Delta_r^{\circ}(\kappa,\lambda)$:
\[
\langle \Delta_r^{\circ}(\kappa,\lambda), \Delta_r^{\circ}(\kappa,\lambda) \rangle_{\mathrm{BB}} = 2 \langle P_r^{\circ}(\kappa,\lambda), P_r^{\circ}(\kappa,\lambda) \rangle_{\mathrm{NT}},
\]
and Conjecture~\ref{conj:GK-delta} can be recast as follows:

\begin{conj}
\label{conj_5_2_2024_02_02}
We have
$$
\langle P_r^{\circ}(\kappa,\lambda), P_r^{\circ}(\kappa,\lambda) \rangle_{\mathrm{NT}} =\frac{p^{-r}\cdot C(\kappa,\lambda) \cdot \LL(f\otimes \mathrm{ad}^0(g),\psi_{\kappa}^{-\frac{1}{2}},1)}{\Omega_{f\otimes {\rm ad}(g)}}\LL'(f,\psi_{\kappa}^{-\frac{1}{2}}, 1)\,,
$$
where $C(\kappa,\lambda)$ is (generically) a non-zero algebraic constant which interpolates $p$-adically as $(\kappa,\lambda)$ varies.
\end{conj}

\begin{remark}
    Recall from \S\ref{subsubsec_root_numbers_2024_02_09_1652} that we assume that the global root numbers of the motives associated to $f\otimes g \otimes g^c$ and $f$ are both equal to $-1$. As a result, we have
    $$L(f\otimes g \otimes g^c,\psi_{\kappa}^{-\frac{1}{2}},2)=0=L(f,\psi_{\kappa}^{-\frac{1}{2}},1)$$
    at the central critical points.
\end{remark}

\section{Compatibility of de Rham and \'etale picture}

Throughout this section, we fix a positive integer $r$ and $(\kappa,\lambda)\in \cA_r^{(2)}$. To ease our notation, let us also write $X_r$ in place of $X_1(Np^r)$ and put $f=\f_\kappa$, $g=\g_\lambda$. Our main goal in the present section is to prove Proposition~\ref{prop_main_comparison_Sec5}, which can be thought of as a compatibility between the formation of the Chow--Hegner point (cf. \S\ref{subsubsec_NT_ChowHeeg}) and the morphism
$$ T_{f}\otimes T_{g}\otimes T_{g^c}(\psi_\lambda^{-1}\chi_\cyc^{-1})\xrightarrow{{\rm id}\otimes \langle\,,\,\rangle_\lambda} T_{f}$$
on the level of cohomology induced by Poincar\'e duality.

\subsection{de Rham Abel--Jacobi maps}
\label{subsec_51_2023_11_21_1417}
Let us consider the de Rham Abel--Jacobi maps 
\begin{align}
\label{eqn_2023_11_21_1515}
    \begin{aligned}
        {\rm AJ}_{\rm dR}^{(2)}\,:\,\, {\rm CH}_0^2(X_r^3;\QQ_p(\mu_{p^r}))\lra &\left({\rm Fil}^2(H^3_{\rm dR}(X_r^3/\QQ(\mu_{p^r})))\right)^\vee\\
        &\xrightarrow[\sim]{{\rm comp}_{\rm dR}}\left({\rm Fil}^2D_{\rm dR}(H^3_{\etale}(\overline{X}_r^3,\QQ_p))\right)^\vee\otimes_{\QQ_p}\QQ_p(\mu_{p^r}) \\
        {\rm AJ}_{\rm dR}^{(1)}\,:\,\, {\rm CH}_0^1(X_r;\QQ_p(\mu_{p^r}))\lra &\left({\rm Fil}^1(H^1_{\rm dR}(X_r/\QQ_p(\mu_{p^r})))\right)^\vee\\
        &\xrightarrow[\sim]{{\rm comp}_{\rm dR}}\left({\rm Fil}^1 D_{\rm dR}(H^1_{\etale}(\overline{X}_r,\QQ_p))\right)^\vee\otimes_{\QQ_p}\QQ_p(\mu_{p^r})\,;
    \end{aligned}
\end{align}
cf. \cite[\S5]{Besser2000} and \cite[\S3]{BesserLZ}. 

\subsubsection{} According to \cite[Theorem B]{NekovarNiziol2016} (see also \cite{Niziol1997}), we have the following compatibility with the \'etale Abel--Jacobi map (cf. \S\ref{subsubsec_2023_11_21_1550}):
\begin{equation}
    \label{eqn_2023_11_21_1542}
    \begin{aligned}
         \xymatrix{
    {\rm CH}_0^2(X_r^3;\QQ_p(\mu_{p^r})) \ar[r]^-{{\rm cl}_1}\ar[r]\ar[rrd]_-{{\rm AJ}_{\rm dR}^{(2)}}& H^1_{\rm f}(\QQ_p(\mu_{p^r}),H^3_{\etale}(\overline{X}_r^3,\QQ_p)(2)) \ar[r]^-{\log_{\rm BK}}& D_{\rm dR}(H^3_{\etale}(\overline{X}_r^3,\QQ_p)(2))/{\rm Fil}^0  \otimes_{\QQ_p}\QQ_p(\mu_{p^r})   \ar[d]^-{\sim} \\
&&{\rm Fil}^2\, D_{\rm dR}(H^3_{\etale}(\overline{X}_r^3,\QQ_p))^{\vee} \otimes_{\QQ_p}\QQ_p(\mu_{p^r}) 
    }
    \end{aligned}
\end{equation}
where $\log_{\rm BK}$ is the Bloch--Kato logarithm and the vertical isomorphism is induced by Poincar\'e duality. 
\subsubsection{}
Similarly, we have the following commutative diagram:
\begin{equation}
    \label{eqn_2023_11_21_1605}
    \begin{aligned}
         \xymatrix{
    {\rm CH}_0^1(X_r;\QQ_p(\mu_{p^r})) \ar[r]^-{{\rm cl}_1}\ar[r]\ar[rrd]_-{{\rm AJ}_{\rm dR}^{(1)}}& H^1_{\rm f}(\QQ_p(\mu_{p^r}),H^1_{\etale}(\overline{X}_r,\QQ_p)(1)) \ar[r]^-{\log_{\rm BK}}& D_{\rm dR}(H^1_{\etale}(\overline{X}_r,\QQ_p)(1))\big{/}{\rm Fil}^0  \otimes_{\QQ_p}\QQ_p(\mu_{p^r})   \ar[d]^-{\sim} \\
&&{\rm Fil}^1\, D_{\rm dR}(H^1_{\etale}(\overline{X}_r,\QQ_p))^{\vee} \otimes_{\QQ_p}\QQ_p(\mu_{p^r})\,. 
    }
    \end{aligned}
\end{equation}

\subsubsection{} Recall from \cite[\S10]{KLZ2} the classes $\omega_{f}, \omega_{g^c}\in {\rm Fil}^1 D_{\rm dR}(H^1_{\etale}({\overline X}_r,\QQ_p))\otimes \QQ(\mu_{Np^r})$ and $\eta_{g}\in D_{\rm dR}(H^1_{\etale}({\overline X}_r,\QQ_p))\otimes \QQ(\mu_{Np^r})$, and let us put 
$$\omega^{(g)}:=\omega_{f}\otimes \eta_g \otimes \omega_{g^c} \in {\rm Fil}^2D_{\rm dR}(H^3_{\etale}(\overline{X}_r^3,\QQ_p))\,.$$
Let us set 
$$\log_{\omega^{(g)}}:\,H^1_{\rm f}(\QQ_p(\mu_{p^r}),H^3_{\etale}(\overline{X}_r^3,\QQ_p)(2))\xrightarrow{\omega^{(g)}\,\circ\,\log_{\rm BK}} \QQ_p(\mu_{Np^r}).$$

We recall that, in order to define the element $\Delta_r^{\rm (tr)}(\kappa,\lambda)$, we rely on the modified Poincar\'e duality as in \eqref{eqn_2024_02_08_1848}, which is given by 
$$ T_{g}\otimes T_{g^c}(\psi_\lambda^{-1}\chi_\cyc^{-1})\xrightarrow{a_p(g)^r\lambda_{N_\g p^r}(g)^{-1}\,\left\langle\,\,,\,\,\right\rangle_r} \overline{\QQ}_p\,.$$
For our purposes in the present section, it will be convenient to consider its variant  $\Delta_r^{\rm (P)}(\kappa,\lambda)$ that one obtains using the Poincar\'e duality pairing 
$$T_{g}\otimes T_{g^c}(\psi_\lambda^{-1}\chi_\cyc^{-1})=T_g\otimes T_{\overline{g}}\,(\chi_\cyc^{-1})\xrightarrow{\,\sim\,} T_g\otimes T_g^*\xrightarrow{\langle\,\,,\,\,\rangle_r} \overline{\QQ}_p$$ instead (cf. \eqref{eqn_2024_02_09_0658} and Remark~\ref{remark_2023_11_22_0818}). Here, $T_g^*$ is the $g$-isotypic submodule of $H^1_{\textup{\' et}}(Y_1(N_\g p^r), {\rm Sym}^{k}(\mathscr{H}_{\ZZ_p}^\vee))$, which coincides with a lattice in the Galois representation denoted by $M_{L_\mathfrak{P}}(g)$ in \cite[\S2.8]{KLZ2}.
Note then that these two cohomology classes are related via
\begin{equation}
    \label{eqn_2024_02_09_0654}
  \Delta_r^{\rm tr}(\kappa,\lambda) = a_p(g)^r\lambda_{N_\g p^r}(g)^{-1}\, \Delta_r^{\rm (P)}(\kappa,\lambda)\,.
\end{equation}

\begin{lemma}
    \label{lemma_2023_11_12_0840}
    We have
    \begin{align}
        \label{eqn_lemma_2023_11_12_0840_1}
        \log_{\omega_f}\circ\,\res_p\,(\Delta_r^{\rm (P)}(\kappa,\lambda))=\log_{\omega^{(g)}}\,&\circ\, \res_p\,(\Delta_r^{\etale}(\kappa,\lambda))={\rm AJ}_{\rm dR}^{(2)}(\Delta_r^\circ(\kappa,\lambda))(\omega^{(g)})\,,\\
                \label{eqn_lemma_2023_11_12_0840_2}
       \log_{\omega_f}\,\circ\,\res_p\,(P_r^{\circ}(\kappa,\lambda))&={\rm AJ}_{\rm dR}^{(1)}(P_r^{\circ}(\kappa,\lambda) )(\omega_f)\,.
    \end{align}

\end{lemma}

\begin{proof}
    The first equality in \eqref{eqn_lemma_2023_11_12_0840_1} follows from the commutative diagram 
    \begin{equation}
        \label{eqn_2024_02_09_0658}
        \begin{aligned}
        \xymatrix{
        H^1_{\rm bal}(\QQ_p,T_{f}^\dagger\,\otimes\, T_{g}\,\otimes\,T_{\overline{g}}(\chi_\cyc^{-1}) )\ar[rr]^-{\log_{\omega^{(g)}}} \ar[d]_{{\rm id}\otimes \langle\,,\,\rangle_r}&&\overline{\QQ}_p\\
        H^1_{\rm f}(\QQ_p,T_f^\dagger) \ar[rru]_-{\log_{\omega_f}}&&
        }
    \end{aligned}
    \end{equation}
   and the definition of $\Delta_r^{\rm (P)}(\kappa,\lambda)$, whereas the second equality therein follows also from definitions and \eqref{eqn_2023_11_21_1542}. 
    
    The assertion \eqref{eqn_lemma_2023_11_12_0840_2} also follows from definitions and \eqref{eqn_2023_11_21_1605}.
\end{proof}

\subsection{Compatibility}
\label{subsec_52_2023_11_21_1418}
We are now in a position to state and prove the main results of the present section. Recall from the footnote in \S\ref{subsubsec_NT_ChowHeeg} that $Z\subset X_r^2$ denotes the diagonal copy of $X_r$. Let us denote by $[Z]\in {\rm CH}^1(X_r^2;\QQ_p(\mu_{p^r}))$ the class represented by $Z$. 

\begin{proposition}
    \label{prop_main_comparison_Sec5}
    ${\rm AJ}_{\rm dR}^{(2)}(\Delta_r^\circ(\kappa,\lambda))(\omega^{(g)})={\rm AJ}_{\rm dR}^{(1)}(P_r^{\circ}(\kappa,\lambda))(\omega_f)$.
\end{proposition}
\begin{proof}
The argument we present below is borrowed from \cite{daubthesis}, which we adapt to our setting with slight modifications. We recall that 
$$\omega^{(g)}:=\omega_{f}\otimes \eta_g \otimes \omega_{g^c} \in {\rm Fil}^2D_{\rm dR}(H^3_{\etale}(\overline{X}_r^3,\QQ_p))\,,$$ 
and 
$${\rm cl}_{\rm dR}: {\rm CH}^1(X_r^2;\QQ_p(\mu_{p^r}))\lra H^2_{\rm dR}(X_r^2/\QQ_p(\mu_{p^r}))(1)$$ 
is the de Rham cycle class map.  As in \cite[Lemma 4.2.2]{daubthesis}, we identify the cycle $[Z]$ with the Hecke correspondence denoted by $T_{g,n}$ in op. cit. (which one also views as an element of ${\rm CH}^1(X_r^2;\QQ_p(\mu_{p^r}))$).  By \cite[Propositions 2.1.2(1)]{daubthesis},  we observe that the projector $\epsilon_0$ acts on the components of the K\"{u}nneth decomposition of $H^2_{\rm dR}(X_r \times X_r)$ as follows: It annihilates $H^0_{\rm dR}(X_r) \otimes H^2_{\rm dR}(X_r)$ and $H^2_{\rm dR}(X_r) \otimes H^0_{\rm dR}(X_r)$, therefore gives rise to an element 
$${\rm cl}_{\rm dR}(\epsilon_0[Z])  \in H^1_{\rm dR}(X_r) \otimes H^1_{\rm dR}(X_r)\,.$$

By definition, $\eta_g$ and $\omega_{g^c}$ pair to $1$ under the Poincar\'e duality pairing, and $Z$ is the graph of the identity endomorphism $X_r\xrightarrow{\rm id} X_r$.  We therefore have an identification of ${\rm cl}_{\rm dR}(\epsilon_0[Z])$ with $\eta_g \otimes \omega_{ g^c}$. As a result, it suffices to prove that
 \begin{equation}
      \label{eqn_2023_11_22_1150}
      {\rm AJ}_{\rm dR}^{(1)}(P_r^{\circ}(\kappa,\lambda))(\omega_f)={\rm AJ}_{\rm dR}^{(2)}(\Delta_r^\circ(\kappa,\lambda))(\omega_f\otimes {\rm cl}_{\rm dR}(\epsilon_0[Z]))\,.
  \end{equation}
We recall that the Chow--Heegner cycle $P_r^{\circ}(\kappa,\lambda)$ is defined as the image of the cycle $\Delta_r^{\circ}(\kappa,\lambda)$ under the compositum of the maps (that we denote by $\Pi_{Z,*}$\,, following Daub)
\[
\CH^2_0(X_r^3;\Q(\mu_{p^r}))\otimes_{\Z} F_{\kappa,\lambda} \, \xrightarrow{\mathrm{pr}_{123}^*} \, \CH^2_0(X_r^4;\Q(\mu_{p^r}))\otimes_{\Z} F_{\kappa,\lambda} \, \xrightarrow{\mathrm{pr}_{4,*}(\mathscr Z_r\cdot(-) )} \, \CH^1_0(X_r;\QQ_p(\mu_{p^r}))\otimes_{\Z} F_{\kappa,\lambda}\,.
\]
We recall from the footnote in \S\ref{subsubsec_NT_ChowHeeg} that the cycle class $\mathscr Z_r \in \CH^2_0(X_r^4;\Q(\mu_{p^r}))$ can be represented by $\Pi_Z$. Then, as in \cite[pp. 8-9]{daubthesis}, we have the induced maps 
\begin{align*}
{\rm Fil}^1 H^1_{\rm dR}(X_r/\QQ_p) &\xrightarrow{\,\Pi_Z^{*}\,} {\rm Fil}^2 H^3_{\rm dR}(X_r^3/\QQ_p)\,. 
\end{align*}
 We therefore have $P_r^{\circ}(\kappa,\lambda)=\Pi_{Z,*}(\Delta_r^\circ(\kappa,\lambda))$, and $\omega_f\otimes {\rm cl}_{\rm dR}(\epsilon_0[Z])=\Pi_Z^{*}(\omega_f)$.  With these identifications at hand, the sought-after equality \eqref{eqn_2023_11_22_1150} follows from \cite[Proposition 2.3.5]{daubthesis}.
\end{proof}

\begin{corollary}
    \label{cor_prop_main_comparison_Sec5}
    $\log_{\omega_f}\circ\,\res_p\,(\Delta_r^{\rm (tr)}(\kappa,\lambda))=a_p(g)^r\,  \log_{\omega_f}\circ\,\res_p\,(P_r^{\circ}(\kappa,\lambda))$.
\end{corollary}

\begin{proof}
    This is immediate on combining Lemma~\ref{lemma_2023_11_12_0840}, Proposition~\ref{prop_main_comparison_Sec5}, and \eqref{eqn_2024_02_09_0654}.
\end{proof}

\begin{remark}
\label{remark_2023_11_22_0818} 
The proof of Proposition~\ref{prop_main_comparison_Sec5} combined with \eqref{eqn_2023_11_21_1542}, \eqref{eqn_2023_11_21_1605} and the diagram \eqref{eqn_2024_02_09_0658} show that the square and the triangle in the following diagram are Cartesian:
    \begin{equation}
    \label{eqn_2023_09_08_1218}
\begin{aligned}
    \xymatrix{
     {\rm CH}_0^2(X_r^3;\QQ(\mu_{p^r})) \ar[rr]^-{{\rm pr}_{\kappa,
    \lambda}\,\circ\, \res_p\,\circ\, {\rm cl}_1}\ar[d]_{(\Pi_{\mathscr{C}_g})_*} && H^1_{\rm f}(\QQ_p(\mu_{p^r}),T_f\otimes {\rm ad}(T_g))\ar[d]^{{\rm id}\otimes {\rm tr}} \ar[rr]^-{\omega^{(g)}\,\circ\,\log_{\rm BK}}&& \overline{\QQ}_p  \\
    {\rm CH}_0^1(X_r;\QQ_p(\mu_{p^r})) \ar[rr]_-{{\rm pr}_{\kappa}\,\circ\, \res_p\,\circ\, {\rm cl}_1} && H^1_{\rm f}(\QQ_p(\mu_{p^r}),T_f)\ar[rru]_-{\omega_f\,\circ\,\log_{\rm BK}}  &&
    }
\end{aligned}
\end{equation}
\end{remark}

\subsection{Third reduction step}
We combine Lemma~\ref{lemma_second_reduction} together with Corollary~\ref{cor_prop_main_comparison_Sec5} to reduce the proof of Conjecture~\ref{conj_main_6_plus_2} (which is equivalent to Conjecture~\ref{conj_main_6_plus_2_bis}) to the following comparison between the Chow--Heegner point $P_r^{\circ}(\kappa,\lambda)$ and the twisted Heegner point $Q_\kappa$, recorded as \eqref{eqn_prop_3rd_reduction_step} below.
\begin{proposition}
    \label{prop_3rd_reduction_step}
    Conjecture~\ref{conj_main_6_plus_2_bis}  follows if \eqref{item_deg6} holds and 
    \begin{equation}
        \label{eqn_prop_3rd_reduction_step}
     \res_{L_r/\QQ(\mu_{p^r})} \,P_r^{\circ}(\kappa,\lambda) \pm\lambda_{Np^{r}}(\g_\lambda)^2\, p^{1-r}[L_{r}:H_{p^r}]^{-1}\cdot \mathscr{M}_{\ref{lemma_second_reduction}}\cdot Q_\kappa \hbox{\,\,\,\,is torsion}
    \end{equation}
    for a subset consisting of $(\kappa,\lambda)\in \cA^{(2)}$ which is dense in $\cW_2$.
\end{proposition}


\section{Towards a proof of Conjecture~\ref{conj_main_6_plus_2}}
\label{sec_7_1_2024_03_06_1602}
We finally put the pieces together and explain how the considerations in the previous sections can be combined with Howard's twisted Gross--Zagier formula (cf. Theorem~\ref{twisted_GZ_formula}) to deduce Conjecture~\ref{conj_main_6_plus_2}.

\subsection{Howard's twisted Gross--Zagier formulae} 
\label{subsec_7_2_2024_02_06_1438}

Throughout this section, we fix a positive integer $r$ and $(\kappa,\lambda)\in \cA_r^{(2)}$. Recall that $\psi_\kappa$ denotes the central character of $\f_\kappa$, which we assume to have conductor $p^r$ (so that $s_\kappa=r$). Note that such specializations $(f,g)$ still are dense in the $\cW_\f\times \cW_\g$.

The following is a restatement of Howard's twisted Gross--Zagier formula in our particular case of interest.

\begin{theorem}[Howard]
        \label{twisted_GZ_formula} 
        We have
        $$\LL'(f, \psi_{\kappa}^{-\frac{1}{2}}, 1)=4p^{-r}\, \langle f,  f\rangle\, \frac{[{\rm SL}_2(\ZZ):\Gamma_0(p^rN_\f )]^{-1}}{\LL(f \otimes \epsilon_K, \psi_{\kappa}^{-\frac{1}{2}},  1)}\,\langle Q_\kappa, Q_\kappa \rangle_{\mathrm{NT}}\,.$$ 
\end{theorem}

\begin{proof}
    This follows from \cite[Theorem 5.6.2]{howardtwisted2009}, the complex factorization of the derivative of the $L$-function and the relation 
    $$||\phi_f||^2 =  \langle f,  f\rangle\, [{\rm SL}_2(\ZZ):\Gamma_0(p^rN_\f )]^{-1}\,, $$
    cf. \cite[p. 1403]{IchinoIkeda2010}; see also \cite[p. 445]{Hsieh}. We remark that the additional factor $\zeta_\QQ(2)^{-1}=6/\pi$ in op. cit. is present due to the normalization of Haar measures (which is different from that in \cite{howardtwisted2009}).
\end{proof}

\begin{remark}
    As is well-known, we have 
    $$[{\rm SL}_2(\ZZ):\Gamma_0(p^rN_\f)] = p^rN_\f \prod_{\ell | pN_\f}(1 + \frac{1}{\ell})=p^{r-1} \prod_{\ell | pN_\f} (1+\ell)\,,$$
    where the second equality is because we assume that $N_\f$ is square-free. We may therefore rewrite the identity of Theorem~\ref{twisted_GZ_formula} as
    \begin{equation}
        \label{eqn_2024_03_06_Howard_bis}
        \LL'(f, \psi_{\kappa}^{-\frac{1}{2}}, 1)=4p^{1-2r}\, \langle f,  f\rangle\, \times\,\dfrac{\prod_{\ell | pN_\f}(1 + \ell)^{-1}}{\LL(f \otimes \epsilon_K, \psi_{\kappa}^{-\frac{1}{2}},  1)}\,\times\,\langle Q_\kappa, Q_\kappa \rangle_{\mathrm{NT}}\,.
    \end{equation}
\end{remark}

\subsection{Chow--Heegner points vs. Heegner points}
\label{subsec_5_3_2023_09_08_1159}
Our goal in this section is to prove the comparison \eqref{eqn_prop_3rd_reduction_step} between Chow--Heegner points and Heegner points (for a fixed $(\kappa,\lambda)$ as in the start of \S\ref{sec_7_1_2024_03_06_1602}). This, in view of \ref{subsubsec_228_2024_03_06_1611} and Proposition~\ref{prop_3rd_reduction_step}, concludes the proof of Conjecture~\ref{conj_main_6_plus_2}.

\subsubsection{} We may assume\footnote{A conjecture of Greenberg (cf. \cite{Greenberg_1994_families}, see also \cite{TrevorArnold_Greenberg_Conj}, Conjectures 1.3 and 1.4) asserts that this  is always the case, but we need not assume the validity of this conjecture.} without loss of generality that there exists $\kappa$ such that $L'(f, \psi_{\kappa}^{-\frac{1}{2}}, 1)\neq 0$, as otherwise, it follows from Conjecture~\ref{conj_5_2_2024_02_02} (which we assume) and Theorem~\ref{twisted_GZ_formula} that both $P_r^{\circ}(\kappa,\lambda)$ and $Q_\kappa$ are torsion points (on the relevant Jacobian variety), and there is nothing to prove.

\subsubsection{} 
\label{subsubsec_2024_03_06_1703}
In this case, it follows from Lemma~\ref{lemma_2023_11_18_1624} that Howard's big Heegner point $\mathfrak{z}$ is non-torsion. It follows from \cite[Corollary~3.4.3]{howard2007Inventiones} that Nekov\'a\v{r}'s extended Selmer group $\widetilde{H}^1_{\rm f}(\QQ,T_\f^\dagger)$ (which coincides with the Greenberg Selmer group $H^1_{\rm f}(\QQ,T_\f^\dagger)$ which we have defined in \S\ref{subsubsec_211_2023_09_27_1236}) is of $\cR_\f$-rank one. 

By the control theorem for extended Selmer groups (cf. the proof of \cite{howard2007Inventiones}, Corollary 3.4.3) and identifying the extended Selmer groups therein with the Greenberg Selmer groups (which we may for all but finitely many $\kappa$), it follows that the $E_\kappa$-vector space $H^1_{\rm f}(\QQ,T_{\f_\kappa}^\dagger[1/p])$ is one-dimensional for all but finitely many $\kappa$ as above. We may discard this finite set if necessary, and assume without loss of generality that 
$${\rm dim}_{E_\kappa}\, H^1_{\rm f}(\QQ,T_{\f_\kappa}^\dagger[1/p])=1\,.$$
\subsubsection{}
It follows from the inflation-restriction sequence that we have a natural isomorphism
$$H^1(\QQ,T_{\f_\kappa}^\dagger[1/p])\xrightarrow{\,\,\sim\,\,}H(L_r,T_{\f_\kappa}[1/p])^{\psi_\kappa^{\frac{1}{2}}}\,.$$
Moreover, using \cite[Corollary B.5.3]{rubin00} with $V=F^+T_{\f_\kappa}^\dagger[1/p]$, this isomorphism restricts to an isomorphism
$$H^1_{\rm f}(\QQ,T_{\f_\kappa}^\dagger[1/p])\xrightarrow{\,\,\sim\,\,}H_{\rm f}(L_r,T_{\f_\kappa}[1/p])^{\psi_\kappa^{\frac{1}{2}}}\,.$$
In view of our discussion in \S\ref{subsubsec_2024_03_06_1703}, we may therefore assume without loss of generality that 
$${\rm dim}_{E_\kappa}\, H_{\rm f}(L_r,T_{\f_\kappa}[1/p])^{\psi_\kappa^{\frac{1}{2}}}=1\,.$$

\subsubsection{} We are now ready to state and prove the main result of this section:

\begin{theorem}
    \label{thm_main_sec_7_2024_03_06}
    For every pair $(\kappa,\lambda)$ as above, \eqref{prop_3rd_reduction_step} holds true under our running assumptions\footnote{Besides the standard hypotheses on the Hida families we work with, these include \eqref{item_deg6}, \eqref{item_GKS}, and \eqref{item_NA}.}.
\end{theorem}

\begin{proof}
    Let us combine the statement of Conjecture~\ref{conj_5_2_2024_02_02} (which we assume) and \eqref{eqn_2024_03_06_Howard_bis} to conclude that
    \begin{align}
        \label{eqn_thm_main_sec_7_2024_03_06_1}
        \begin{aligned}
\langle P_r^{\circ}(\kappa,\lambda), P_r^{\circ}(\kappa,\lambda) \rangle_{\mathrm{NT}} &=\frac{ p^{-r}\cdot C(\kappa,\lambda) \cdot \LL(f\otimes \mathrm{ad}^0(g),\psi_{\kappa}^{-\frac{1}{2}},1)}{\Omega_{f\otimes {\rm ad}(g)}}\\
&\qquad\times  4p^{1-2r}\, \langle f,  f\rangle\, \times\,\dfrac{\prod_{\ell | pN_\f}(1 + \ell)^{-1}}{\LL(f \otimes \epsilon_K, \psi_{\kappa}^{-\frac{1}{2}},  1)}\,\times\,\langle Q_\kappa, Q_\kappa \rangle_{\mathrm{NT}}\,.
\end{aligned}
    \end{align}
Observe that, by construction, both the Chow--Heegner point $P_r^{\circ}(\kappa,\lambda)$ and the twisted Heegner point $Q_\kappa$ belong to the 1-dimensional $E_\kappa$-vector space $H_{\rm f}(L_r,T_{\f_\kappa}[1/p])^{\psi_\kappa^{\frac{1}{2}}}$. As a result, \eqref{prop_3rd_reduction_step} is equivalent to checking that
 \begin{align}
        \label{eqn_thm_main_sec_7_2024_03_07_2}
        \begin{aligned}
\frac{C(\kappa,\lambda) \cdot \LL(f\otimes \mathrm{ad}^0(g),\psi_{\kappa}^{-\frac{1}{2}},1)}{\Omega_{f\otimes {\rm ad}(g)}}\times  4p^{1-3r}\, \langle f,  f\rangle\, \times\,\dfrac{ \prod_{\ell | pN_\f}(1 + \ell)^{-1}}{\LL(f \otimes \epsilon_K, \psi_{\kappa}^{-\frac{1}{2}},  1)}\\
=\lambda_{Np^{r}}(g)^4\, p^{2-4r}(p-1)^2\cdot \mathscr{D}(\kappa,\lambda)\cdot\mathfrak{g}(\psi_{\kappa}^{\frac{1}{2}})\cdot C_f^- \cdot \dfrac{\Lambda(f\otimes\Ad^0g,\psi_{\kappa}^{-\frac{1}{2}},1)}{a_p(f)^{-r}\,\Omega_{f}^- \, \Omega_{g}^{\rm ad}}\,.
    \end{aligned}
    \end{align}
    Reorganizing \eqref{eqn_thm_main_sec_7_2024_03_07_2}, it suffices to prove that  all 4 factors $\mathscr{A},\mathscr{B},\mathscr{L}_1$, and $\mathscr{L}_2$ that appear in the expression 
 \begin{align}
        \label{eqn_thm_main_sec_7_2024_03_07_3}
        \begin{aligned}C(\kappa,\lambda)&= \underbrace{p\left(\frac{p-1}{2}\right)^2\prod_{\ell | pN_\f}(1 + \ell)}_{\mathscr A}\,\cdot\,\underbrace{\lambda_{Np^{r}}(g)^4 \,\mathscr{D}(\kappa,\lambda)}_{\mathscr B} \\
        &\qquad\qquad \times \quad \underbrace{\mathfrak{g}(\psi_{\kappa}^{\frac{1}{2}})\,\dfrac{ C_f^-\, \LL(f \otimes \epsilon_K, \psi_{\kappa}^{-\frac{1}{2}},  1)}{a_p(f)^r\,\Omega_{f}^-}}_{\mathscr{L}_1}\quad \times \quad \underbrace{\dfrac{\Omega_{f\otimes {\rm ad}(g)}}{\,a_p(f)^{-2r}\langle f,  f\rangle\,p^{-r}\,\mathfrak{c}_\g(\lambda)^2\Omega_{g}^{\rm ad}}}_{\mathscr{L}_2}\, \times\, \mathfrak{c}_\g(\lambda)^2
        \end{aligned}
    \end{align}
    are algebraic, and that they $p$-adically interpolate as $\kappa$ and $\lambda$ (equivalently, $f$ and $g$) varies. Note that $\mathscr{A}\in \QQ$ is an absolute constant, whereas $\mathscr{B}$ and $\mathscr{L}_1$ readily have the required property (note that the latter is interpolated by the restriction $\cL_p^{\rm Kit}(\hf)(\kappa, \frac{{\rm wt}(\kappa)}{2}+1)$ of the Mazur--Kitagawa $p$-adic $L$-function to the central critical line; cf. \cite{BCS}, \S3.1). Finally, it follows from the definition of $\Omega_{f\otimes {\rm ad}(g)}$ (cf. Equation~\ref{eqn_2024_03_11_1724}) and $\Omega_{g}^{\rm ad}$ (cf. the statement of Lemma~\ref{lemma_first_reduction}) that
    $$\mathscr{L}_2=-2^{-6}p^{-r} \underbrace{a_p({g})^{2r}a_p(\overline{g})^{2r}}_{p^{2r}}\,\underbrace{\mathfrak{g}(\psi_\lambda)^{-1}\mathfrak{g}(\overline{\psi}_\lambda)^{-1}}_{p^{-r}}=-2^{-6}\,,$$
    where the second equality follows from the properties of the Gauss sum (note that $\psi_\lambda$ is necessarily an even character) and \cite[Theorem 4.16.17]{Miyake89}. As a result, $\mathscr{L}_2$ has the required property as well and our proof is complete.
\end{proof}

\newpage

\bibliographystyle{amsalpha}
\bibliography{references}

\end{document}